\documentclass[final,1p,times]{elsarticle}
\usepackage{amssymb}
\usepackage{algcompatible}
\usepackage{algorithm}
\usepackage{amsfonts}
\usepackage{amsmath}
\usepackage{amsthm}
\usepackage{amssymb}
\usepackage{array}
\usepackage{booktabs}
\usepackage{caption}
\usepackage{chapterbib}
\usepackage{chngcntr}
\usepackage{colortbl}
\usepackage{enumerate}
\usepackage{fancyhdr}
\usepackage[final]{pdfpages}
\usepackage{fullpage}
\usepackage{graphicx}
\usepackage{here}
\usepackage{hhline}
\usepackage{latexsym}
\usepackage{longtable}
\usepackage{lscape}
\usepackage{multicol}
\usepackage{multirow}
\usepackage{placeins}
\usepackage{prettyref}
\usepackage{rotating}
\usepackage{sectsty}
\usepackage{setspace}
\usepackage{subfig}
\usepackage{times}

 \newcommand{\ra}[1]{\renewcommand{\arraystretch}{#1}}
\usepackage{algorithm}
 \usepackage[]{algpseudocode}
\definecolor{Gray}{gray}{0.9}
\newcolumntype{L}[1]{>{\raggedright\let\newline\\\arraybackslash\hspace{0pt}}m{#1}}
\newcolumntype{C}[1]{>{\centering\let\newline\\\arraybackslash\hspace{0pt}}m{#1}}
\newcolumntype{R}[1]{>{\raggedleft\let\newline\\\arraybackslash\hspace{0pt}}m{#1}}

\usepackage{xfrac}

\usepackage{url}
\usepackage{breakurl}
\usepackage[breaklinks]{hyperref}

 \usepackage{tikz}
\usetikzlibrary{arrows,shapes,snakes,automata,backgrounds,petri,fit,positioning,matrix}
\newtheorem{theorem}{Theorem}

\newtheorem{lemma}{Lemma}
\newtheorem{definition}{Definition}
\newtheorem{claim}{Claim}

\makeatletter
\def\ps@pprintTitle{%
   \let\@oddhead\@empty
   \let\@evenhead\@empty
   \let\@evenfoot\@oddfoot
}

\begin{document}

\begin{frontmatter}

\title{A Monotone Approximate Dynamic Programming Approach for the Stochastic Scheduling,
Allocation, and Inventory Replenishment Problem: Applications to Drone and Electric Vehicle Battery Swap Stations}

%
\author[label1]{Amin Asadi}
\address[label1]{Department of Industrial Engineering, University of Arkansas\\ 4207 Bell Engineering, Fayetteville, AR 72701}
%
%
\ead{asadi@email.uark.edu}
%
\author[label1]{Sarah Nurre Pinkley\corref{cor1}}
\ead{snurre@uark.edu}

\begin{abstract}
There is a growing interest in using electric vehicles (EVs) and drones for many applications. However, battery-oriented issues, including range anxiety and battery degradation, impede adoption. Battery swap stations are one alternative to reduce these concerns \textcolor{black}{that} allow the swap of depleted for full batteries in minutes. We consider the problem of deriving actions at a battery swap station when explicitly considering the uncertain arrival of swap demand, battery degradation, and replacement. We model the operations at a battery swap station using a finite horizon Markov Decision Process model for the stochastic scheduling, allocation, and inventory replenishment problem (SAIRP), which determines when and how many batteries are charged, discharged, and replaced over time. We present theoretical proofs for the monotonicity of the value function and monotone structure of an optimal policy for special SAIRP cases. Due to the curses of dimensionality, we develop a new monotone approximate dynamic programming (ADP) method, which intelligently initializes a value function approximation using regression. In computational tests, we demonstrate the superior performance of the new regression-based monotone ADP method as compared to exact methods and other monotone ADP methods. Further, with the tests, we deduce policy insights for drone swap stations.

\end{abstract}

\begin{keyword}
Electric Vehicles and Drones; Battery Swap Station; Markov Decision Processes; Battery Degradation; Monotone Policy \textcolor{black}{and Value Function}; \textcolor{black}{Regression-based Initialization}; 
Approximate Dynamic Programming
\end{keyword}

\end{frontmatter}

\section{Introduction}\label{intro}

Electric vehicles (EVs) and drones hold great promise for revolutionizing transportation and supply chains. The  United States Department of Energy \cite{EVEverywhere} reports that EVs can reduce oil dependence and carbon emissions, but vehicle adoption is hindered by range anxiety, purchase price, recharge times, and battery degradation \citep{RJB15, Saxena2015}. Drone applications have increased in recent years; many organizations are using drones or undergoing testing to use drones for different purposes, including, but not limited to, delivery \citep{AmazonDrones, UPSDrones, DHL, RoyalMail}, transportation \citep{DubaiDrone}, and agriculture \citep{AgDrones2}. However, drone use is restricted by short flight times, long battery recharge times, and battery degradation \citep{Park17, DroneTime, ChargeTime}. An option to overcome these barriers for EVs and drones is a battery swap station. A battery swap station is a physical location that enables the automated or manual exchange of depleted batteries for full batteries in a matter of seconds to a few minutes.

Swap stations have many benefits including their ability to help reduce battery degradation. Battery degradation, or, more specifically, battery capacity degradation, is the act of the battery capacity decreasing over time with use. Each recharge and use of a battery causes a battery to degrade. Degraded battery capacity means EVs and drones have shorter maximum flight times and ranges. Thus, an interesting aspect of managing battery swap stations is that both battery charge and battery capacity are needed; however, the recharging and use of battery charge is the \emph{exact cause} of battery capacity degradation.  Thus, this presents a unique problem where \textcolor{black}{recharging batteries, which enables the system to operate in the short-term,} is harmful for long-term operation. \textcolor{black}{Although all recharging causes degradation, the regular-rate charging used at swap stations reduces the speed in which batteries degrade as compared to fast-charging \citep{Lacey13, Shirk15}.   This increases battery lifespans and causes less environmental waste from disposal.  In spite of this benefit, swap stations still must determine when to recharge and replace batteries.}

The benefits of battery swap stations are not restricted to decreasing battery degradation. Swap stations also alleviate range anxiety by allowing users to swap their batteries in a couple of minutes. Furthermore, battery swap stations are projected to be a key component within a smart grid through the use of battery-to-grid (B2G) and vehicle-to-grid (V2G). B2G and V2G enable a charged battery to \emph{discharge} the stored energy back to the power grid \citep{Dunn2011}. \textcolor{black}{In practice, several companies such as Toyota Tsusho and Chubu Electric Power in Japan \citep{Nuvve2, Nuvve}, NIO and State Grid Corporation in China \citep{Zhang2020}, and Nissan and E.ON in the UK \citep{UKv2g} have installed or plan to install V2G technology.} Swap stations can also reduce the purchase price barrier through a business plan where the swap station owns and leases the high-cost batteries \citep{mak2013infrastructure}.  For the many organizations seeking to use drones, a set of continuously operating drones is often vital. However, continuous operation is difficult because the realized flight time of a drone is often less than the recharge time \citep{DroneTime, ChargeTime}. \textcolor{black}{Thus, automated drone battery swap stations are a promising option because no downtime for recharging is necessary. 
Given the benefits and applications of swap stations, we examine the problem of optimally managing a battery swap station when considering battery degradation.}

\textcolor{black}{We model the operations at a battery swap station using the new class of stochastic scheduling, allocation, and inventory replenishment problems (SAIRPs). In SAIRPs, we decide on the number of batteries to recharge, discharge, and replace over time when faced with time-varying recharging prices, time-varying discharge revenue, uncertain non-stationary demand over time, and capacity\textcolor{black}{-}dependent swap revenue.  SAIRPs consider the \emph{key interaction} between battery charge and battery capacity and link the use and replenishment (recharging) actions of charge inventory with the degradation and replenishment needs of battery capacity inventory. Battery charge and capacity are linked because each recharge and discharge of a battery causes the battery capacity to degrade, and the level of battery capacity directly limits the maximum amount of battery charge. To replenish battery capacity, we must determine when and how many batteries to replace over time.  For SAIRPs, the combination of battery charge and battery capacity is necessary to satisfy non-stationary, stochastic demand over time.}

We \textcolor{black}{model the problem as a} finite horizon MDP model \textcolor{black}{allowing us to capture} the non-stationary elements of battery swap stations over time, including mean battery swap demand, recharging price, and discharging revenue \citep{Puterman05}. \textcolor{black}{The MDP's state space} is two-dimensional, indicating the total number of fully charged batteries and the average capacity of all batteries at the station. The action of the \textcolor{black}{model} is two-dimensional. The first dimension \textcolor{black}{indicates the} total number of batteries to recharge or discharge. The second dimension \textcolor{black}{indicates} the total number of batteries to replace. The \textcolor{black}{selected} action results in an immediate reward, \textcolor{black}{equal to profit}, comprised of capacity-dependent revenue from battery swaps, revenue from discharging batteries back to the power grid, cost from recharging batteries, and cost from replacing batteries. The \textcolor{black}{system transitions to a new state according to a discrete probability distribution representing battery swap demand over time, the current state, and the selected action. For our MDP model of the stochastic SAIRP, we seek to determine an optimal policy that maximizes the expected total reward, which is equal to the station's expected total profit.} \textcolor{black}{A standard solution method for solving MDPs is backward induction (BI) \citep{Puterman05}. We solve a set of modest-sized SAIRPs using BI to provide a baseline for comparing the approximate solution methods; however, as Asadi and Nurre Pinkley \cite{Asadi19} showed, BI is not effective for deriving optimal policies for realistic-sized SAIRPs.}

Stochastic SAIRPs suffer from the curses of dimensionality, thus, we investigate theoretical properties of the problem to inform more efficient solution methods.  \textcolor{black}{We prove} that the stochastic SAIRP has a monotone non-decreasing value function in the first, second, and both dimensions of the state. We also prove that general SAIRPs \textcolor{black}{violate the sufficient conditions for the existence of a monotone optimal policy in} the second dimension of the \textcolor{black}{state}. However, if the number of battery replacements in each decision epoch is constrained to be less than a constant upper bound, we prove there exists a monotone optimal policy for the second dimension of the \textcolor{black}{state} in the stochastic SAIRP.

To overcome the curses of dimensionality, we exploit these theoretical results and investigate efficient solution methods. 
We investigate methods \textcolor{black}{that} exploit our proven monotone structure, including Monotone Backward Induction (MBI) \citep{Puterman05} and monotone approximate dynamic programming (ADP) algorithms.  First, we examine Jiang and Powell's \cite{Jiang15} Monotone Approximate Dynamic Programming (MADP) algorithm which exploits the monotonicity of the value function.  \textcolor{black}{Next, we propose a new} regression-based Monotone ADP algorithm, which we denote MADP-RB.  In our MADP-RB, we build upon the foundation of MADP and introduce a regression-based approach to intelligently initialize the value function approximation.  

We design a comprehensive set of experiments using Latin hypercube sampling (LHS). We compare the performance of ADP methods with the BI and MBI for the LHS's generated scenarios of \textcolor{black}{a} modest size. Experimentally, we show our regression-based ADP generates near-optimal solutions for modest SAIRPs. Besides, using the same LHS scenarios, we solve large-scale SAIRPs with our proposed ADP algorithms. We demonstrate that our proposed ADP approaches can overcome the inherent curses of dimensionality of SAIRPs that BI, and MBI failed to succeed.

\emph{\textbf{Main Contributions.}} The main contributions of this work are as follows: 
(i) \textcolor{black}{we demonstrate that stochastic SAIRPs violate the sufficient conditions for the optimality of a monotone policy in the second dimension of the state} and prove the existence of a monotone optimal policy \textcolor{black}{for the second dimension of state} when an upper bound is placed on the number of batteries replaced in each decision epoch; 
(ii) we prove the monotone structure for the MDP value function;
(iii) we propose a regression-based monotone ADP method by utilizing the theoretical structure of the MDP optimal value function to intelligently approximate the initial value function and make updates \textcolor{black}{in} each iteration;
\textcolor{black}{(iv) we computationally demonstrate the superior performance of our regression-based monotone ADP algorithm and deduce managerial insights about managing battery swap stations.}

\indent  The remainder of the paper is organized as follows. In Section \ref{sec:LitReview}, we outline literature relevant to our modeling approach, solution approaches, and EV and drone applications.  In Section \ref{ProblemStatement}, we formally define the stochastic scheduling, allocation, and inventory replenishment problem as a two-dimensional Markov Decision Process. In Section \ref{TheoreticalResult}, we present theoretical results for the stochastic SAIRP.  In Section \ref{sec:Sol}, we \textcolor{black}{present solution methods and} outline the monotone ADP algorithm with regression-based initialization to solve stochastic SAIRP instances.
In Section \ref{ComputationalResults}, we present results and insights from computational tests of the solution methods and realistic instances of the stochastic SAIRP.  We summarize the contributions in Section \ref{Conclusion} and provide opportunities for future work.

\section{Literature Review} \label{sec:LitReview}

There is growing interest surrounding electric vehicles (EVs) and drones in industry and academia.  We proceed by discussing the relevant literature pertaining to (i) the EV and drone swap station application; (ii) the background knowledge for the proposed approach using aspects of \textcolor{black}{optimal timing and reliability,} inventory management\textcolor{black}{,} and equipment replacement problems; (iii) the scientific works that explain the lithium-ion battery degradation process; and (iv) ADP approaches that address the curses of dimensionality.  To the best of our knowledge, no past research has derived the structure of the optimal policy and value function for the scheduling, allocation, and inventory replenishment problem nor solved the realistic-sized instances of SAIRPs \textcolor{black}{to derive insights} for managing the operations at a battery swap station faced with battery degradation.

Swap stations were initially introduced for EVs and thus have a more extensive research base. However, there is growing interest surrounding drone battery swap stations. We first examine the work on managing the internal operations of a battery swap station that are most similar to the model presented in this paper and Asadi and Nurre Pinkley \cite{Asadi19}.  Widrick et al. \cite{Widrick16} develop an inventory control MDP for a swap station that only considers the number of batteries to recharge and discharge over time but excludes battery capacity levels, degradation, and replacement.  They prove the existence of a monotone optimal policy only when the demand is governed by a non-increasing discrete distribution (e.g., geometric). Nurre et al. \cite{Nurre14} also consider determining the optimal charging, discharging, and swapping at a swap station using a deterministic integer program \textcolor{black}{that} excludes uncertainty. Worley and Klabjan \cite{worley2011optimization} examine an EV swap station with uncertainty and seek to determine the number of batteries to purchase and recharge over time. Note, purchasing batteries is fundamentally different from battery replacement in SAIRPs. Worley and Klabjan \cite{worley2011optimization} examine the \textcolor{black}{one-time} purchase of batteries to open a swap station and do not consider purchasing decisions over time. Contrarily, we assume the initial number of batteries at the swap station is previously determined and instead consider \textcolor{black}{replacing} batteries over time.  Sun et al. \cite{Sun14} propose a constrained MDP model for determining an optimal charging policy at a single battery swap station and examine the tradeoffs between quality of service for customers and energy consumption costs.

Other research considers a mix of long-term strategic and short-term operational swap station decisions.  Schneider et al. \cite{Schneider18} consider a network of swap stations that seeks to determine the long-term number of charging bays and batteries to locate at each station and the short-term number of batteries to recharge over time. Schneider et al. \cite{Schneider18} do consider charging \emph{capacity}; however, their use of capacity indicates the number of batteries that can be recharged at one time in the station and do not model \emph{battery capacity}. Kang et al. \cite{Kang16} propose the EV charging management problem, which determines the optimal locations for a network of swap stations and further determines the charging policy for each location. Their definition of charging policy only considers charging and excludes discharging or replacement. 
Excluding the explicit charging actions over time, Zhang et al. \cite{Zhang14} determine the number of batteries that are necessary for swapping over time.  For further studies in the area of EV operations management, we refer the reader to a review by Shen et al. \cite{Shen19}.

A common limitation of the aforementioned research is that it fails to account for battery degradation. To the best of our knowledge, there are very few articles \textcolor{black}{that} consider battery degradation. Asadi and Nurre Pinkley \cite{Asadi19} are the first to introduce stochastic SAIRPs for managing battery swap stations with degradation. However, they do not theoretically analyze this problem class, \textcolor{black}{do not introduce intelligent approximate dynamic solution methods that exploit the theoretical results, and do not provide insights from solving realistic-sized SAIRPs.} 

Others have examined battery degradation in a deterministic setting without any uncertainty \citep{Kwizera2018, Park17, Tan18}.
Sarker et al. \cite{Sarker15} consider the problem of determining the next day operation plan for a battery swap station under uncertainty.  They do consider battery degradation; however, they solely penalize battery degradation with a cost in the objective and do not link it to a reduction in operational capabilities.

Others have examined battery swap stations from different perspectives.  Researchers have examined how to find the optimal number and location of swap stations in a system \citep{Shavarani2018, Kim13, HONG18}.  Extending this idea further, Yang and Sun \cite{Yang2015} look to locate swap stations and route vehicles through the swap stations.  Others have examined how to locate and/or operate swap stations \textcolor{black}{that} are coordinated with green power resources  \citep{Pan10}, stabilize uncertainties from wind power \citep{GZW12}, or coordinate with the power grid \citep{Dai2014}.

\textcolor{black}{ Our research is related to optimal timing and reliability problems. There is a rich literature on finding the optimal timing of decisions to maximize systems' lifespan and reliability. For instance, researchers maximize the expected quality-adjusted life years by finding the optimal timing of living-donor liver transplantation \citep{Maillart04}, biopsy test \citep{Chhatwal10, Zhang12}, and replacement of an Implantable Cardioverter Defibrillator generator \citep{Khojandi14}. There are two options for the actions in these works (e.g., transplant/wait, take/skip the biopsy test, replace/not replace). However, our action determines the \emph{number} of batteries to recharge/discharge and replace in each epoch because it is not a single battery that enables the station to operate. Instead, it a set of batteries that enables operation, which creates a significantly larger action space that is dependent on the number of batteries at the station. Similar to our work is that of \citep{bloch01}, which determines the optimal timing and duration of a degrading repairable system. There is extensive research in the nexus of optimization, reliability, and systems maintenance. We refer the reader to the recent review paper by \cite{JONGE18} for further study.} 

\textcolor{black}{Our research can be placed under the umbrella of inventory management and equipment replacement problems} with stochastic elements. There is a large research base examining these types of problems under different characteristics. We proceed by reviewing a small sample of this body of knowledge by focusing on foundational work and research most similar to the scope of this paper. 
Researchers have  extensively studied 
inventory problems, including those with stochastic demand \citep{StochasticInventory}, two- and multi-echelon supply chains \citep{Clark1960, Clark1972}, and multiple products \citep{DeCroix1998}. A desirable feature of the solutions to inventory problems is that the optimal policy has a simple structure. A classic example of such an optimal policy is the $(s,S)$ policy that indicates to order up to $S$ units when the inventory level drops below $s$ \citep{Scarf}. Others have examined more sophisticated inventory problems which include scheduling production \citep{vanderLaan1997,ElHafsi2009,Maity2011,Golari2017}, performing maintenance or replacement \citep{VanHorenbeek2013}, and ordering spare parts for maintenance \citep{Elwany2008, Rausch2010}. Additionally, researchers have examined perishable inventory that degrades over time \citep{Nahmias1982} or inventory that can be recycled or remanufactured in a closed-loop supply chain \citep{Toktay2000, Zhou2011, Govindan2015}. 

The proposed work is distinct from this previous literature as it links the actions of recharging batteries to the actions that must be taken for replacing battery capacity.  No prior work includes the counter-intuitive property that the act of maintaining the system in the short term (e.g., through recharging batteries which can be analogous to short-term maintenance or short-term inventory replenishment) is harmful for long\textcolor{black}{-}term performance (e.g., future need to replace equipment or replenish other types of inventory).

A novel component of our work is the consideration of battery degradation within the decision-making process. Battery degradation is most traditionally measured based on calendar life or cycles, where a cycle consists of one use and one recharge \citep{Lacey13, Ribbernick15}.  Using physical experiments, simulation, and mathematical modeling, researchers aim to capture the rate of battery degradation for different batteries and conditions such as temperature and depth of discharge \citep{Plett11, Abe12, Ribbernick15, Hussein15, Dubbary11}. We approximate battery degradation using a linear degradation factor derived from the work of \cite{Lacey13} and \cite{Ribbernick15}, as is consistent with other research using a linear forecast \citep{Xu2018, Abdollahi15, wood11}.

Our MDP model suffers from the curses of dimensionality due to \textcolor{black}{the} very large \textcolor{black}{size of all MDP elements together, including state and action spaces, transition probability, and reward}. Approximate dynamic programming (ADP) is a method that has had great success in determining near-optimal policies for large-scale MDPs \citep{Powell}. Researchers have used ADP methods to solve problems in energy, healthcare, transportation, resource allocation, and inventory management \citep{Bertmis02, Powell05, Simao09, Erdelyi10, Maxwell10, Cimen15, Meissnera18, Cimen17, Nasrollahzadeh18}. Jiang and Powell \cite{Jiang15} propose a monotone ADP algorithm that is specifically designed for problems with monotone value functions.  \textcolor{black}{In this paper, we prove that the value function of the stochastic SAIRP has a non-decreasing monotone structure. Hence, we utilize Jiang and Powell's \cite{Jiang15} monotone ADP algorithm and enhance it by adding a regression-based initialization.}

\section{Problem Statement}\label{ProblemStatement}

In this section, we present \textcolor{black}{and model} the Markov Decision Process (MDP) model of the scheduling, allocation, and inventory replenishment problem (SAIRP) \textcolor{black}{ that} considers stochastic demand for swaps over time, non-stationary costs for recharging depleted batteries, non-stationary revenue from discharging, and capacity-dependent swap revenue. The MDP model captures the dynamic average battery capacity over time, the associated replacement \textcolor{black}{policies}, and the interaction between battery charge and battery capacity at a battery swap station. We note, this model was originally presented in Asadi and Nurre Pinkley \cite{Asadi19}; however, we believe it is necessary to provide the reader with the formal problem definition to enable understanding of the main theoretical and algorithmic contributions that follow. \textcolor{black}{We use a finite horizon MDP to capture} the high variability of data over time, including the mean demand for battery swaps, the price for recharging batteries, and the revenue earned from discharging batteries back to the power grid. \textcolor{black}{The uncertainty in the system is the stochastic demand for battery swaps (i.e., exchange of a depleted battery for a fully-charged battery). We model this uncertainty (stochastic demand) using the random variable, $D_t$, for each time period $t$. These random variables are explicitly used to calculate the transition probabilities.} The objective is to maximize the expected \textcolor{black}{total reward} of the swap station and determine optimal policies which dictate how many batteries to recharge, discharge, and replace over time. \textcolor{black}{For our model, the expected total reward equals the expected total profit calculated as the revenue from satisfying demand and discharging batteries to the power grid minus the costs from recharging and replacing batteries.}

We formulate our MDP model with the following elements. We define $T$ as the finite set of decision epochs, which are the discrete periods in time in which decisions are made. By defining $N$ as the terminal epoch, $T=\{1,\ldots, N-1\}, \; N<\infty$.

We denote the two-dimensional state of the system at time $t$, $s_t=(s^1_t,s^2_t) \in S= (S^1 \times S^2)$, as the total number of fully charged batteries, $s^1_t \in S^1$, and the average capacity of all batteries, $s^2_t \in S^2$, at the swap station.  In the design of $s_t^1$, we only consider that batteries are either fully charged or depleted.  The number of full batteries at time $t$, $s^1_t$, is a\textcolor{black}{n integral} value between 0 and $M$, where $M$ is the total number of batteries in the station, thus, $S^1 = \{0, 1, 2, \ldots, M\}$.  

We \textcolor{black}{use an aggregated MDP in which we} track the discretized average battery capacity rather than \textcolor{black}{a disaggregated MDP, which tracks} each battery capacity individually\textcolor{black}{,} to reduce the curses of dimensionality from the second dimension of the state. \textcolor{black}{The disaggregated MDP severely suffers from the curse of dimensionality as the state space's size grows exponentially as the number of batteries increases.} We discretize the average battery capacity where $S^2=\{ 0, \theta, \theta+\varepsilon, \theta+2\varepsilon, \ldots, 1\}$, in which $\theta$ equals the lowest acceptable average battery capacity and $\varepsilon$ in the discretized capacity increment. State zero in $S^2$ is an absorbing state representing that the average battery capacity dropped below $\theta$.  To discourage the station from allowing the battery capacity to drop below $\theta$ thereby resulting in lower quality batteries at the station, we disallow charging, discharging, swapping, and replacement when in this absorbing state. \textcolor{black}{Hence, the set of feasible actions when in an absorbing state, $s^2_t < \theta$ or $s^2_t =0$, only includes no recharge/discharge and no replacement.}  \textcolor{black}{We note, with this aggregated modeling proposed by Asadi and Nurre Pinkley \cite{Asadi19}, the problem size and complexity are reduced, which is not always necessary when using approximate solution methods. However, the aggregated model allows us to benchmark the performance of new and existing approximate solution methods and analyze larger SAIRP instances. Further, we previously showed that the results do not significantly change with aggregation \citep{Asadi19}.}

We denote the two-dimensional action to represent the number of batteries to recharge/discharge, $a^1_t $, and the number of batteries to replace, $a^2_t$, at time $t$. \textcolor{black}{In our aggregated MDP model, there is no known difference between the capacity of batteries as we only track the average capacity of all batteries. In reality, swap stations, applying the aggregated MDP, may track/not track the capacity of each battery. If, consistent with the model, the swap station does not track individual battery capacity values, we assume the specific batteries that are selected to be recharged/replaced or discharged/swapped are \emph{arbitrarily} selected from the set of empty and fully-charged batteries, respectively. However, if the swap station does track the individual battery capacity value, we assume that the station selects to recharge/discharge and swap batteries with the highest capacity values and selects to replace batteries with the lowest capacity values. With this selection mechanism, individual battery capacity values will be closer to the average battery capacity of the system and, thus, further emphasizes the aggregated modeling decision.} Regarding the first dimension of the action, \textcolor{black}{$a^1_t$,} we attribute a positive value to the number of batteries to recharge, \textcolor{black}{$a_t^{1+}$}, and a negative value to the number of batteries to discharge,  \textcolor{black}{$a_t^{1-}$}. To clarify the distinction between recharging and discharging actions, we define positive recharging, $a_t^{1+}$, and discharging actions, $a_t^{1-}$, with Equations \eqref{eq:1} and \eqref{eq:2}. \textcolor{black}{We note that only dealing with the positive number of batteries that are recharged or discharged using Equations \eqref{eq:1} and \eqref{eq:2} is helpful to clarify the forthcoming state transition, probability transitions, and reward calculations.} 

\begin {equation}\label{eq:1}
a^{1+}_t=
\left\{
	\begin{array}{ll}
		a^{1}_t  & \text{ if} \ \ \ a^{1}_t \ge 0, \\
		0 & \text{ otherwise},
	\end{array}
\right.
\end {equation}

\begin {equation}\label{eq:2}
a^{1-}_t=
\left\{
	\begin{array}{ll}
		|a^{1}_t|  &   \text{ if} \ \ \ a^{1}_t < 0, \\
		0 & \text{ otherwise}.
	\end{array}
\right.
\end {equation}

\textcolor{black}{The action $a_t^1$ represents both the number of batteries that are recharged, when $a_t^1$ is positive, and the number of batteries that are discharged, when $a_t^1$ is negative.  We designed the action in this way as it is not beneficial to recharge and discharge at the same epoch, as they will cancel each other out and cause the capacity to degrade.  Thus, we select one value for $a_t^1$ for each time $t$, and state, $s_t$. Depending on whether the selected action is positive or negative indicates whether recharging or discharging will occur.} 
  We denote the number of plug-ins in the station as $\Phi$.  We assume all plug-ins are capable of supplying energy from the grid to recharge batteries and receiving energy from batteries discharged using Battery to Grid \citep{Dunn2011}.  \textcolor{black}{We define the first dimension of action as $a^1_t \in A^1_t=\{\max(-s^1_t, -\Phi), \ldots, 0, \ldots, \min(M-s^1_t-a^2_t, \Phi)\}$, which limits the number of discharged batteries by the minimum of the number of plug-ins and the number of full batteries ($-$min$(s^1_t, \Phi)$ = max$(-s^1_t, -\Phi)$) and limits the number of recharged batteries by the minimum of the number of plug-ins and the number of depleted batteries that were not replaced. }  In the second dimension of the action space, $a^2_t \in A^2_{t}=\{0,\ldots, M-s^1_t\}$, we only allow depleted batteries to be replaced \textcolor{black}{at} each epoch $t$ which arrive in epoch $t+1$ with full charge and capacity. We define $A_{s_t} = (A^1_{s_t} \times A^2_{s_t}) \subseteq (A^1_t \times A^2_t)$ as the set of feasible actions for the state $s_t$ at time $t$. \textcolor{black} {In our model, the set of feasible actions when in an absorbing state, $s^2_t < \theta$ or $s^2_t = 0$, only includes no recharge/discharge and no replacement; i.e.,  $A_{(s^1_t, 0)} = \{(0, 0)\}$.}

In Figure (\ref{fig:eventTiming}), we display the timing of the operations at the swap station including recharging, discharging, replacing, and swapping between epochs $t$ and $t+1$.  We assume that \textcolor{black}{the time between two consecutive epochs is sufficient} to recharge or discharge a battery completely. \textcolor{black}{In our model, we could preemptively recharge, discharge or replace batteries for future time periods.} Therefore, depleted (full) batteries 
selected for recharging (discharging) in epoch $t$ are fully charged (depleted) at the start of epoch $t+1$. 
When \textcolor{black}{stochastic} demand for a battery swap arrives in epoch $t$\textcolor{black}{,} we can swap up to the number of fully-charged batteries in our inventory which equals the number of fully-charged batteries at the start of $t$ minus the number of discharged batteries. \textcolor{black}{We subtract the fully-charged batteries assigned to be discharged as they are unavailable for swapping until the next decision epoch.}

\begin{figure}[!h]
\begin{center}
\begin{tikzpicture}[shorten >=1pt,draw=black!50, node distance=\layersep, scale=0.5]
\draw (0,0) -- (30,0);
\draw (1,3) -- (1,-3);
\node [below, label={[align=left, font=\footnotesize]  Start of\\   epoch $t$}] at (1,3) {};
\node [below, label={[align=center, font=\footnotesize]  \# Fully charged \\ and average capacity \\ $(s^1_t, s^2_t)$}] at (1,-6) {};
\draw[->] (4,0) -- (4,4);
\node [below, label={[align=center, font=\footnotesize]  Replace $a^2_t$ \\ batteries }] at (4,4) {};
\draw[->] (8,0) -- (8,4);
\node [below, label={[align=center, font=\footnotesize]  Recharge $a^{1+}_t$\\  batteries}] at (8,4) {};
\draw[->] (8,0) -- (8,-4);
\node [below, label={[align=center, font=\footnotesize]  Discharge $a^{1-}_t$\\  batteries}] at (8,-6) {};
\draw [decorate,decoration={brace,amplitude=10pt,mirror,raise=4pt},yshift=0pt] (8.5,0) -- (21.5,0) node [black,midway,xshift=0.00cm, yshift=-0.8cm] {};
\node [below, label={[align=center, font=\footnotesize, color=blue] Demand $D_t$  occurs\\ in epoch $t$, satisfied\\ if fully charged \\ batteries available}] at (15,-5.5) {};
\draw [line width=0.1cm, color=blue] (8.5,0) -- (21.5,0);
\draw[->] (22,4) -- (22,0);
\node [below, label={[align=center, font=\footnotesize]   Recharged $a^{1+}_t$ and  \\ discharged $a^{1-}_t$ \\ batteries complete }] at (22,4) {};
\draw[->] (22,-4) -- (22,0);
\node [below, label={[align=center, font=\footnotesize]  New replaced $a^{2}_t$\\ batteries arrive\\ }] at (22,-7.2) {};
\draw (28,3) -- (28,-3);
\node [below, label={[align=center, font=\footnotesize]  Start of\\  epoch $t+1$}] at (28,3) {};
\node [below, label={[align=center, font=\footnotesize] \# Fully charged and \\ average capacity \\ $(s^1_{t+1}, s^2_{t+1})$}] at (28,-6.5) {};
\end{tikzpicture}
\caption{Diagram outlining the timing of events for the SAIRP model.}
\label{fig:eventTiming}
\end{center}
\end{figure}
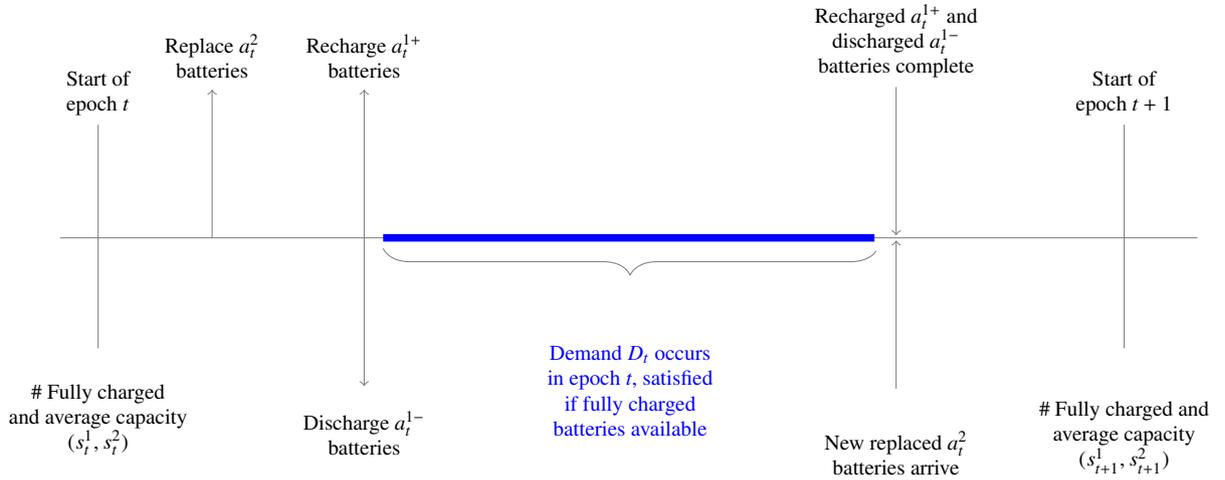

Transition probabilities indicate the likelihood of transitioning between states when considering the uncertainty of the system. In our MDP model, the uncertainty \textcolor{black}{in the system is the stochastic demand for battery swaps (i.e., exchange of a depleted battery for a fully-charged battery)} at each decision epoch $t$, $D_t$. The amount of satisfied swap demand in epoch $t$ equals $\min\{ D_t, s^1_t-a^{1-}_t\}$ wherein the second term indicates the number of full batteries that are not already discharging at $t$. We outline the state transition for the first dimension of the state in Equation \eqref{eq:3} which determines the number of full batteries in epoch $t+1$ based on the number of full, recharged, discharged, replaced, and swapped batteries in epoch $t$.
 
 \begin{equation}\label{eq:3}
s^1_{t+1}=s^1_{t} + a^2_{t}+a^{1+}_{t}- a^{1-}_{t}-\min\{ D_t, s^1_t-a^{1-}_t \}. 
 \end{equation}

The second state transitions according to Equation \eqref{eq:4}\textcolor{black}{,} which determines the future average capacity in $t+1$ based on the current average capacity and the number of full, recharged, discharged, and replaced batteries in epoch $t$.  We assume that all batteries swapped at time $t$ have a capacity equal to the average capacity of the batteries at the swap station.  We justify the assumption with the following logic.  Batteries previously swapped in epoch $t_1 < t$\textcolor{black}{,} which are in use outside of the station between $t_1$ and $t$ and need to be swapped again in epoch $t$\textcolor{black}{,} have a capacity similar to the average station capacity at $t$ when the swap station is used regularly (i.e., $t - t_1$ is small).  We define $\delta^C$ to represent the amount of battery capacity degradation from one battery cycle. We adopt the cycle-based degradation measure \citep{Abdollahi15, Lacey13} and assume that batteries do not degrade when not in use. Further, without loss of generality, we attribute the degradation from a full cycle to the recharge/discharge portion of the cycle. \textcolor{black}{We use $round()$ to represent that Equation \eqref{eq:4} returns values in the discretized state space, $S^2$, with $\varepsilon$ precision.}

 \vspace{-0.1in}
\begin{multline} 
g^2(s_t^1, s^2_t, a_t^1, a_t^2) = s^2_{t+1}=  round\bigg(\frac{(s^2_{t}- \delta^C) (a^{1+}_t +a^{1-}_t)  + a^2_t  + s^2_t(M - a^{1+}_t - a^{1-}_t - a^2_t)}{M} \bigg). \label{eq:4}
 \end{multline} 
 
\textcolor{black}{In the first term in the numerator of Equation (4), we multiply the summation of the number of recharged ($a^{1+}_t > 0 $) and discharged ($a^{1-}_t > 0$) batteries by the reduced average capacity $(s^2_{t}- \delta^C)$ due to the recharging/discharging actions.} The second term adds the $a^{2}_t$ replaced batteries with 100\% capacity. The third term maintains the same capacity for batteries not recharged, discharged, and replaced.  These terms are all averaged over the $M$ batteries in the swap station. The system enters the absorbing state $0 \in S^2$ when the average capacity is less than $\theta$.  To discourage entrance into this absorbing state, no recharging, discharging, swapping, or replacement is allowed. This setting ensures that swap stations should take appropriate actions \emph{before} allowing the average capacity to drop below $\theta$.  Thus, the transition of the second dimension of the state is precisely defined with Equation \eqref{eq:4a}.

\begin {equation}\label{eq:4a}
f^2(s_t^1, s^2_t, a_t^1, a_t^2) = s_{t+1}^2 =
\left\{
	\begin{array}{ll}
		g^2(s_t^1, s^2_t, a_t^1, a_t^2) &   \text{ if} \ \ \ g^2(s_t^1, s^2_t, a_t^1, a_t^2) \geq \theta, \\
		0 & \text{ otherwise.}
	\end{array}
\right.
\end {equation}

In Equation \eqref{eq:TrP}, we define the probability of transitioning from state $s_t = (s_t^1, s_t^2)$ in epoch $t$ to the state $j = (j^1, j^2)$ in epoch $t+1$ when action $a_t = (a_t^1, a_t^2)$ is taken.

\begin{footnotesize}
\begin{equation}\label{eq:TrP}
p(j^1, j^2 \mid s^1_t, s^2_t,  a^{1+}_t, a^{1-}_t, a^2_t)=
\begin{cases}
p_{s^1_t + a^2_t + a^{1+}_t- a^{1-}_{t}-j^1} & \mbox{if } a^2_t + a^{1+}_t < j^1 \leq s^1_t + a^2_t + a^{1+}_t -a^{1-}_t \mbox{ and } \\
&  j^2 = f^2(s_t^1, s^2_t, a_t^1, a_t^2),  \\
q_{s^1_t + a^2_t + a^{1+}_t- a^{1-}_{t}-j^1} & \mbox{if } j^1=a^2_t + a^{1+}_t \mbox{ and } j^2 = f^2(s_t^1, s_t^2, a_t^1, a_t^2),  \\
0 & \mbox{otherwise.} 
\end{cases}
\end{equation}
\end{footnotesize}

 We define $p_j = P(D_t = j)$ and $q_u = \sum_{j=u}^{\infty} p_j = P (D_t \ge u)$. Each probability in Equation \eqref{eq:TrP} depends on the number of batteries swapped, i.e., $s^1_t + a^2_t + a^{1+}_t- a^{1-}_{t}-j^1$ (see Equation \eqref{eq:3}).  When no batteries are swapped in epoch $t$, the station still has $s^1_t + a^2_t + a^{1+}_t- a^{1-}_{t}$ fully charged batteries \textcolor{black}{at} epoch $t+1$.  Instead, if all available fully-charged batteries in epoch $t$ are swapped, the station will \textcolor{black}{have} $a^2_t + a^{1+}_t$ fully charged batteries \textcolor{black}{at epoch $t+1$,} which are the result of the $a^2_t$ replaced and $a^{1+}_t$ recharged batteries in epoch $t$.

In Equation \eqref{eq:TrP}, the probability of transitioning to another state is non-zero only when Equation \eqref{eq:4a} is satisfied.  When the transition probability is $P (D_t = s^1_t + a^2_t + a^{1+}_t- a^{1-}_{t}-j^1)$, the demand for swaps is less than or equal to the number of full batteries available for swapping (as in condition 1 of Equation \eqref{eq:TrP}).  Alternatively, when the demand for swaps is greater than the number of available full batteries, the state transitions according to the cumulative probability $P (D_t \ge s^1_t + a^2_t + a^{1+}_t- a^{1-}_{t}-j^1)$.  If Equation \eqref{eq:4a} is not satisfied, $j^1$ is lower than the total number of batteries recharged and replaced ($a^2_t + a^{1+}_t$), or $j^1$ exceeds than the maximum number of fully charged batteries, $s^1_t + a^2_t + a^{1+}_t -a^{1-}_t$, the probability of transition is zero.

To clarify the transition probability function, we illustrate using an example. Consider the case when \textcolor{black}{at} epoch $t$, the swap station has 80 full-batteries and 20 depleted batteries in inventory (i.e., $M = 100, s^1_t = 80$), the average battery capacity equals 0.85, and we take the action to recharge 10 batteries (i.e., $a^1_t = a^{1+}_t = 10$) and replace 5 batteries (i.e., $a^2_t = 5$).  For this example, we assume recharging or discharging for one time period results in a capacity degradation equal to 0.01 (i.e., $\delta^C = 0.01$)  and the discretized capacity increment is also 0.01 (i.e., $\varepsilon = 0.01$). If there is no demand for battery swaps (i.e., $D_t =0$), \textcolor{black}{at epoch} $t+1$ the station will have 95 full batteries with a discretized average capacity equal to 0.86.  Thus, the probability of transitioning to a state with more than 95 full batteries or an average capacity not equal to 0.86 is zero.  Contrarily, if the demand for swaps is 80 or more (i.e., $D_t \geq 80$), then all full batteries in inventory will be swapped and the number of full batteries at \textcolor{black}{at epoch} $t+1$ equals $a^2_t + a^{1+}_t = 10+5 =15$. Thus, the probability of transitioning to a state with less than 15 full batteries is zero.  Further, the probability of transitioning to a state with exactly 15 full batteries and average capacity equal to 0.86 indicates that demand for swaps met or exceeded $s^1_t - a^{1-}_{t} = 80$.  Lastly, consider the case that we transition to a state with 30 full batteries and average capacity equal to 0.86.  The 30 full batteries is between the minimum, $a^2_t + a^{1+}_t = 15$, and maximum, $s^1_t + a^2_t + a^{1+}_t- a^{1-}_{t} = 95$ number of full batteries; thus, we know that $a^2_t + a^{1+}_t = 15$ batteries arrive at the end of $t$ indicating that we swapped  $s^1_t + a^2_t + a^{1+}_t- a^{1-}_{t} - j^1 = 80 + 5 + 10 - 0 - 30 = 65$ batteries \textcolor{black}{at} epoch $t$.  As follows, the probability of transitioning to this state equals the probably that demand for swaps equals 65, i.e., $P(D_t = 65)$.

The actions taken seek to maximize the expected total reward.  The expected total reward depends on the immediate reward earned \textcolor{black}{at} each epoch.  Specifically, the immediate reward is the profit earned. \textcolor{black}{In our setting, swap stations earn revenue from swapping and/or discharging fully-charged batteries and incur costs to recharge and/or replace depleted batteries.} We calculate the immediate reward \textcolor{black}{at} epoch $t$ according to the state of the system $s_t = (s^1_t, s^2_t)$, the taken action $a_t= (a^1_t, a^2_t)$, and the future state $s_{t+1} = (s^1_{t+1}, s^2_{t+1})$. Specifically, the immediate reward is calculated according to Equation \eqref{eq:imRew},

\begin{equation}\label{eq:imRew}
r_t (s_t, a_t, s_{t+1})= \rho_{s_t^2} (s^1_{t}+a^2_{t}+ a^{1+}_{t} - a^{1-}_{t}-s^1_{t+1}) - K_t a^{1+}_t + J_t a^{1-}_t - L_t a^2_t,
\end{equation}

\noindent where $s^1_{t} + a^2_{t}+a^{1+}_{t}- a^{1-}_{t}-s^1_{t+1}$ equals the number of batteries swapped and the time-dependent recharging cost, discharging revenue, and replacement cost are defined as  $K_t$, $J_t$, and $L_t$, respectively.  \textcolor{black}{We note that SAIRPs consider two aspects of a battery, charge and capacity. In this model, the fully-charged/empty batteries are not necessarily full-capacity as they might already be degraded due to the previous recharge/discharge actions. Thus, the average capacity of batteries can take a value less than 100\%.} We assume the realized swap revenue depends on the current average capacity.  Thus, we define $\rho_{s_t^2}$ to be the capacity-dependent revenue per battery swapped in Equation \eqref{eq:rho}.

\begin{equation} \label{eq:rho}
\rho_{s_t^2} = \beta \left(1+\frac{s^2_t - \theta}{1-\theta}\right) =\frac{\beta(1+s^2_t - 2 \theta)}{1-\theta}.
\end{equation}

We set $\beta \ge \text{max}_{t\in T} \ J_t$ to ensure the swap station is profitable with each battery swapped (i.e., the swap revenue is no less than the maximum recharging cost). 
\textcolor{black}{We use the average capacity of batteries as the indicator of the quality of batteries in the station when developing the revenue per swap function. Revenue per battery swap is a linear function of the average capacity of batteries in the station. This setting ensures that the stations can gain higher revenue when the average capacity is higher. It also provides an incentive for swap stations to replace batteries for higher revenue and benefits customers by receiving higher quality batteries.} In the design of Equation \eqref{eq:rho}, when the average capacity is at the lowest operational value ($s_t^2 = \theta$), the revenue per swap $\rho_{s_t^2}$ equals $\beta$, which is at least equal to the maximum price paid for recharging batteries.  When the swap station has an average battery capacity equal to 1, $s^2_t = 1$, then $\rho_{s_t^2} = 2 \beta$ which equates to a higher revenue earned due to higher customer satisfaction from swapping a higher quality battery.  Hence, in our design, the revenue per swap has a value between $[\beta,2\beta]$ depending on the average capacity of batteries in the station at time $t$.  We calculate the terminal reward in Equation \eqref{eq:TerminalRew} as the potential revenue from swapping all remaining fully charged batteries provided the average battery capacity is at least $\theta$.

\begin{equation}
r_N (s_N)=
\left\{
	\begin{array}{ll}
		\rho_{s_N^2} s^1_N  & \text{ if} \ \ \ s^2_N \geq \theta, \\
		0 & \text{ otherwise}.
	\end{array}
\right.
\label{eq:TerminalRew}
\end{equation}

Using the probability transition function and the immediate reward, we define the immediate expected reward in Equation \eqref{eq:8a}.
\begin{equation}\label{eq:8a}
r_t (s_t, a_t)= \sum_{s_{t+1} \in S}  \left[ p_t(s_{t+1} \mid s_t, a_t)(\rho_{s_t^2} (s^1_{t}+a^2_{t}+ a^{1+}_{t} - a^{1-}_{t}-s^1_{t+1})) \right] - K_t a^{1+}_t + J_t a^{1-}_t - L_t a^2_t.
\end{equation}

We define the decision rules, $d_t(s_t): s_t \rightarrow A_{s_t}$, as a function of the current state and time. Our decision rules determine the selected action $a_t \in A_{s_t}$ when the system is in $s_t$ at decision epoch $t \in T$. In our problem setting, we use deterministic Markovian decision rules because we choose which action to take provided we know the current state \citep{Puterman05}. A policy $\pi$ consists of a sequence of decision rules ($d^\pi_1(s_1), d^\pi_2(s_2), \ldots, d^\pi_{N-1}(s_{N-1}$) for all decision epochs.  The expected total reward of policy $\pi$, denoted $\upsilon_N^\pi(s_1)$ when the system starts in state $s_1$ \textcolor{black}{at time $t$=1} is calculated according to Equation \eqref{eq:upsilon}.

\begin{equation}
\upsilon_N^{\pi}(s_1) = \mathbb{E}_{s_1}^{\pi}\left[\sum_{t=1}^{N-1}r_t(s_t,a_t)+r_N(s_N)\right].
\label{eq:upsilon}
\end{equation}

\textcolor{black}{In Section \ref{sec:Sol}, we describe our solution methodology to find optimal/near-optimal solutions to maximize the expected total reward of the stochastic SAIRPs.} 

\section{Theoretical Results}\label{TheoreticalResult}

In this section, we prove theoretical properties regarding the structure of the optimal SAIRP policy and value function.  First, we \textcolor{black}{ show that the stochastic SAIRPs violate the sufficient conditions for the optimality of a monotone policy in the second dimension of the state.}  Second, we prove the existence of a monotone optimal policy for the second dimension of the \textcolor{black}{state} in a special case of the SAIRP.  Lastly, we prove the monotonicity of the value function when considering the first, second, and both dimensions of the \textcolor{black}{state}.  In the remainder of this section, we present the main theorems and point the reader to the appendices for the formal mathematical proofs.


\subsection{Monotone Policy}\label{MonotonePolicy}

Our investigation in proving the structure of an optimal policy for the SAIRP is motivated by the desire to exploit efficient algorithms \textcolor{black}{that} require less computational effort to find optimal policies and increase the ability to solve larger problem instances \citep{Puterman05}.  Widrick et al. \cite{Widrick16} examined the problem of managing a battery swap station when only considering battery charge, or equivalently the first dimension of our MDP model.  They proved the existence of a monotone optimal policy when demand is governed by a non-increasing discrete distribution.  In our investigation of the second dimension, our intuition was that monotonicity would be preserved for the  stochastic SAIRP.  Informally, this equates to the optimal policy indicating to replace more batteries when the average capacity is lower.  However, in Lemma~\ref{Thm1}, we prove a counter-intuitive result that, in general, \textcolor{black}{the sufficient conditions for the optimality of a monotone policy in the second dimension of the state do not exist for the stochastic SAIRPs.}  Instead, we are able to prove the existence of a monotone optimal policy for the second dimension of the \textcolor{black}{state} when an upper bound is placed on the number of batteries replaced at the swap station in each decision epoch.

First, we formally define a monotone optimal policy.  A non-increasing monotone policy $\pi$ has the property that for any $s_i, s_j \in S$ with $s_i \le s_j$ \textcolor{black}{(for multi-dimensional states, please see the partial ordering definition in Definition \ref{Def2} of Appendix \ref{APP2})}, there exist decision rules $ d^{\pi}_t(s_i) \ge d^{\pi}_t(s_j)$  for each $t = 1, \dots, N-1$ \citep{Puterman05}. \textcolor{black}{The sufficient conditions for the existence of a monotone optimal policy in the second dimension of the state are as follows  \citep{Puterman05}.}

\begin{enumerate}
\item $r_t (s^2_t, a^2_t )$ is non-decreasing in $s^2_t$ for all $a^2_t \in A'$. 
\item $q_t(k \mid s^2_t , a^2_t)$ is non-decreasing in $s^2_t$ for all $k \in S^2$ and $a \in A'$. 
\item $r_t (s^2_t, a^2_t )$ is a subadditive function on $S^2 \times A'$. 
\item $q_t(k \mid s^2_t , a^2_t)$  is subadditive on $S^2 \times A'$ for every $k \in S^2$. 
\item $r_N (s_N)$ is non-decreasing in $s^2_N$.
\end{enumerate}

\noindent Where $A'$ includes all possible actions for the second dimension of the action space. Specifically, \mbox{$A' = \{\cup_{s_t \in S} A^2_{s_t}\}$}. \textcolor{black}{We note that $q_t(k \mid s, a) = \sum_{j=k}^{\infty} p_t(j \mid s, a)$, which is the sum of the probabilities from $k$ to $\infty$, in general. For the second dimension, we have $q_t(k \mid s^2_t, a^2_t) = \sum_{j^2=k}^{\infty} p_t(j^2 \mid s^2_t, a^2_t)$.}

\textcolor{black}{In Lemma~\ref{Thm1}, w}e prove that one of the aforementioned conditions is not satisfied for stochastic SAIRPs.  
\textcolor{black}{In Theorem~\ref{Thm2}, w}e are able to prove that a monotone optimal policy in the second dimension of the \textcolor{black}{state} does exist when there is an upper bound on the number of batteries replaced in each decision epoch.
We refer the reader to Appendix \ref{APP1} for full details of the proof of Lemma~\ref{Thm1} and Theorem~\ref{Thm2}.

\begin{lemma}\label{Thm1} \textcolor{black}{The stochastic SAIRPs violate the sufficient conditions for the optimality of a monotone policy in the second dimension of the state. }
\end{lemma}

\setcounter{theorem}{0}
\begin{theorem}\label{Thm2} There exist optimal decision rules $d^*_t : S \rightarrow A_{s_t}$ for the stochastic SAIRP which are monotone non-increasing in the second dimension of the \textcolor{black}{state} for \mbox{$t= 1, \dots, N-1$} if there is an upper-bound $U$ on the number of batteries replaced \textcolor{black}{at} each decision epoch where $U=\frac{M\varepsilon}{2(1-s^2_t)}$, when $M$ is the number of batteries at the swap station and $\varepsilon$ is the discretized increment in capacity.      
\end{theorem} 

\textcolor{black} {We provide an example to explain the optimality of the monotone policy in the second dimension of the state. Consider a swap station with $M=100$ batteries, a discretized capacity increment $\varepsilon= 0.01$, and a replacement threshold $\theta = 0.8$. The monotone policy is optimal when the maximum number of batteries replaced per epoch, $U$, is between 2 and 50.  The specific value between 2 and 50 depends on the value of the average capacity.  If the average capacity $s^2_t = 0.8$, then $U=\frac{100(0.01)}{2(1-0.8)} = 2$ whereas if $s^2_t = 0.99$, then $U=\frac{100(0.01)}{2(1-0.99)} = 50$. We note that when $s^2_t$ = 1, then $U=\frac{100(0.01)}{2(1-1)} = \infty$ meaning there is no limit on the number of replaced batteries.  However, as the average capacity is already at the highest value of 1, it is not advantageous for swap stations to incur the cost for replacing a full capacity battery after which the average capacity will remain at 1.}  Although there is a restriction on the number replaced in each epoch, there are no restrictions on the consistent replacement of batteries over multiple consecutive decision epochs.

\subsection{Monotone Value Function}\label{MonotoneValuefunction}

We now investigate the structure of the value function for stochastic SAIRPs. Although proving that a value function has a monotone structure is a weaker result than proving the structure of an optimal policy, it enables the application of computationally efficient solution methods. We prove that the MDP value function for the stochastic SAIRP is monotone non-decreasing in the first, second, and both dimensions. These results directly motivate our selection of efficient approximate dynamic programming algorithms.

A value function $V(s)$ is monotone non-decreasing in state $s$, if for any $s_i, s_j \in S$ with $s_i \le s_j$ we have $V(s_i) \le V(s_j)$ for any given action in any decision epoch $t$  \citep{Papadaki07, Jiang15}.  The MDP value function for the stochastic SAIRP is given in Equation \eqref{eq:u} which is comprised of the immediate expected reward (as given by Equation \eqref{eq:imRew}) and the transition probabilities (as given by Equation \eqref{eq:TrP}).  In Theorem~\ref{Thm3}, we show that the value function is monotone in $s^2_t$. This means for any given action, in each decision epoch $t$, as the average capacity increases, the MDP value function will not decrease.

\begin{theorem}\label{Thm3} The MDP value function of the stochastic SAIRP is monotonically non-decreasing in $s^2_t$.      
\end{theorem} 

 In Theorem~\ref{Thm4}, we prove that the value function is monotone in $s^1_t$. This means for any given action in each decision epoch $t$, as the number of fully-charged batteries increases, the MDP value function will not decrease. If the demand for the MDP model is governed by a non-increasing discrete distribution, this result is implied from the result of Widrick et al. \cite{Widrick16}.  However, we strengthen the result as we \emph{do not} require a non-increasing discrete distribution in Theorem~\ref{Thm4}.

 \begin{theorem}\label{Thm4} The MDP value function of the stochastic SAIRP is monotonically non-decreasing in $s^1_t$.      
\end{theorem} 

 When considering both dimensions simultaneously, in Theorem~\ref{Thm5}, we prove that the value function is monotone in $(s^1_t, s^2_t)$.   

\begin{theorem}\label{Thm5} The MDP value function of the stochastic SAIRP is monotonically non-decreasing in $(s^1_t, s^2_t)$.      
\end{theorem}

\textcolor{black}{In a multi-dimensional setting, we need to define the concepts of partial ordering and partially non-decreasing function, which are given by Definitions \ref{Def2} and \ref{Def3} in Appendix \ref{APP2}.} 
We \textcolor{black}{also} refer the reader to Appendix \ref{APP2} for full details of the proofs of Theorems~\ref{Thm3},  \ref{Thm4}, and \ref{Thm5}.

\section{Solution Methodology}\label{sec:Sol} 
\textcolor{black}{This section presents the solution methods used to solve the stochastic Scheduling, Allocation, and Inventory Replenishment Problem (SAIRP). First, we briefly describe the dynamic programming solution methods with the backward induction (BI) approach to provide exact solutions when the problem is not large-scale. Next, we present the approximate dynamic programming methods to overcome the curses of dimensionality and yield high-quality solutions for the stochastic SAIRPs.}     

\subsection{Exact Solution Method: Dynamic Programming} \label{sec:DP}
\textcolor{black}{Backward induction (BI) is an exact solution method to find optimal policies for the Markov Decision Process (MDP) problems \citep{Puterman05}.}  Our goal is to find the optimal policy $\pi^*$ that maximizes the expected total reward \textcolor{black}{given by Equation \eqref{eq:upsilon}}. We attribute the optimal value function, $V_t^*(s_t)$, to the optimal policy. We calculate the optimal value based on cumulative values of taking the best actions onward from decision epoch $t$ to $N$ when in state $s_t$ at time $t$ (see Equation \eqref{eq:u}). We use Bellman equations as presented in Equation \eqref{eq:u} to find optimal policies and corresponding optimal value functions for $t=1,\ldots,N-1$ and $s_t\in S$.
\begin{equation}
V_t(s_t)=\max_{a_t\in A_{s_t}}\left\{r_t(s_t,a_t)+\sum_{j\in S}p_t(j \mid s_t, a_t)u_{t+1}(j)\right\}.  
\label{eq:u}
\end{equation}

\textcolor{black}{The BI algorithm starts from $t=N$ and sets $V_N(s_N) = r_N(S_N)$ according to Equation \eqref{eq:TerminalRew}. Then, it finds the actions that maximize $V_t (s_t)$ for every state $s_t$ moving backward in time ($ t = N-1, \dots, 1$) using Equation \eqref{eq:Arg}. The optimal expected total reward over the time horizon is $V_1^*(s_1)$ where $s_1$ is the state of the system at the first decision epoch.    
\begin{equation}
a^*_{s_t,t} = \text{arg max}_{a_t\in A_{s_t}}\left\{r_t(s_t,a_t)+\sum_{j\in S}p_t(j|s_t, a_t)u_{t+1}(j)\right\}.
\label{eq:Arg}
\end{equation}}

\textcolor{black}{Now, we clarify terms used in the remainder of this paper. A \emph{sample path of demand} is the collection of a realized demand (uncertainty element) per time period generated from a given probability distribution.  A \emph{sample path of state} is comprised of the collection of consecutive visited states, one per time period.  To calculate the visited states, we need the decision rule returned by a solution method for all states or visited states over time, the \emph{sample path of demand}, and the present state.  Using this information, we use the state transition functions (Equations \eqref{eq:3} and \eqref{eq:4a}) to calculate the \emph{sample path of state}.  A \emph{sample path of policy} is the set of consecutive decision rules of the visited states of the system.  We use the term \emph{instance} to refer to an example of stochastic SAIRP, specifically when we discuss the size of stochastic SAIRPs. The term \emph{scenario} is used to refer to examples within our space-filling designed experiment that include different values for parameters. These values are generated such that to cover the designed experiment space (see Section \ref{LHS}).  
}
 
\subsection{Approximate Dynamic Programming Solution Methods} \label{sec:ADP}

In this section, we outline our monotone approximate dynamic programming algorithm with regression-based initialization (MADP-RB).  Approximate dynamic programming is a proven solution method that overcomes the curses of dimensionality \citep{Powell}.  Using the foundation of the monotone approximate dynamic programming algorithm proposed by Jiang and Powell \cite{Jiang15}, we make enhancements by exploiting our theoretical results to intelligently approximate the initial value function approximation and update the approximation with each algorithmic iteration.

The stochastic SAIRP suffers from the curses of dimensionality \textcolor{black}{considering the size of all MDP elements together.} Asadi and Nurre Pinkley \cite{Asadi19} showed \textcolor{black}{that} the size of the \textcolor{black}{the state space, the action space, the transition probability function, and the optimal policy are $O(\frac{M(1- \theta)}{\varepsilon})$,  $O(M^2)$, $O(M^4N\frac{1- \theta}{\varepsilon})$, and $O(MN\frac{1- \theta}{\varepsilon})$, respectively,} where $M$, $N$, $\varepsilon$, and $\theta$ are the number of batteries, the time horizon, the capacity increment, and the replacement threshold, respectively.  \textcolor{black}{For instance, the size of the transition probability function is $O(10^{15})$ for a realistic-sized problem with $M=100$ batteries, planning over a one month time horizon in one hour increments $N = 744$ with discretized battery capacity in increments of $\varepsilon = 0.001$.} Due to these large sizes, standard MDP solution methods, such as backward induction (BI), were ineffective in solving realistic-sized instances of the stochastic SAIRP \citep{Asadi19}.   Although there are many different ADP algorithms and approaches \citep{Powell}, there is no standard method to link the best algorithm to solve any particular problem.  However, using the problem structure is good practice when developing efficient and effective algorithms.  As we proved in Theorems~\ref{Thm3}, \ref{Thm4}, and \ref{Thm5}, the value function is monotonically non-decreasing in both dimensions, $s^1_t$ and $s^2_t$. Hence, it is reasonable to utilize and enhance the monotone approximate dynamic programming algorithm proposed by Jiang and Powell \cite{Jiang15}.  This algorithm has already shown promising performance for several application areas \citep{Jiang15}.   We proceed by outlining the core steps of \citep{Jiang15}'s monotone approximate dynamic programming (MADP) algorithm while highlighting our additions and changes to create the monotone approximate dynamic programming algorithm with regression-based initialization (MADP-RB).  To aid with the explanation, in Algorithm \eqref{alg:Amin}, we outline the MADP and underline the enhancements for our MADP-RB.  First, we introduce the notation necessary for the ADP algorithms in Table (\ref{tab:Notation}).

\begin{table}[h]
\ra{1.5}
\caption{Notation used in the ADP algorithms}
\footnotesize
\centering
\scalebox{1}{
\begin{tabular}{@{}ll@{}}
Notation &					Description \\ \hline
$maxIteration \textcolor{black}{+ 1}$ & 			The maximum number of regression-based initialization iterations\\  
$\overline{M}$ & 			The starting number of batteries used for the small SAIRPs solved using BI \\
$\overline{T}$ & 			The time horizon in the small SAIRP \\
$u^{iter}_t(s_t)$ & 			The optimal value of being in state $s_t$ at iteration $iter$ and time $t$\\
$\tau$ & 					The maximum number of core ADP iterations\\
${V}^n_t (s_t)$ & 			The optimal value of being in state $s_t$ at time $t$ for iteration $n$ \\
$\overline{V}^n_t (s_t)$ & 		The approximate value of being in state $s_t$ at time $t$ for iteration $n$ \\
$\hat{\upsilon}_t^n(s_t^n)$ & 	The observed value of state $s^n_t$ at time $t$ for iteration $n$ \\
$z_t^n(s_t^n)$ & 			The smoothed value of being in state $s_t$ at time $t$ for iteration $n$ \\ \hline
\end{tabular}}
\label{tab:Notation}
\end{table}

\subsubsection{Monotone ADP with Regression-Based Initialization}

In this section, we describe the core steps of the MADP algorithm proposed by Jiang and Powell \cite{Jiang15} and our enhancement using regression-based initialization in Algorithm \eqref{alg:Amin}. 
We display our enhancements in Algorithm \eqref{alg:Amin} with underlines to make it more clear for the reader. We proceed by explaining the implementation of the algorithm.

\textcolor{black}{The Monotone ADP with Regression-Based Initialization (MADP-RB) has two main steps. In the first step, we intelligently initialize the value function approximation using a linear regression function. The coefficients of the regression function are derived from feeding the optimal solutions of small SAIRPs. The second step is the core MADP algorithm that consists of updating the approximated values of visited and non-visited states over time through an iterative process. The states are visited over time at each iteration using the information of the present state, realized uncertainty, and taken action. The approximated value of the visited state is updated based on the observed value and previous approximated value. At each iteration, the monotonicity operator updates the approximated value of the non-visited states over time. We proceed by explaining each step in detail.}

The first step of the MADP algorithm is to initialize the value function approximation for all decision epochs such that the monotonicity of the value function is preserved.  Commonly, this is done by assigning a constant value, e.g., 0, to $\overline{V}^0_t\textcolor{black}{(s_t)}$ for all $s_t \in S$ and $t =1, \ldots N-1$.  However, using 0 or any constant value fails to exploit how the monotone value function changes based on state and time.  Thus, our enhancement to the MADP algorithm is to intelligently approximate the initial value function approximation by exploiting the monotonicity of the value function.  To do so, we iteratively calculate the optimal value function for small but increasing problem instances.  Then, we use linear regression to approximate the initial value function for larger problem instances.  To further explain this enhancement, we outline how these steps can be applied to the stochastic SAIRP.

In the stochastic SAIRP, as $M$ and $T$ increase, the problem suffers from the curses of dimensionality. \textcolor{black}{For instance, the transition probability is $O(10^{15})$ for a realistic-sized problem when $M=100$, $N = 744$, and $\varepsilon = 0.001$}.  Thus, we first optimally solve small instances of the stochastic SAIRP with small numbers of batteries $\overline{M} \ll M$ and time periods $\overline{T} \ll T$.  We repeat this step by slowly increasing the number of batteries by one until we reach a user defined maximum number of iterations (or until it is computationally infeasible to optimally solve small instances with BI). For each instance, we determine the optimal value function and the associated trends based on changes in the state, time, and number of batteries.  From the value functions for smaller instances of the problem, we use linear regression on \textcolor{black}{the decision epoch ($t$), the state of the system ($s_t$), and the number of batteries in the station ($M$)} to approximate the initial value function approximation for larger problem instances.   See lines 1-6 in Algorithm \eqref{alg:Amin} \textcolor{black}{that} describe this regression-based enhancement which are new and distinct from Jiang and Powell \cite{Jiang15}.  We complete the initialization phase in line 7, where we set the approximate value for the terminal epoch for all states and all core iterations $n=1, \ldots, \tau$ to the terminal reward (see Equation \eqref{eq:TerminalRew}).

\begin{algorithm}[ht]
\caption{Monotone Approximate Dynamic Program with Regression-based Initilization}
\label{alg:Amin}
\begin{algorithmic}[1]
  \footnotesize
\State Initialize $\overline{M}$ batteries and $\overline{T}$ time periods, where $\overline{M} \ll M$ and $\overline{T} \ll T$.
\State \underline{Set iteration counter \emph{iter} = 0.}
\For {\textcolor{black}{\underline{\emph{iter} $\le$ \emph{MaxIterations}}}}
	\State \underline{$u^{iter}_t(s_t) \leftarrow$ Use backward induction to find the value function for the problem with $\overline{M} + iter$} \linebreak \indent \underline{batteries and $\overline{T}$ time periods, 
	where $\overline{M} \ll M$ and $\overline{T} \ll T$.}
\EndFor
\State \underline{Initialize $\overline{V}^0_t\textcolor{black}{(s_t)}$ using linear regression on the combined $u^{iter}_t(s_t)$ for \textcolor{black}{all states} at $t =1, \ldots, N-1$.}
\State Set $\overline{V}^n_N(s) = r_N(s)$ for $s \in S$ and $n = 1, \ldots, \tau$  \State Set $n = 1$
\WHILE{$n \leq \tau$}
\State Select initial state $S_1^n$
	\For{$t = 1, \ldots, N-1$}
		\State Sample an observation of the uncertainty, $D_t$, determine optimal action $a_t^n$ and future value $\hat{\upsilon}_t^n(s_t^n)$.
		\State Smooth the new observation with the previous value, \begin{equation} z_t^n(s_t^n) = ({1 - \alpha_n})\overline{V}^{n-1}_t(s_t^n) + {\alpha_n}\hat{\upsilon}_t^n(s_t^n) \label{eq:Smoothing} \end{equation}
		\State Perform value function monotonicity projection operator as in Jiang and Powell (2015)
		\State Determine next state, $S_{t+1}^n$
	\EndFor
	\State Increment $n = n+1$
\ENDWHILE
\end{algorithmic}
\end{algorithm}

The algorithm iteratively proceeds in accordance with \citep{Jiang15} in lines 8-12.  With each iteration, we select the best action starting from a random initial state (line 10) and move forward in time (line 11) using a sample observation of uncertainty and the approximate value of the future state, $\overline{V}^{n-1}_{t+1} (j)$ (line 12).  Specifically, we pick the best action (line 12) using Equation \eqref{eq:Arg3} and store the current observed value using Equation \eqref{eq:uOberveration}. \textcolor{black}{In Equation \eqref{eq:uOberveration}, we use  the previous approximation of the future states, $\overline{V}^{n-1}_{t+1} (s_{t+1})$, as the approximation of $\mathbb{E}({V}_{t+1} \mid s_t, a_t$). }

\begin{equation}
 a^n_{s_t,t} =
\underset{a_t\in A_{s_t}}{\arg\max}\left\{r_t(s_t,a_t) + \overline{V}^{n-1}_{t+1} (s_{t+1})\right\}. 
\label{eq:Arg3}
\end{equation}
 
\begin{equation}
\hat{\upsilon}_t^n(s_t^n)=
  \max_{a_t\in A_{s_t}}\left\{r_t(s_t,a_t) + \mathbb{E}({V}_{t+1} \mid s_t, a_t)\right\} 
\label{eq:uOberveration}
\end{equation}

In line 13, we use a combination of the current value function approximation, $\overline{V}^{n-1}_t(s_t^n)$, and the current observed value, $\hat{\upsilon}_t^n(s_t^n)$ to calculate the smoothed value of being in state $s_t$, $ z_t^n (s_t^n)$.  This combination is weighted based on a \textcolor{black}{stepsize} function, \textcolor{black}{$\alpha_n$}. Traditional \textcolor{black}{stepsize} functions, including $\sfrac{1}{n}$ \citep{Powell} and harmonic  \citep{Powell, Rettke16, Meissnera18}, usually smooth the value function using pure observations for early iterations and gradually put less weight on the observation and put more on the approximations as the number of iterations increases. 

Next, we apply the monotonicity projection operator as defined in Jiang and Powell \cite{Jiang15} (line 14).  \textcolor{black}{We note that the algorithm stores the value of all the states, and t}he monotonicity operator adjusts the value of non-visited states according to the value of the visited state. It ensures that no approximated value violates the monotonicity of the value function. For instance, consider the visited state $s_t$ and an arbitrary state $\widetilde{s_t}$ such that $s_t \le \widetilde{s_t}$ and $\overline{V}_t(s_t) > \overline{V}_t(\widetilde{s_t})$. The monotonicity operator increases the approximation for $\widetilde{s_t}$ up to $\overline{V}_t(s_t)$ and preserves the monotonicity property of the value function. Similarly, we decrease the approximation for lower states (i.e., $\widetilde{s_t} \le s_t$) with higher value approximations than the visited state. Lastly, the algorithm moves forward in time to the next decision epoch until the last decision epoch, wherein a new iteration begins.  Every new iteration starts from an arbitrary state and steps forward in time until the input number of iterations are completed.

\subsubsection{\textcolor{black}{Stepsize} Function} \label{Sec:Stepsize}

Finding a good \textcolor{black}{stepsize} is a problem-dependent procedure that requires empirical experiments \citep{Powell}.  \textcolor{black}{A} \textcolor{black}{stepsize} function, $\alpha_n$ is used to scale the current observed value and $(1-\alpha_n)$ is used to scale the current value function approximation.  Because the value function approximation is often initially set to constant values and therefore is not informative, many \textcolor{black}{stepsize}s start with higher $\alpha_n$ values that emphasize the observed value.  Then, as more iterations are conducted, $\alpha_n$ is decreased in order to place a greater emphasis on the value function approximation. \textcolor{black}{We note that, a \textcolor{black}{stepsize} function needs to satisfy the three basic conditions given by Powell \cite{Powell} to guarantee convergence.} A basic example of such \textcolor{black}{stepsize} function is $\sfrac{1}{n}$. Our preliminary experiments show that using the $\sfrac{1}{n}$ \textcolor{black}{stepsize} function is not appropriate as the rate of decreasing \textcolor{black}{$\alpha_n$} over iterations is too fast for SAIRPs. Instead, we use the harmonic \textcolor{black}{stepsize} function \citep{Powell} \textcolor{black}{that} uses a user-defined parameter, $w$, to controls the rate of decrease over iterations.  In Equation \eqref{eq:css1}, we provide the formal definition of the harmonic \textcolor{black}{stepsize} function.

\begin{equation}
 \alpha_{\textcolor{black}{n}} = \frac{w}{w+n-1}
\label{eq:css1}
\end{equation}

Additionally, \textcolor{black}{we use the Search-Then-Converge (STC) \textcolor{black}{stepsize} rule that can control the rate of decrease in $\alpha_{n}$ by appropriately setting the parameter values \citep{Powell}. Hence, the STC function is suitable for cases like ours that need an extended learning phase \citep{Powell}.  The STC rule was initially proposed by  Darken and Moody \cite{Darken92}; however, we use the generalized STC formula given by George and Powell \cite{George06} presented in Equation \eqref{eq:GPSTC}.}

\textcolor{black}{\begin{equation}
\alpha_{n} = \alpha_0 \frac{\bigg(\frac{\mu_2}{n} + \mu_1 \bigg)}{\bigg(\frac{\mu_2}{n} + \mu_1  + n ^{\zeta}-1\bigg)}
\label{eq:GPSTC}
\end{equation}}

The harmonic and \textcolor{black}{STC} \textcolor{black}{stepsize} functions follow the common process of weighting the observation higher earlier and decreasing this weighting with each iteration. \textcolor{black}{The harmonic and STC stepsize functions are classified as deterministic as their values do not change based on the observations. In contrast, adaptive stepsize functions are sensitive to changes in the observations. To broaden our investigation, we also use the adaptive stepsize first introduced in George and Powell \cite{George06} in Section \ref{ComputationalResults}.}

\section{Computational Results}\label{ComputationalResults}

In this section, we present the results and insights from computational experiments on different SAIRP instances.  First, we present the data used to solve modest and realistic-sized SAIRP instances.  We clarify how we distinguish modest, realistic-sized, and small SAIRP instances.  We denote SAIRP instances as modest if they are optimally solvable using backward induction (BI) (i.e., 7 batteries total and only a one-week time horizon).  We denote SAIRP instances as realistic-sized if they include larger numbers of batteries and the system is considered for longer time horizons (i.e., 100 batteries over a one-month time horizon).  Due to the curses of dimensionality, the realistic-sized instances are not optimally solvable using BI.  Within our monotone approximate dynamic programming with regression-based initialization solution procedure, we initially solve small SAIRP instances that are optimally solvable using BI in a matter of minutes (i.e., 2, 3, or 4 batteries over a one-week time horizon).  After explaining the data for these instances, we proceed by explaining in detail the regression-based initialization used in MADP-RB, the Latin hypercube sampling (LHS) designed experiments, solution method comparison, and solutions/insights for both the modest and realistic-sized SAIRPs.

\subsection{Explanation of Data} \label{S-Data}

For the computational results, we use realistic data representing the costs to recharge, discharge, and replace a battery, the demand, and the battery degradation rate.  To avoid redundancy in presenting similar insights for both EVs and drones, we focus on drones. First, we present the associated data, and then \textcolor{black}{we} show the results of a comprehensive set of experiments in Sections \ref{LHS} and \ref{ADP-large}. 

First, we set parameters associated with drone battery costs.  We set the cost to recharge a battery using the historical power prices from the Capital Region, New York area in 2016 \citep{Grid16}.  We use the time frame with the highest total power price for the modest and realistic-sized instances with the time horizon of a week and a month, respectively. Hence, as displayed in Figure (\ref{fig:season}), we select December 12-18 and the month of December for the modest and realistic-sized problems, using dashed and solid lines, respectively.  We multiply these historical time-varying power prices by the maximum capacity of a battery to calculate the non-stationary, time-varying costs to recharge a battery.  We assume a drone battery has a maximum capacity equal to 400 Wh as in the DJI Spreading Wing S1000 battery \citep{SpreadingW}.  We assume the revenue earned from discharging a battery back to the grid is equal to the charge price. Consistent with level 2 or 3 battery charging  \citep{morrow08, Ribbernick15, Tesla17}, we assume a depleted (full) battery takes one hour (i.e., \textcolor{black}{time between two consecutive} decision epoch\textcolor{black}{s}) of recharging (discharging) to become full (depleted).  We use the purchase price for batteries to calculate the cost of battery replacement.  Assuming the price per kWh of a battery is approximately \$235 \citep{Thinkprogress17}, we set the replacement cost of a drone battery to be \$100. This calculated replacement cost serves as the baseline price\textcolor{black}{,} which is used and varied in the Latin hypercube designed experiments in Section \ref{LHS}.

\begin{figure}[h]
\begin{center}
\includegraphics[scale=0.25]{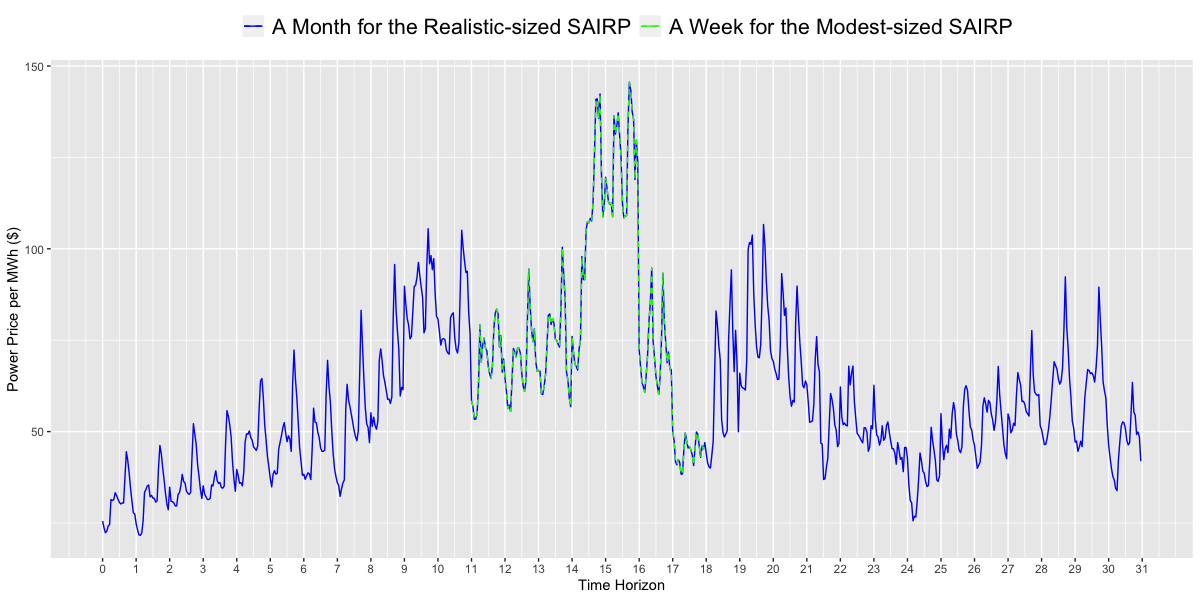}
\vspace{-0.05in}
\caption{Power price fluctuations over December (Realistic-sized SAIRP) and the week of Dec.\ 12-18 (Modest Size SAIRP), 2016 in the Capital Region, New York.}
\label{fig:season}
\end{center}
\end{figure}

In the absence of real data representing the number of customers demanding swaps at the station over time, we use the methodology of Asadi and Nurre Pinkley \cite {Asadi19}, \cite{Widrick16}, and \cite{Nurre14} to derive the mean demand at the swap station over time.  We assume the mean demand, $\lambda_t$, is equivalent to the historical arrival of customers at Chevron gas stations \citep{H2A}.  Using $\lambda_t$, we assume the demand follows a Poisson distribution where $t$ is the hour of the day.  We scale $\lambda_t$ to be in line with the number of batteries in the problem instance. Let $\lambda'_t$ be the scaled demand for $M'$ number of batteries. Because $\lambda_t$ is originally used for $M=7$, we calculate $\lambda'_t$ values by multiplying $\sfrac{M'}{7}$ by $\lambda_t$. In Figure (\ref{demand}), we display the mean demand by hour over a one-week time horizon for the modest instances with $M=7$. For longer time horizons, we assume the mean arrival of demand repeats every week.

\begin{figure}[h]
\begin{center}
\includegraphics[scale=0.2]{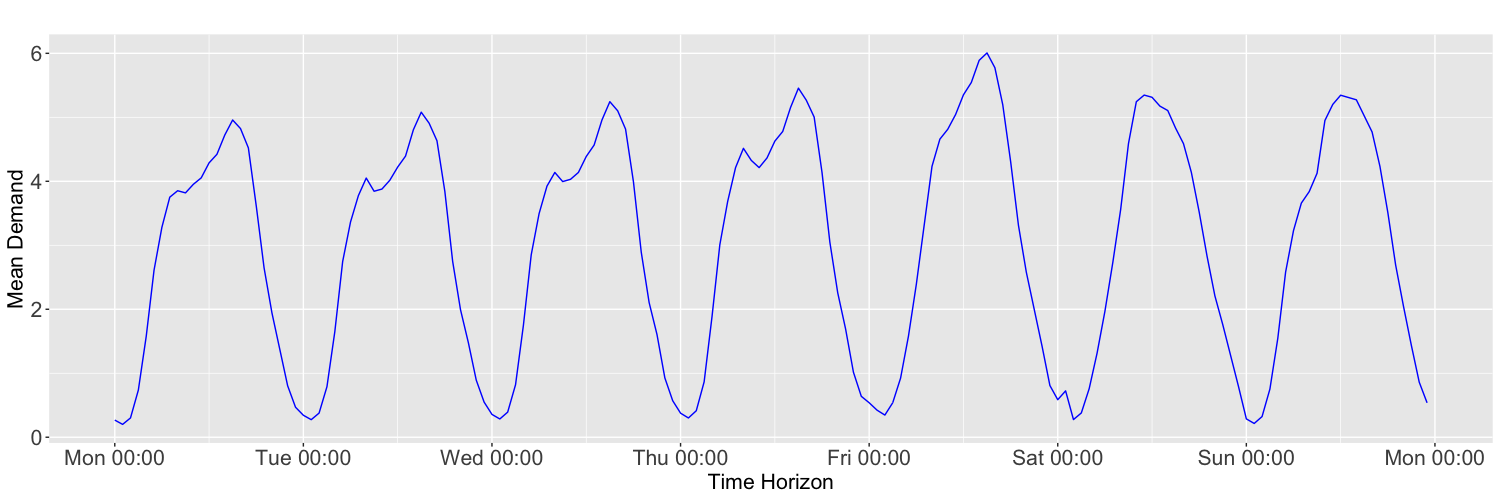}
\label{Figure 3}
\vspace{-0.1in}
\caption{Mean demand for swaps over time.}
\label{demand}
\end{center}
\end{figure}

We use existing studies to calculate the battery degradation rate per cycle.  Although there are several factors \textcolor{black}{that} influence the battery degradation process, research states that the capacity fading has a linear behavior, especially in the first 500 cycles \citep{Dubbary11, Hussein15, Lacey13, Ribbernick15}. 
\textcolor{black}{We note that the standard number of cycles for a drone Lithium-ion battery is 300 to 500 ($\delta^C \approx 0.1\%$) \citep{Liion17}. We select a higher value for the degradation rate to account for the elements that accelerate the degradation process, such as temperature from continuous use and recharging in swap stations. Thus, we select $\delta^C = 2\%$ as the baseline degradation rate and vary this value in our computational experiments to capture how changes in the degradation rate impact the policy and performance of the system. We note, the model is robust in that future experiments can be conducted for different baseline $\delta^C$ values to represent different degradation characteristics.} In general, the industry\textcolor{black}{-}accepted battery end of life value is 80\% of capacity.  As follows, we set the replacement threshold $\theta = 80\%$ \citep{wood11, DEBNATH2014}. \textcolor{black}{We note that $ s^2_t < 80\%$ is equivalent to the absorbing state of the system $s^2_t = 0$ wherein the feasible action set is $A_{(s^1_t, 0)} = \{(0, 0)\}$. Hence, we can only replace batteries when the average capacity of batteries is not less than 80\%.}

\subsection{Regression-Based Initialization} \label{RBI}
In this section, we explain our regression-based method to initialize the MADP-RB intelligently. We examine the empirical experiments of small SAIRPs and detect that the value function of the optimal policies $V_t(s_t)$ is a function of the decision epoch ($t$), the state of the system ($s_t$), and the number of batteries in the station ($M$). To clarify, we display the optimal values $V_t(s_t)$ of scenario 6 from the Latin hypercube designed experiments presented in Section \ref{LHS}.  We display two consecutive decision epochs in Figure (\ref{pmerged}).  In this example, we compute and show the optimal $V_t(s_t)$ when we solve small SAIRPs with $M = 2, 3, 4$. 

As shown in Figure (\ref{pmerged}), the horizontal axis denotes the states, and the vertical axis shows the corresponding values. Among various techniques to estimate or forecast $V_t ^0 (s_t)$, we want a fast and simple method to generate and assign the initial approximations. Therefore, we propose using the linear regression function presented in Equation \eqref{eq:reg} to initialize the approximated value function, $\overline V_t^0 (s_t)$.
\begin{equation}
 \overline V_t ^0 (s_t)= h_0 + h_1M+ h_2s^1_t + h_3s^2_t + h_4t
\label{eq:reg}
\end{equation}
Having solved the small SAIRPs, we find the appropriate values for $h_0$, $h_1$, $h_2$, $h_3$, and $h_4$.  In Sections \ref{LHS} and \ref{ADP-large}, we demonstrate how the intelligent initial value function approximation leads to superior results.

\begin{figure}[h]
\begin{center}
\includegraphics[scale=0.19]{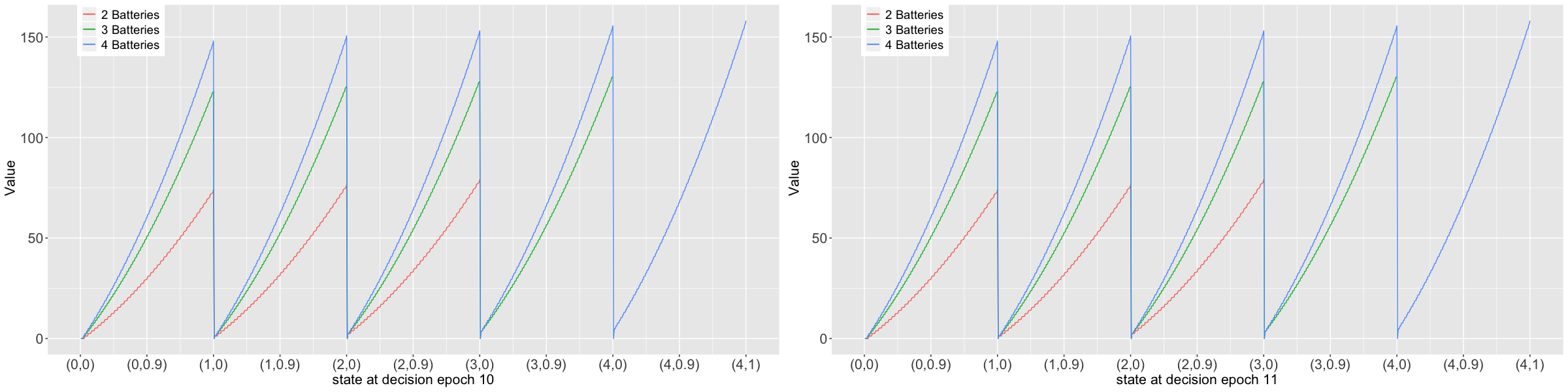}
\caption{An instance of the optimal values over states in two consecutive decision epochs}
\label{pmerged}
\end{center}
\end{figure}

\subsection{Latin Hypercube Designed Experiments for Modest SAIRPs} \label{LHS}

In this section, we perform a Latin hypercube sampling designed experiment \textcolor{black}{that is used to} assess the quality of the solution methods and to deduce insights for the battery swap station application.  A Latin hypercube sampling (LHS) designed experiment is a space-filling design with broad application in computer simulation \citep{Montgomery08}.  Using the LHS configuration of Asadi and Nurre Pinkley \cite{Asadi19}, \textcolor{black}{with the test set of 40 scenarios generated to cover the design space of parameters}, we first quantify the performance of MBI, MADP, and MADP-RB.

In prior work\textcolor{black}{,} Asadi and Nurre Pinkley \cite{Asadi19} was able to optimally solve 40 modest LHS scenarios using BI.  We run all computational tests using a high-performance computer with \textcolor{black}{four shared memory quad Xeon octa-core 2.4 GHz E5-4640 processors and 768GB of memory.} Even on a high-performance computer, the largest scenario Asadi and Nurre Pinkley \cite{Asadi19} is able to optimally solve using BI is with $M=7$, $\varepsilon = 0.001$, and $N=168$ which represents a full week of operations where each decision epoch is one hour.  Using these modest scenarios, herein, we are specifically concerned with quantifying the speed and performance of MBI, MADP, and MADP-RB against a known optimal policy and optimal expected total reward.  The factors and their associated lower and upper bounds are defined as in Table (\ref{LHSfactors}) (we refer the reader to Asadi and Nurre Pinkley \cite{Asadi19} for a complete justification for how these low and high values are calculated).

\textcolor{black}{\emph{\textbf{Memory Intensive Processing vs. Compute Intensive Processing.}} We can implement BI and MBI in two ways. In the first way, which we denote \emph{`Memory Intensive'}, we calculate and store all of the transition probability values, so they are readily available throughout the execution of the algorithm. In the second way, which we denote \emph{`Compute Intensive'}, we do not store any probability values and instead calculate each probability value when needed for calculations through the execution of the algorithm. As is indicated in their names, the \emph{Memory Intensive} way requires more memory and the \emph{Compute Intensive} way requires more time as it needs to calculate probability values numerous times. We implement and use both ways to solve the modest SAIRP instances on a high-performance computer (HPC). We note, on this HPC, we have different options. If we want to run the algorithm for longer, we have up to 72 hours of computational time available per run with access to only 768 GB of memory using four shared memory quad Xeon octa-core 2.4 GHz E5-4640 processors. However, if we want to use more memory, we have access to 3 TB of memory (four shared memory Intel(R) Xeon(R) CPU E7-4860 v2 2.60GHz processors) but only 6 hours of computational time. Thus, our results reflect these limitations. }

\textcolor{black}{In Table \eqref{TM}, we report the memory used and computation times for BI and MBI by instance and method. We use red to highlight the instances where we either exceeded the memory or time available and thus, do not find a solution. As is evident by the results, BI and MBI are not tractable solution methods for even modest-sized SAIRPs. MBI does outperform BI; however, the computation time is significant even for $M = 10$. Although not shown in Table \eqref{TM}, when $M \geq 12$, MBI exceeds the 72-hour time limit. We note that our approximate solution methods do not run into memory issues. Hence, we move forward with the faster processing method, \emph{Memory Intensive}, to solve SAIRPs hereafter.}

\begin{table}[H]\centering
\ra{1.2}
\caption{\textcolor{black}{Time and memory used for different size of SAIRPs using memory intensive and compute intensive methods.}}
\scalebox{0.7}{
\begin{tabular}{@{}ccccccccc@{}}\toprule
& \multicolumn{3}{c}{  Memory Intensive} & & \multicolumn{3}{c}{  Compute Intensive  } \\ \cline {2-4} \cline{6-8} 
& Memory & BI Computation & MBI Computation & & Memory & BI Computation & MBI Computation\\
Size 		& Used (GB)		& Time (h) & Time (h) & & Used (GB)	& Time (h) & Time (h) \\ \midrule 
$M=7$ 	& 830 			&  3.8  	&   3.6  & & 0.2  &  29.2  	&   7.5  \\
$M=8$ 	& 1120 			&  $\textcolor{red}{>6} $ 	& $\textcolor{red}{>6} $    & & 0.2  &  52.6  	&   11.3 \\
$M=9$ 	& 2030 			& $\textcolor{red}{>6} $ 	& $\textcolor{red}{>6} $     & & 0.2  &  $\textcolor{red}{>72} $  	&   16.3 \\
$M=10$ 	& $\textcolor{red}{>3072}$  & $\textcolor{red}{>6} $   	&$\textcolor{red}{>6} $     & & 0.2  &  $\textcolor{red}{>72} $  	&   24.8 \\
 \bottomrule
\end{tabular}
}
\label{TM}
\end{table}

We computationally test MBI for the general case when no limits are placed on how many batteries are replaced per decision epoch and the MADP and MADP-RB methods with the harmonic, \textcolor{black}{Search-Then-Converge (STC), and adaptive} 
\textcolor{black}{stepsize} functions. For the core ADP procedure, we run the scenarios for \textcolor{black}{$\tau = 500000$} iterations. For these \textcolor{black}{stepsize} functions, based on extensive computational experiments, we present the most favorable results.  \textcolor{black}{For the harmonic \textcolor{black}{stepsize} function, Powell \cite{Powell} recommends that $\alpha_{n} < 0.05$ in the last iteration ($\tau =  500000$). Thus, we set $\alpha_{\tau} = 0.05$ and $w = 25000$. We note, our tests show that the average optimality gap increases when $w$ is below 25000, thus we use $w = 25000$.  For setting the STC parameters, we let $\alpha_0 = 1$ to put the total weight on the observation at the beginning of the procedure. This results in $\alpha_1 = 1$. We test six values for $\mu_1$ = 10, 100, 500, 600, 1000, and 10000, five values for $\mu_2$ = 10, 100, 1000, 10000, and 100000, and 3 values for $\zeta$ = 0.6, 0.7, and 0.8 to adjust the tuple of parameters $(\mu_1, \mu_2, \zeta)$. Over these 90 experiments, we observe that $(\mu_1, \mu_2, \zeta)$ = (600, 1000, 0.7) yields the best result in terms of the average optimality gap. 
Thus, we use this tuple of parameter values for the STC stepsize. For the adaptive \textcolor{black}{stepsize} function, we use the setting presented in Powell \cite{Powell}. In the adaptive stepsize, we need to estimate the bias and its variance. For smoothing the observation and approximation of the bias and its variance, Powell \cite{Powell} recommends using a harmonic stepsize that tends to a value between 0.05 and 0.1 in the last iteration.  Consistent with this recommendation, we use the harmonic stepsize function with $w=25000$.}  For the MADP-RB, we set $\overline{M} = 2$, $\overline{T} = T = 168$ hours, MaxIterations = 2 \textcolor{black}{(i.e., 3 iterations are used as the input for the linear regression and value of $iter$ can be 0, 1, and 2)}, and use BI to optimally solve the scenarios for 2, 3, and 4 batteries.  We then plug these results into Equation \eqref{eq:reg} to set the initial approximation as in line 6 of Algorithm \eqref{alg:Amin}.

 \begin{table}[ht]\centering
\ra{1.3}
\caption{Factors with associated low and high values for use in the Latin hypercube designed experiment.}
\scalebox{0.8}{
\begin{tabular}{@{}ccc@{}}\toprule
Factor & Low & High \\ \midrule
Basic revenue per swap ($\beta$)  &1& 3 \\
Replacement cost $L_t$ &  2 & 100  \\ 
Battery degradation factor $(\delta^C)$ & 0.005 & 0.02 \\
 \bottomrule
\end{tabular}}
\vspace{0.1in}
\label{LHSfactors}
\end{table}

\textcolor{black}{Later, we will show that adding the regression-based initialization to MADP significantly improves the quality of solutions. One can argue that initialization using any monotone value function leads to the same result. We tested this argument using a set of arbitrary monotone value functions and report the best-founded function. We call the algorithm with this arbitrary monotone value function initialization MADP-M. We note that in MADP, the initial value function is a constant and equals to $\overline V_t^0 (s_t) = 0 \ \forall s_t \in S$. In MADP-M, the monotone value function used for initialization is $\overline V_t ^0 (s_t) = \big(\frac{\beta(1+s^2_t - 2 \theta)}{1-\theta}\big) (s^1_t) + k(N-t)$ where $k$ is a constant. The first term is a monotone value function that equals the revenue generated from swapping all of the full batteries when in state $s_t$ at time $t$. The second term is a non-negative term with an opposite relationship with time, which means the value of being in every state, $s_t$, is higher at earlier decision epochs. The logic behind adding the non-negative term is as follows. The value of being in state $s_t$, $V_t(s_t)$, accumulates the reward from time $t$ onward to the end of the time horizon, so we may expect to gain a higher reward (profit) over a longer time. Although $V_t(s_t)$ is not always decreasing in $t$, the general trend is observed in the optimal values of small SAIRPs. Our tests show that MADP-M with $k \le 1$ outperforms MADP. Then, we test $k = 0.1, 0.2, \dots, 1$ and observe that MADP-M with $k=0.5$ yields the best result.} 

In Table (\ref{tab:SADP}), we summarize the 40 scenarios and the results comparing the approximate methods to backward induction. Specifically, we calculate an \textcolor{black}{\emph{average optimality gap over all scenarios}}, comparing MADP (monotone approximate dynamic programming), \textcolor{black}{MADP-M (monotone approximate dynamic programming algorithm with arbitrary monotone value function initialization)}, MADP-RB (monotone approximate dynamic programming algorithm with regression-based initialization), and Monotone Backward Induction (MBI) to the optimal expected total reward calculated using BI.  \textcolor{black}{In Equation \eqref{OptGap} for each scenario of} MADP\textcolor{black}{, MADP-M,} and MADP-RB, we input the expected total reward of the last iteration into \emph{the expected reward of the approximated method}. \textcolor{black}{Then, we can calculate the average and maximum optimality gaps over 40 scenarios given by the last two rows of Table (\ref{tab:SADP}).} 

\begin{equation}
\text{Optimality  Gap} = \bigg|\frac{\text{Expected Reward BI - Expected Reward  Approximated  Method} }{\text{Expected Reward BI}}\bigg| * 100\%
\label{OptGap}
\end{equation}

Immediately evident from Table (\ref{tab:SADP}) is that MADP-RB led to significantly smaller average and maximum optimality gaps than Jiang and Powell's \cite{Jiang15} MADP, \textcolor{black}{MADP-M,} and MBI. \textcolor{black}{Overall, we find that MADP-M outperforms MADP in regards to the average optimality gap. This demonstrates that initializing an ADP approach, even with an arbitrarily monotone value function does result in benefit. However, we further find that when considering both the average and maximum optimality gaps, MADP-RB significantly outperforms MADP-M. This demonstrates the significant benefit of the regression-based initialization. We proceed by further analysis of MADP and MADP-RB.} 


We note, MBI is a reasonable approximation for scenarios with high replacement costs\textcolor{black}{, but it does not provide competitive optimality gaps when replacement cost is low.}  However, when we limit the number of batteries replaced in each epoch in accordance with Theorem~\ref{Thm2}, the optimality gap is 0.00\%.  \textcolor{black}{MBI is smarter than BI and can save computational time when searching for the best policies. In other words, we do not need to loop over all the actions when using the monotonicity property of the optimal policy. However, as we showed, BI and even MBI are not computationally tractable for realistically sized instances of SAIRPs, which necessitates using approximate solution methods.}

\begin{table}[!htbp]\centering
\ra{0.77}
\caption{Optimality Gap (\%) of MADP, \textcolor{black}{MADP-M,} MADP-RB, and MBI}
\scalebox{0.62}{
\begin{tabular}{@{}ccrcrrrcrrrcrrrrrr@{}}\toprule
		&	&  		& 	& \multicolumn{3}{c}{  Harmonic \textcolor{black}{Stepsize} }		& &	\multicolumn{3}{c}{STC \textcolor{black}{Stepsize}}	& &	\multicolumn{3}{c}{Adaptive \textcolor{black}{Stepsize}}	& & 	 MBI \\	\cline{5-7}	\cline{9-11} \cline{13-15} 
		\addlinespace[0.2cm]
Scenario   & $\beta$ & $L_t$	& $\delta^c$	& 	MADP 	&  	MADP-M		&   	MADP-RB 	&  	& MADP 	&  	MADP-M		&   	MADP-RB  	&  	& MADP 	&  	MADP-M		&   	MADP-RB & &	 \\	\midrule	\addlinespace[0.05cm]
1	&	1.03	&	45	&	0.009	&	27.11	&	6.66		&	8.26		&	&	30.67	&	8.57		&	10.45	&	&	20.00	&	7.71		&	8.77	&	&	0.00	\\
2	&	1.08	&	82	&	0.011	&	14.14	&	13.68	&    12.10		&	&	16.09	&	12.19	&	10.58	&	&	7.67		&	3.90		&	2.40	&	&	0.09	\\
3	&	1.13	&	61	&	0.005	&	43.50	&	12.14	&	6.27		&	&	47.01	&	17.28	&	7.54		&	&	11.69	&	8.31		&	7.38	&	&	6.05	\\
4	&	1.20	&	69	&	0.009	&	26.08	&	8.28		&	1.26		&	&	30.10	&	10.24	&	3.00		&	&	19.84	&	8.81		&	3.34	&	&	0.00	\\
5	&	1.23	&	89	&	0.008	&	14.80	&	2.20		&	8.62		&	&	18.62	&	0.53		&	7.88		&	&	4.61		&	0.15		&	1.73	&	&	0.00	\\
6	&	1.26	&	47	&	0.010	&	27.76	&	12.06	&	9.72		&	&	30.10	&	13.51	&	10.95	&	&	19.66	&	12.37	&	9.59	&	&	0.00	\\
7	&	1.32	&	98	&	0.019	&	12.11	&	22.34	&	9.15		&	&	13.73	&	22.8		&	9.32		&	&	9.07		&	4.17		&	0.71	&	&	0.00	\\
8	&	1.39	&	94	&	0.011	&	13.49	&	8.38		&	12.76	&	&	16.33	&	7.97		&	10.82	&	&	7.28		&	2.56		&	2.68	&	&	0.00	\\
9	&	1.41	&	87	&	0.014	&	  6.93	&	19.06	&	15.37	&	&	9.39		&	18.87	&	14.55	&	&	4.33		&	6.35		&	3.31	&	&	0.00	\\
10	&	1.48	&	36	&	0.006	&	25.41	&	15.61	&	9.20		&	&	33.78	&	15.97	&	9.58		&	&	26.50	&	14.75	&	3.47	&	&	7.25	\\
11	&	1.53	&	68	&	0.016	&	17.28	&	7.24		&	3.17		&	&	18.05	&	6.56		&	1.60		&	&	13.51	&	1.43		&	0.94	&	&	0.00	\\
12	&	1.56	&	28	&	0.018	&	15.04	&	2.22		&	1.25		&	&	15.37	&	0.78		&	0.06		&	&	12.79	&	6.46		&	6.11	&	&	5.95	\\
13	&	1.60	&	57	&	0.010	&	27.51	&	14.45	&	10.10	&	&	31.56	&	15.74	&	11.84	&	&	20.77	&	14.18	&	1.46	&	&	0.00	\\
14	&	1.68	&	31	&	0.017	&	30.02	&	15.21	&	6.18		&	&	28.94	&	15.29	&	4.92		&	&	23.95	&	14.52	&	10.45&	&	11.05	\\
15	&	1.71	&	62	&	0.006	&	26.79	&	10.13	&	5.25		&	&	33.69	&	10.78	&	4.62		&	&	15.92	&	8.63		&	1.19	&	&	0.00	\\
16	&	1.76	&	44	&	0.017	&	19.14	&	0.36		&	1.14		&	&	20.91	&	0.32		&	0.20		&	&	14.72	&	3.62		&	2.88	&	&	0.00	\\
17	&	1.83	&	55	&	0.016	&	17.36	&	3.00		&	1.84		&	&	17.63	&	2.39		&	1.13		&	&	13.29	&	1.46		&	1.83	&	&	0.00	\\
18	&	1.88	&	51	&	0.019	&	11.38	&	10.00	&	7.55		&	&	13.74	&	10.51	&	6.57		&	&	8.32		&	1.32		&	1.00	&	&	0.00	\\
19	&	1.94	&	66	&	0.012	&	17.12	&	2.01		&	3.99		&	&	18.23	&	2.35		&	3.06		&	&	11.17	&	2.91		&	0.17	&	&	0.00	\\
20	&	1.95	&	73	&	0.019	&	11.22	&	11.59	&	9.15		&	&	13.51	&	10.74	&	8.47		&	&	9.27		&	1.76		&	0.82	&	&	0.00	\\
21	&	2.00	&	83	&	0.012	&	16.41	&	1.95		&	4.27		&	&	19.50	&	3.06		&	3.95		&	&	11.45	&	3.29		&	0.01	&	&	0.00	\\
22	&	2.07	&	20	&	0.013	&	60.56	&	52.87	&	0.38		&	&	53.56	&	54.87	&	1.59		&	&	56.97	&	54.10	&	0.39	&	&	53.97	\\
23	&	2.15	&	90	&	0.013	&	13.40	&	3.44		&	12.19	&	&	15.41	&	2.44		&	11.46	&	&	9.88		&	0.27		&	2.06	&	&	0.00	\\
24	&	2.16	&	41	&	0.007	&	33.37	&	31.29	&	2.02		&	&	27.58	&	29.17	&	3.07		&	&	35.00	&	33.40	&	0.94	&	&	30.69	\\
25	&	2.22	&	79	&	0.015	&	  7.42	&	7.37		&	12.45	&	&	10.11	&	7.45		&	12.03	&	&	5.07		&	3.08		&	1.55	&	&	0.00	\\
26	&	2.27	&	8	&	0.007	&	64.60	&	67.51	&	8.99		&	&	68.96	&	70.09	&	8.91		&	&	66.91	&	70.85	&	8.79	&	&	47.85	\\
27	&	2.32	&	25	&	0.017	&	59.25	&	45.07	&	4.37		&	&	59.72	&	40.87	&	5.03		&	&	44.85	&	44.52	&	5.13	&	&	49.80	\\
28	&	2.37	&	74	&	0.014	&	  8.40	&	8.27		&	13.87	&	&	8.27		&	7.62		&	12.80	&	&	4.45		&	3.34		&	2.95	&	&	0.00	\\
29	&	2.40	&	7	&	0.018	&	68.69	&	69.38	&	15.65	&	&	72.46	&	72.84	&	15.22	&	&	75.11	&	73.61	&	15.47&	&	74.43	\\
30	&	2.50	&	14	&	0.008	&	65.68	&	68.12	&	7.05		&	&	68.96	&	67.59	&	2.48		&	&	64.64	&	68.68	&	7.38	&	&	43.56	\\
31	&	2.51	&	24	&	0.015	&	59.93	&	54.56	&	2.08		&	&	48.80	&	54.09	&	1.12		&	&	59.03	&	51.93	&	3.00	&	&	56.76	\\
32	&	2.56	&	34	&	0.009	&	49.50	&	41.41	&	1.64		&	&	36.87	&	42.90	&	2.67		&	&	49.92	&	38.64	&	0.25	&	&	30.92	\\
33	&	2.64	&	39	&	0.010	&	52.19	&	34.27	&	7.55		&	&	49.33	&	36.56	&	6.50		&	&	37.41	&	32.55	&	7.73	&	&	25.82	\\
34	&	2.68	&	54	&	0.016	&	23.60	&	11.83	&	8.06		&	&	23.82	&	12.74	&	8.13		&	&	19.71	&	12.34	&	9.42	&	&	7.20	\\
35	&	2.71	&	3	&	0.007	&	31.16	&	34.05	&	7.69		&	&	31.22	&	33.44	&	7.45		&	&	33.33	&	33.40	&	7.69	&	&	53.11	\\
36	&	2.77	&	16	&	0.012	&	68.79	&	69.50	&	6.64		&	&	71.81	&	68.35	&	6.62		&	&	68.49	&	66.33	&	6.29	&	&	56.91	\\
37	&	2.83	&	21	&	0.013	&	69.68	&	64.86	&	2.29		&	&	69.44	&	63.10	&	3.12		&	&	67.79	&	63.29	&	3.33	&	&	58.57	\\
38	&	2.90	&	77	&	0.015	&	  8.18	&	4.20		&	5.16		&	&	9.62		&	3.62		&	5.48		&	&	4.58		&	1.10		&	1.14	&	&	0.00	\\
39	&	2.91	&	96	&	0.020	&	10.27	&	7.51		&	11.55	&	&	10.91	&	7.61		&	11.18	&	&	7.05		&	1.17		&	1.51	&	&	0.00	\\
40	&	2.96	&	12	&	0.005	&	59.05	&	66.65	&	7.42		&	&	63.61	&	67.87	&	6.93		&	&	65.77	&	68.09	&	7.69	&	&	17.20  \\ \midrule
\textbf{Avg Gap }	&		&		&		&	30.86	&	23.52	&	7.09	&	&	31.94	&	23.74	&	6.82	&	&	26.54	&	21.23	&	4.07 &	&	15.93	\\
\textbf{Max Gap}	&		&		&		&	69.68	&	69.50	&	15.65	&	&	72.46	&	72.84	&	15.22	&	&	75.11	&	73.61	&	15.47 &	&	74.43	
\\\bottomrule
\end{tabular}}
\label{tab:SADP}
\end{table}

As is expected, ADP methods take significantly less time than BI and MBI, which take on average 3.8 hours and 3.6 hours, respectively.  The computational time of MADP-RB includes three major operations in Algorithm \eqref{alg:Amin}: (i) obtaining the data for the regression function in steps 1-5;
(ii) deriving the regression function in step 6;
and (iii) calculating the monotone ADP results in steps 7-18. 
The average computational time for operations (i), (ii), and (iii) are 902 seconds (0.25 hours), less than 1 second, and \textcolor{black}{1025} seconds (\textcolor{black}{0.28 hours}), respectively. We note that (iii) is equivalent to the average computational time of MADP without the regression-based initialization. \textcolor{black}{In other words, if we execute the initialization steps offline, there is no difference between the computational time of MADP and MADP-RB as finding the coefficients of the regression function takes less than a second. }
As the problem instance of the SAIRP increases significantly, the required memory and time for BI and MBI also increase significantly.  However, MADP and MADP-RB enable us to solve realistic-sized SAIRP instances, as we demonstrate in Section \ref{ADP-large}.

Next, to compare the two ADP methods, we quantify a measure of convergence. One way to evaluate an ADP method is to focus not only on the ending performance but also on the convergence with each iteration. Thus, we calculate \textcolor{black}{\emph{the average optimality gap over all iterations}, $g^s_m$, for each method $m$ and scenario $s$. In this calculation, we average the optimality gap for all iterations ($\tau$) as in Equation \eqref{AvgDev}.}   

\begin{equation}\label{AvgDev}
g^s_m = \frac{1}{\tau}\sum^ \tau_{i=1} \frac{\big|\text{Expected total reward at iteration } i - \text{Optimal expected total reward} \big|}{\text{Optimal expected total reward}}*100\%
\end{equation}

\begin{figure}[!htbp]
\begin{center}
\includegraphics[scale=0.3]{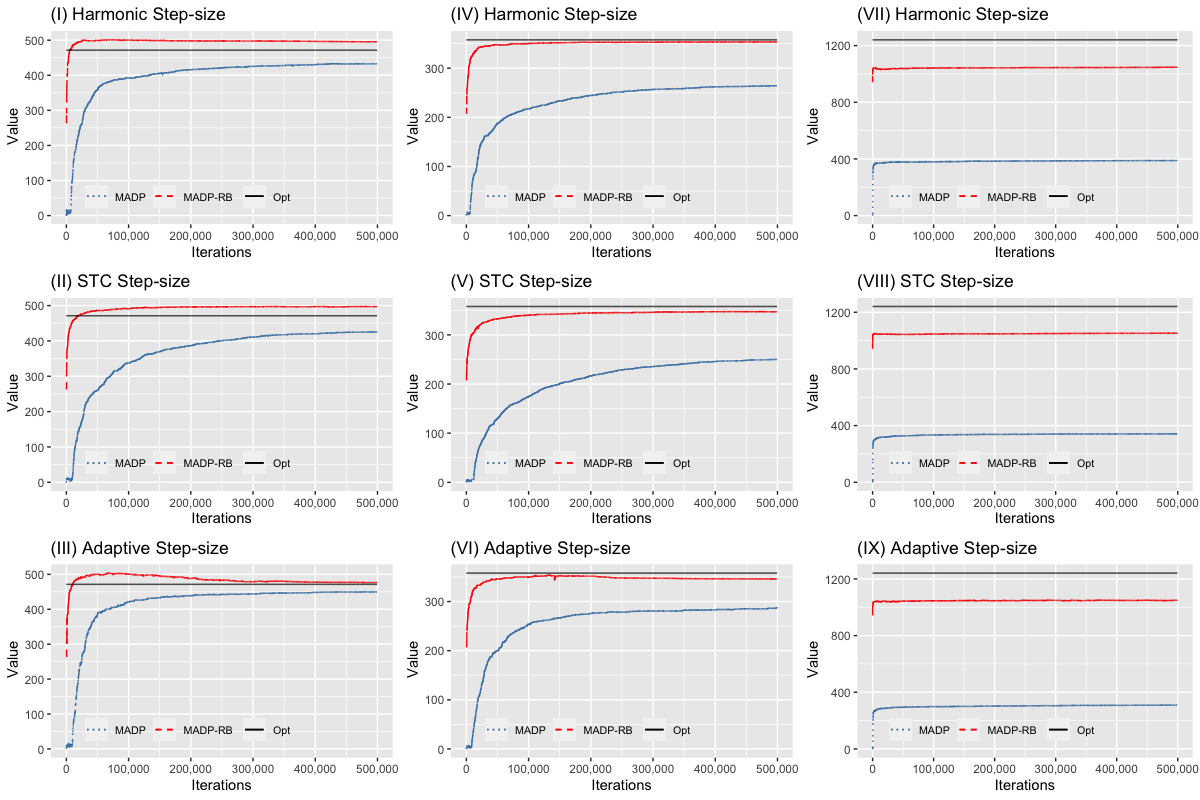}
\caption{Expected Total Reward Convergence of MADP and MADP-RB Using Different \textcolor{black}{Stepsize} Functions Vs Optimal Expected Reward.}
\label{ConvergenceSmall}
\end{center}
\end{figure} 

In Figure (\ref{ConvergenceSmall}), we show the expected total reward (value) as a function of the number of iterations for MADP and MADP-RB with the three different \textcolor{black}{stepsize}s.  The rows of Figure (\ref{ConvergenceSmall}) correspond to the harmonic, \textcolor{black}{STC, and adaptive} 
\textcolor{black}{stepsize}s, respectively.  The columns correspond to the scenarios which result in the \textcolor{black}{low, average, and high} $g^{s}_m$ in the LHS where $m$ includes MADP with the three \textcolor{black}{stepsize}s.  Within each subfigure, we display three lines showing the progression of the value by iteration for MADP and MADP-RB and the optimal value \textcolor{black}{for the state $s_1$, which corresponds to the initial state of the swap station when all batteries are full and new with full capacity.}  From the subfigures, we observe that MADP-RB consistently outperforms MADP in regards to convergence. 

\textcolor{black}{Now, we explain the possibility of convergence to a value greater than optimal (see Figures \eqref{ConvergenceSmall}-(I), \eqref{ConvergenceSmall}-(II), and \eqref{ConvergenceSmall}-(III)). The approximated value depends on two factors: (i) the previous approximation value and (ii) the present observation value. These values are smoothed using a stepsize function and set equal to the approximation value for the next iteration. Hence, as either value can be greater than the optimal value in any iteration, the smoothed value can take a value higher than the optimal value.  For example, suppose the initial approximation is too high. In that case, the smoothed value can be higher than the expected value of being in any state.}

\textcolor{black}{We also calculate the \emph{average optimality gap over iterations over scenarios} for each method $m$ using Equation \eqref{eq:Gbar}. In this equation, $m$ = MADP and MADP-RB, and the total number of scenarios in our Latin Hypercube designed experiments is 40.
\begin{equation}\label{eq:Gbar}
\overline {g_m} = \frac{\sum^{40}_{s=1}  g^s_m}{40} 
\end{equation}} 
In Table (\ref{tab:MADPs}), we display the \textcolor{black}{\emph{average optimality gap over iterations over scenarios}} for MADP and MADP-RB using different \textcolor{black}{stepsize} functions given by \textcolor{black}{Equation \eqref{eq:Gbar}}.  As shown in Table (\ref{tab:MADPs}), the average optimality gap associated with MADP-RB for all \textcolor{black}{stepsize}s is remarkably lower than MADP.  This indicates that MADP-RB is able to more quickly converge to the optimal value as compared to MADP.

\begin{table}[!htbp]\centering
\ra{1.2}
\caption{Comparison between convergence over iterations using MADP and MADP-RB.}
\scalebox{0.8}{
\begin{tabular}{ccccccc}\toprule
\multicolumn{7}{c}{ \textcolor{black}{Average Optimality Gap over Iterations over Scenarios} (\%)} 	\\	
\multicolumn{3}{c}{  MADP}		& &	\multicolumn{3}{c}{MADP-RB}			\\	\cline{1-3}	\cline{5-7}
Harmonic  	&  STC  & Adaptive & &   Harmonic  	&  STC  & Adaptive    \\	\midrule	
36.72	&	40.84  & 32.58 & 	&	6.89	&	8.53 & 5.90	\\
\bottomrule
\end{tabular}}
\vspace{0.1in}
\label{tab:MADPs}
\end{table}

The required computation effort for both ADP approaches depends on, in part, the number of iterations, $\tau$. Thus, we analyze how the value function changes by iteration\textcolor{black}{,} which will inform our choice of $\tau$ for future experiments.  In Figure (\ref{ConvergenceSmall}), we observe the value function plateaus before \textcolor{black}{$\tau = 500000$.}  Furthermore, we notice that both methods do not significantly change after $\tau=100000$ \textcolor{black}{in many scenarios}.  When examining all scenarios \textcolor{black}{solved by MARP-RB}, we calculate the percentage difference between the value at $\tau=100000$ and \textcolor{black}{$\tau=500000$}.  For the harmonic \textcolor{black}{stepsize}, the average percentage difference is \textcolor{black}{1.03\%} and the maximum percentage difference is \textcolor{black}{8.86\%}.  We obtained similar values for \textcolor{black}{the STC and adaptive} \textcolor{black}{stepsize} and scenarios; however, to avoid redundancy, we proceed with an analysis of the harmonic \textcolor{black}{stepsize}.

 As the value does not always convey the operational changes in the policy which decision makers implement, we proceed by analyzing the changes in policy at iteration 100000 and \textcolor{black}{500000}.  We use Equation \eqref{ReleativeDiff} to compute the percentage difference in the policies using at iteration \textcolor{black}{500000} and 100000 when \textcolor{black}{500 sample paths of realized demand generated.  With this equation, we calculate the percentage difference between the summation of the recharging/discharging and replacement actions as we decrease the number of iterations for each scenario. We then average this value over all scenarios. Specifically, in Equation \eqref{ReleativeDiff}, $TC_{s}^{a}$ and $\overline {TC_{s}^{a}}$ represent the total number of actions, recharging/discharging ($a = a^1$) and replacements ($a = a^2$), for scenario $s$ using $\tau = 500000$ and $\tau=100000$, respectively.} That is, $TC_{s}^{a^{1}}$ and $\overline {TC_{s}^{a^{1}}}$ denote the total number of recharging/discharging actions and $TC_{s}^{a^{2}}$ and $\overline {TC_{s}^{a^{2}}}$ are the total number of replacements for scenario $s$ using \textcolor{black}{500000} and 100000 iterations, respectively.  Our results show the average percentage difference in the number of recharged/discharged and replaced batteries over 40 scenarios are \textcolor{black}{6}\% and \textcolor{black}{5}\%, respectively.
\begin{equation}\label{ReleativeDiff}
\begin{split}
\text{ Average (\%) difference \textcolor{black}{for action $a$ over all scenarios}} =  \frac{1}{40} \sum^ {40}_{s=1} \frac{\big| \overline {TC_{s}^{a}} - TC_{s}^{a} \big|}{TC_{s}^{a} }*100\%.
\end{split}
\end{equation}
\noindent Thus, as the changes in the value and policies are not significant, we use $\tau=100000$ in our future computational experiments \textcolor{black}{for solving realistic-sized SAIRPs in Section \ref{ADP-large}}.

\textcolor{black}{Exploiting the monotonicity property does significantly improve the quality of solutions and convergence of ADP methods. To demonstrate this result, we run two versions of standard approximate value iteration (AVI) without and with regression-based initialization, AVI and AVI-RB, respectively. Adding the monotonicity property to AVI and AVI-RB converts them to MADP and MADP-RB, respectively. In Table \eqref{tab:AVI}, we present a summary comparing AVI, MADP, AVI-RB, and MADP-RB  using the harmonic stepsize function and 500000 iterations. In this table, we present the average optimality gap of AVI, MADP, AVI-RB, and MADP-RB over all 40 scenarios of our Latin hypercube designed experiments (see Table \eqref{tab:SADP} for full details of these 40 scenarios). As is evident by these results, using the monotonicity property decreases the average optimality gap. However, we observe that the regression-based initialization is even more impactful than the monotonicity operator. By adding the monotonicity operator, the average optimality gap decreases by 10.31\% (AVI to MADP) but by adding the regression-based initialization, the average optimality gap decreases by 26.42\% (AVI to AVI-RB). The lowest average optimality gap occurs with MADP-RB that has both the monotonicity property and regression-based initialization.}

\begin{table}[H]\centering
\caption{\textcolor{black}{Comparison between the average optimality gap of different approximate solution methods.}}
\begin{center}
\scalebox{0.8}{
\begin{tabular}{c c c c c c } \toprule
Approximate Method & 							AVI & 		MADP &  AVI-RB & 		MADP-RB  \\ \hline  \addlinespace[0.1cm]
Average Optimality Gap (\%) &	41.17\%	&	 30.86\%	& 14.75\%	&		7.09\% \\
\bottomrule
\end{tabular}}
\label{tab:AVI}
\end{center}
\end{table}

\subsection{Monotone ADP Results and Performance for Realistic-sized SAIRPs } \label{ADP-large}
In this next set of computational experiments, we focus on the ability to solve and deduce insights from realistic-sized stochastic SAIRPs.  We proceed by determining the parameters, summarizing the results, including the expected total reward and computational times, presenting sample paths of policies, and analyzing the relationship between the outputs and inputs of the experiments.

For the test instances, we solve all 40 scenarios of the designed experiment presented in Section \ref{LHS} using the realistic data summarized in Section \ref{S-Data}. Specifically, we consider 100 batteries, a one-month time horizon with each decision epoch representing one hour, and scaled mean demand $\lambda'_t = (\lambda_t)( \sfrac{M'}{7})$ where $M'=100$ and $\lambda_t$ is the original mean arrival of demand at time $t$ in line with the original modest-sized problem with $M=7$. Due to the curses of dimensionality, BI and MBI are not capable of solving these large problems; thus, we focus on the performance of MADP and MADP-RB with $\tau=$100000 iterations.

In Table (\ref{tab:Real}), we present the expected total reward and required computational time for the \textcolor{black}{MADP and MADP-RB} methods with \textcolor{black}{harmonic, Search-Then-Converge (STC), and adaptive} \textcolor{black}{stepsize}s. The average computational time of MADP-RB is only 8 minutes longer than MADP for all \textcolor{black}{stepsize} functions as a result of executing the regression-based initialization (see steps 1-6 in Algorithm \eqref{alg:Amin}). We highlight the highest expected reward for each row of Table (\ref{tab:Real}). We acknowledge that the highest expected total reward does not always indicate the lowest optimality gap; however, we use this as a metric of comparison in \textcolor{black}{the} absence of being able to determine the optimal solution and value for these realistic-sized SAIRPs. 

The results are very clear and consistent.  The MADP-RB value is always greater than MADP for all \textcolor{black}{stepsize} functions and scenarios. Within MARP-RB, we observe that harmonic, \textcolor{black}{STC, and adaptive} \textcolor{black}{stepsize} generate the highest expected total reward in \textcolor{black} {11, 3, and 26} scenarios, respectively. 

\begin{table}[h!]\centering
\ra{0.85}
\caption{Results of realistic-sized SAIRPs for the latin hypercube designed experiment.}
\scalebox{0.7}{
\begin{tabular}{@{}c>{\raggedleft}p{2cm}>{\raggedleft}p{2cm}>{\raggedleft}p{2cm}c>{\raggedleft}p{2cm}>{\raggedleft}p{2cm}r @{}}\toprule
		&	\multicolumn{3}{c}{MADP Expected Total Reward (\$)}	& &	\multicolumn{3}{c}{MADP-RB Expected Total Reward (\$)}\\	\cline{2-4}	\cline{6-8} \addlinespace[0.1cm]
Scenario & Harmonic 		&		STC &	 Adaptive 	&	& Harmonic 		&		STC &	 Adaptive   \\ \midrule \addlinespace[0.05cm]
1	&	3468.3	&	2791.7	&	3639.3& &	6964.3	&	7303.6	&\textbf{7884.1}\\	
2	&	5404.1	&	4900.6	&	4773.9&&	7073.1	&	\textbf{7622.2}	&7549.9\\	
3	&	9324.8	&	4219.4	&	8788.3&&	15443.0	&	15840.8	& \textbf{16491.2}\\	
4	&	5397.1	&	3890.3	&	5836.5&&	8684.0	&	9645.6	& \textbf{9710.6}\\	
5	&	4783.1	&	4923.2	&	4986.2&&	10308.5	&	10597.4	& \textbf{11079.1}\\	
6	&	4824.0	&	2363.0	&	5443.4&&	7545.4	&	8123.5	&\textbf{8797.8}\\	
7	&	3779.6	&	4415.5	&	4422.3&&	\textbf{5269.7}	&	5176.2	&5212.2\\	
8	&	4479.5	&	3401.0	&	3521.3&&	8128.5	&	9336.6	&\textbf{9486.3}\\	
9	&	2087.0	&	2546.6	&	2730.5&&	6848.6	&	7446.5	&\textbf{7780.6}\\	
10	&	8461.4	&	8736.6	&	7297.4&&	\textbf{22060.2}	&	19707.0	&20229.4\\	
11	&	3684.1	&	3385.4	&	3749.2&&	6581.5	&	7463.9	& \textbf{7889.3}\\	
12	&	2888.5	&	2842.8	&	3357.9&&	5903.6	&	\textbf{6052.8}	& 5900.4\\	
13	&	6248.1	&	4061.9	&	5994.1&&	11620.6	&	12228.4	&\textbf{12950.7}\\	
14	&	3253.2	&	2578.2	&	3426.1&&	7176.3	&	7258.9	& \textbf{7573.4}\\	
15	&	9983.3	&	9573.4	&	8814.8&&	16484.4	&	16326.9	&\textbf{16614.3}\\	
16	&	4446.9	&	3631.8	&	4474.7&&	7177.7	&	\textbf{7417.1}	&7268.3\\	
17	&	4789.0	&	3875.1	&	5372.4&&	7804.2	&	8601.2	&\textbf{8787.8}\\
18	&	3958.3	&	3634.7	&	4291.3&&	6900.5	&	6998.8 	& \textbf{7260.1}\\
19	&	5855.0	&	4561.5	&	6722.5&&	10514.8	&	10892.3	& \textbf{12124.5}\\	
20	&	4090.8	&	3676.9	&	2051.8&&	7415.2	&	7501.5	&\textbf{7529.3}\\	
21	&	7545.6	&	5562.7	&	7116.9&&	10685.1	&	12764.6	&\textbf{13478.4}\\	
22	&	5454.8	&	5539.8	&	5699.1&&	\textbf{33221.8}	&	31557.5	& 32379.8\\	
23	&	5629.3	&	5707.7	&	7111.5&&	11848.4	&	12005.5	&\textbf{12811.4}\\	
24	&	10593.2	&	10698.7	&	11037.8&&	\textbf{34805.9}	&	32215.2	& 32399.8\\
25	&	4963.0	&	5349.2	&	7186.8&&	10505.9	&	10074.7	& \textbf{11664.4}\\
26	&	51874.6	&	48577.1	&	50855.3&&	84745.6	&	83125.1	& \textbf{85354.7}\\
27	&	5224.7	&	5512.7	&	4699.4&&	16876.2	&	16771.0	& \textbf{17203.0}\\	
28	&	6366.9	&	5006.7	&	6514.1&&	11356.8	&	12461.2	& \textbf{12748.8}\\	
29	&	39244.6	&	34101.7	&	37832.0&&	\textbf{71877.0}	&	70920.1	& 71396.5\\
30	&	25188.0	&	24640.2	&	26046.0&&	79376.8	&	77325.9	& \textbf{80233.1}\\	
31	&	5373.0	&	5971.0	&	6832.8&&	\textbf{33432.0}	&	31564.1	& 32693.5  \\	
32	&	9216.7	&	10299.8	&	10055.9&&	\textbf{44189.8}	&	42772.4	& 43355.7\\	
33	&	9929.9	&	9783.1	&	10293.0 &&	\textbf{29615.5}	&	28922.0	& 28786.7\\	
34	&	5920.8	&	4409.0	&	6452.9 &&	11341.8	&	12115.3& \textbf{13178.0}\\
35	&	1,00038.0	&	98060.7	&	98254.7&&	122397.0	& 122683.0&	\textbf{124779.0}\\	
36	&	18743.7	&	17610.1	&	16362.2&&	\textbf{73562.5}		&	72612.0&	73060.5\\
37	&	8067.0	&	9426.7	&	12138.2&&	\textbf{60294.6}		&	58483.0&	59488.7\\
38	&	7726.4	&	5789.2	&	8285.0&&	12137.7		&	14070.3& \textbf{14674.7}	\\
39	&	5699.1	&	5683.8	&	5168.2&&	11029.9		&	10982.9&	\textbf{11158.6}\\
40	&	68536.0	&	63309.2	&	62951.7&&	\textbf{112542.0}		& 111153.0 &	112073.0	\\  \midrule
Avg CPU 	&	8.5	&	8.5		&	8.6 & &	8.6	&	8.6	& 8.7	\\ 
Time (hours)&		&			&	&&		&		&	\\ 
\bottomrule
\end{tabular}}
\label{tab:Real}
\end{table}

In addition to examining the expected total reward of the two ADP methods, we examine sample paths of the policies \textcolor{black}{when 500 sample paths of realized demand are generated. Our results show that the average percentage of demand met levels off around when the number of sample paths is greater than 200. As a result and to be conservative, f}or all scenarios, we calculate the \textcolor{black}{average} of demand met \textcolor{black}{for 500 sample paths}. We observe that the average of demand met \textcolor{black}{for 500 sample paths} over all the scenarios under MADP-RB is \textcolor{black}{25}\%, \textcolor{black}{29}\%, \textcolor{black}{and 19}\% more than MADP with harmonic, \textcolor{black}{STC, and adaptive} 
\textcolor{black}{stepsize} functions, respectively. The optimal policies are similar across these three \textcolor{black}{stepsize} functions, thus, we proceed by presenting further analysis for MADP-RB policy using the harmonic \textcolor{black}{stepsize} function. 
 
From the sample paths, we observe three typical behaviors within all LHS scenarios. In Figure~\eqref{AvgCapFluctuations}, we displayed \textcolor{black}{examples of} the fluctuations of the average capacity for these three categories \textcolor{black}{when the realized demand equals mean demand}. In Type 1, the average capacity remains over 95\% percent throughout the time horizon (see Figure (\ref{AvgCapFluctuations}a). This is achieved with many replacement actions. The scenarios that exhibit Type 1 behavior have low replacement costs and high revenue per swap (including scenarios 26, 29, 30, 35, 36, 40, 37, and 22). In Type 2 (see Figure (\ref{AvgCapFluctuations}b)), the average capacity does not stay as high as Type 1, but is maintained above 85\% for the entire time horizon except the very end of the month. \textcolor{black}{For 500 sample paths of demand, o}n average, in scenarios with Type 1 behavior, we observe replacement in \textcolor{black}{60}\% of the epochs and when a replacement occurs, \textcolor{black}{15}\% of the batteries are replaced. In contrast, in scenarios with Type 2 behavior, replacement occurs in \textcolor{black}{41}\% of the decision epochs and \textcolor{black}{7}\% of the batteries are replaced when a replacement occurs. 
The scenarios that exhibit Type 2 behavior have average replacement cost and revenue per swap values (includes scenarios 31, 27, 32, 24, 33, 10, and 15).  In these scenarios, it is beneficial to recharge batteries to ensure demand is met, but the station allows battery capacity to degrade and does not replace batteries as frequently.  Finally, in Type 3 (see Figure (\ref{AvgCapFluctuations}c)), the average battery capacity consistently decreases over the time horizon to the replacement threshold of 80\%.  Fewer batteries are replaced as, on average, we observe replacement in \textcolor{black}{11}\% of the decision epochs and only \textcolor{black}{3}\% of batteries are replaced when a replacement occurs.  This demonstrates a reduced priority for maintaining a high battery capacity.  \textcolor{black}{At} epochs when batteries are replaced, a small number of batteries are replaced at a time and are done to ensure the minimum capacity is maintained.  The scenarios that exhibit Type 3 behavior have high replacement costs and low swapping revenue values.

\begin{figure}[h]
\caption{Sample paths of average capacity over time horizon for three types of LHS scenario when realized demand equals means demand.}
\vspace{0.1in}
\includegraphics[scale=0.22]{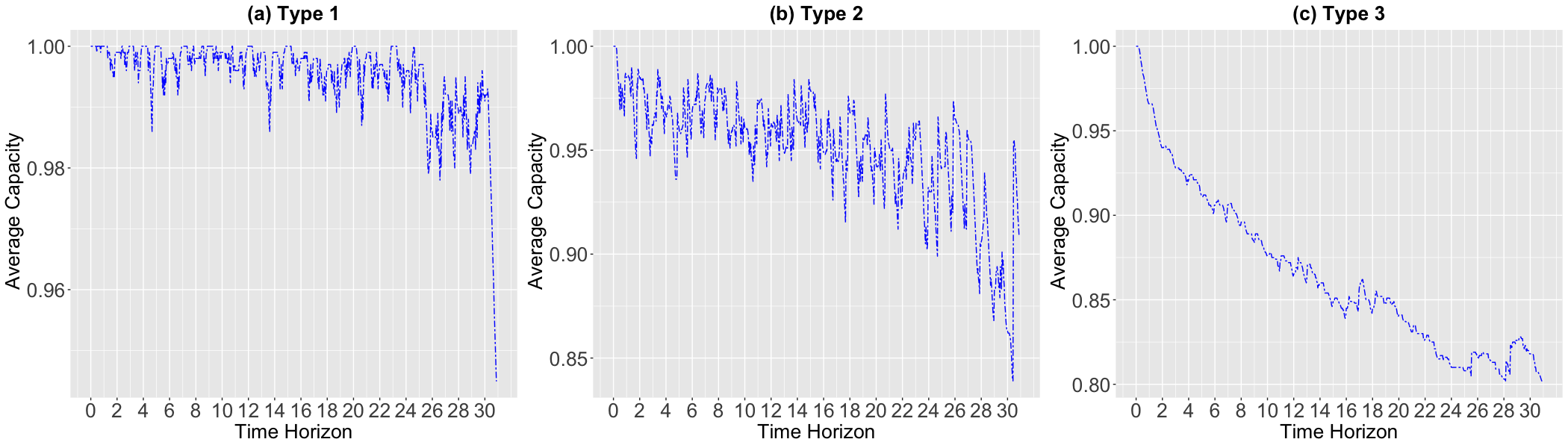}
\label{AvgCapFluctuations}
\end{figure}

Next, we dig into a specific scenario to compare and contrast the policies for the modest and realistic-sized instances. Our intent in this analysis is to determine whether similar actions are taken for the same scenario when considering more batteries over a long time horizon. We observed this is not the case, thereby justifying the need to solve realistic-sized instances. To demonstrate our observations, we present results for scenario 15 \textcolor{black}{that} has meaningful differences with regard to the replacement actions, charging actions, and amount of demand met between the modest and realistic-sized instances. In the realistic-sized problem, on average, we observe replacement in \textcolor{black}{27}\% of the decision epochs, and when replacement occurs \textcolor{black}{4}\% of the batteries are replaced. However, the derived policy in the modest problem consists of no replacement actions. With regard to charging, on average, we see charging actions in \textcolor{black}{97}\% and \textcolor{black}{40}\% of decision epochs for the realistic-sized and modest problems, respectively. With more frequent charging and replacement actions in the realistic-sized problem, the percentage of met demand is higher than the modest problem. The derived policy for the realistic-sized problem satisfies \textcolor{black}{76}\% of demand, which is significantly higher than 54\% of the demand met for the modest problem. From this analysis, while we must solve a finite horizon MDP to capture the time-varying elements, it is necessary to be able to computationally solve instances with longer time horizons and \textcolor{black}{a} larger number of batteries that mimic reality in order to determine the general operating policies. 

In Figure (\ref{Case15Realistic}), we present a sample path of the policy when realized demand equals the mean demand for the realistic-sized SAIRP associated with scenario 15.  
We observe that more batteries are replaced before epochs with high power prices (e.g., 53\% on day 12 before the high power price days of December 12-18), which is consistent among Type 2 scenarios. When this replacement occurs, the battery capacity is higher\textcolor{black}{,} which results in higher swapping revenue to offset to higher costs to recharge batteries. Overall, we observe a consistent trend for recharging batteries and meeting demand during the time horizon. The number of full batteries is kept between 40\% and 60\% of the total number of batteries during the middle of the day and recharging is conducted to raise the number of full batteries to 90\% over night. We show in Figure (\ref{Case15MetDemand}) that the policy meets 100\% of demand during off-peak epochs and more than 50\% during peak epochs.

\begin{figure}[h]
\caption{Sample path of the policy for scenario 15 when realized demand equals mean demand.}
\vspace{0.1in}
\begin{center}
\includegraphics[scale=0.23]{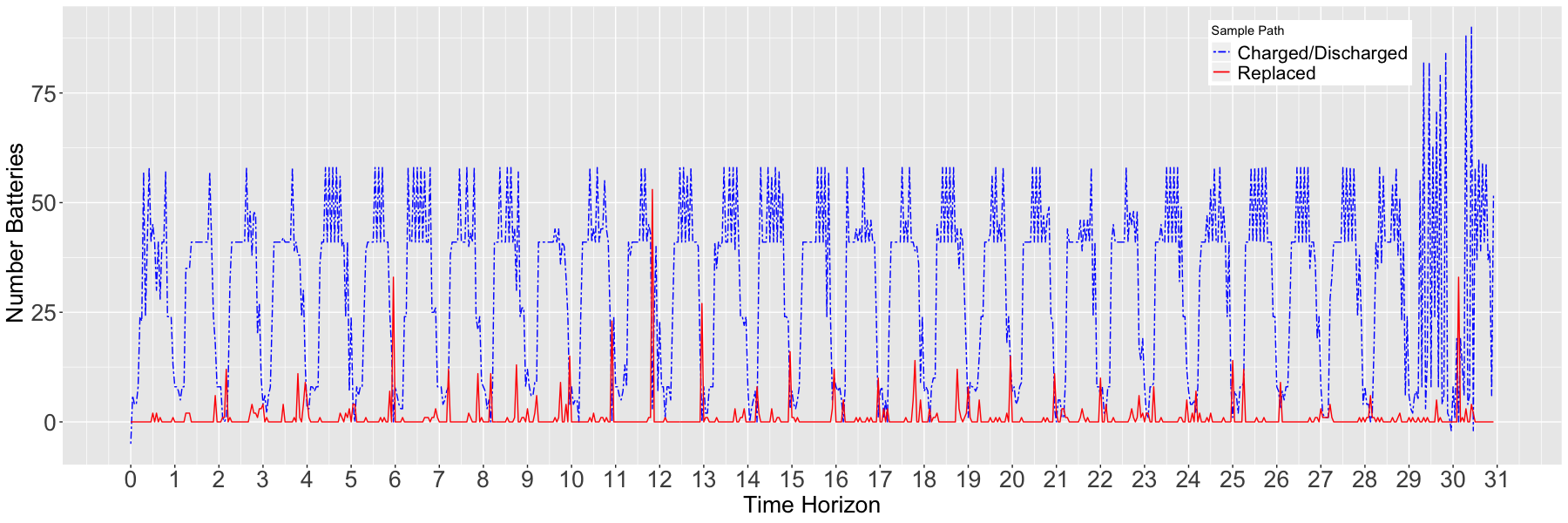}
\end{center}
\label{Case15Realistic}
\end{figure}

\begin{figure}[h]
\caption{Demand and met demand for scenario 15 based on the sample path when realized demand equals mean demand.}
\vspace{0.1in}
\begin{center}
\includegraphics[scale=0.23]{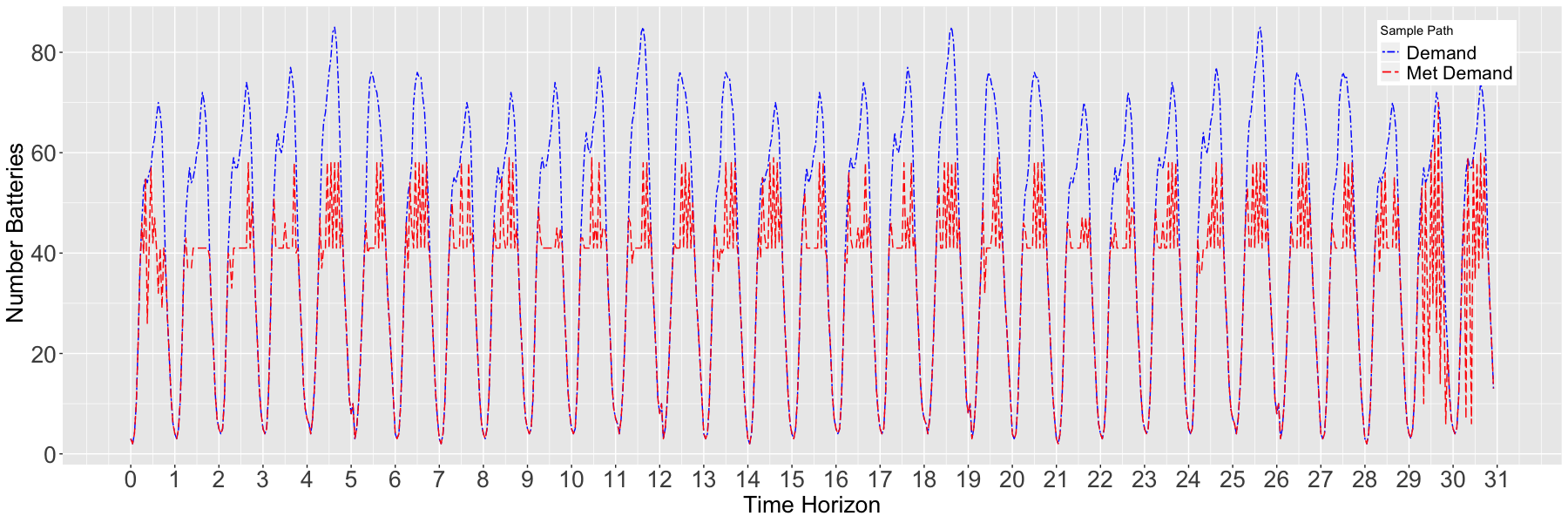}
\end{center}
\label{Case15MetDemand}
\end{figure}

\section{\textcolor{black}{Conclusions}} \label{Conclusion}
We examined a stochastic scheduling, allocation, and inventory replenishment problem (SAIRP) for a battery swap station where there is a direct link between the inventory level and necessary recharging and replacement actions for battery charge and battery capacity. Specifically, the direct link is that the act of recharging a battery to enable short\textcolor{black}{-}term operation is the exact cause for long-term battery capacity degradation. This creates a unique problem where trade-offs between recharging and replacing batteries must be analyzed. We utilized a Markov Decision Process (MDP) model for a battery swap station faced with battery degradation and uncertain arrival of demand for swaps. In the MDP model of the stochastic SAIRP, we determine the optimal policy for charging, discharging, and replacing batteries over time when faced with non-stationary charging prices, non-stationary discharging revenue, and capacity-dependent swap revenue. We prove theoretical properties for the MDP model, including the existence of an optimal monotone policy for the second dimension of the \textcolor{black}{state} when there is an upper bound placed on the number of batteries replaced in each decision epoch. Further, we \textcolor{black}{prove} the monotonicity of the value function. Given these results, we solved the stochastic SAIRP using both backward induction (BI) and monotone backward induction (MBI). However, we run into the curses of dimensionality as we are unable to find an optimal policy for realistic-sized instances.

To overcome these curses, we propose a monotone approximate dynamic programming (ADP) solution method with regression-based initialization (MADP-RB). The MADP-RB builds upon the monotone ADP (MADP) algorithm\textcolor{black}{,} which is fitting for problems with monotone value functions \citep{Jiang15}. In our MADP-RB, we initialize the value function based on an intelligent approximation using regression. 
We used a Latin hypercube designed experiment to test the performance of MBI, MADP, and MADP-RB. In the designed experiment, we examine both modest\textcolor{black}{-}sized problem instances that are optimally solvable using BI and realistic-sized instances. Overall, we observed that MADP-RB resulted in the greatest performance in terms of computational time, optimality gap, convergence, and expected total reward.

Our investigation into SAIRPs opens the avenue for many opportunities for future work. In terms of the model and solution method, researchers should examine state reduction/aggregation approaches and quantify how they impact the quality of the solutions and computational time.  \textcolor{black}{ Future research should consider different mechanisms within a disaggregated MDP model that captures individual battery states and actions. Furthermore, it is valuable to solve the disaggregated MDP using approximate solution methods and compare the insights with the presented aggregated MDP. Moreover, it is interesting to study discretizing the state of charge of batteries and allow partial recharging decisions. } For the battery swap station application, future work should consider different types of charging options, multiple swap station locations, and multiple classes of demand dictated by how long the battery must operate to satisfy the demand (e.g., how far a drone must fly).

\section*{Acknowledgement}
This research is supported by the Arkansas High Performance Computing Center which is funded through multiple National Science Foundation grants and the Arkansas Economic Development Commission.

\clearpage

\section*{APPENDICES}
\subsection{Monotonicity of the Second Dimension of the \textcolor{black}{State} }\label{APP1}



\setcounter{theorem}{0}
\setcounter{lemma}{0}

In this Appendix, in Lemma~\ref{Thm1}, we prove that \textcolor{black}{the stochastic SAIRPs violate the sufficient conditions for the optimality of a monotone policy in the second dimension of the state.}  Then, in Theorem~\ref{Thm2}, we prove the special conditions under which a monotone policy in the second dimension of the \textcolor{black}{state} is optimal.  First, we \textcolor{black}{state Lemma~\ref{Thm1} and then} prove   \textcolor{black}{Lemmas~\ref{lema1}, \ref{lema2}, \ref{lema3}, and \ref{lema4}}  that we use \textcolor{black}{to prove} Lemma~\ref{Thm1} and Theorem~\ref{Thm2}.

\setcounter{lemma}{0}

\begin{lemma}\label{Thm1} \textcolor{black}{The stochastic SAIRPs violate the sufficient conditions for the optimality of a monotone policy in the second dimension of the state.}
 \end{lemma} 

\begin{lemma}\label{lema1}
If  $ s^2_t  \ge \widetilde{s^2 _t}$ and $a^2_t, a^1_t, s^1_t$ are constant then $j^2 \ge \widetilde{j^2 }$.
  \end{lemma}

\begin{proof}
Using Equation \eqref{eq:4}, the second state of the system transitions from $s^2_t$ to $j^2$ and from $\widetilde{s^2_t}$ to $\widetilde{j^2 }$ when action $a_t = (a^1_t, a^2_t)$ is taken. In Equation \eqref{eq:4}, we use the $round()$ function to return values in the discretized state space.  Without loss of generality, we examine the state transition for the second state of the system without the $round()$ function, as the $round()$ function does not alter the validity of the arguments.  Thus, we have

\begin{equation}\label{Added1}
\begin{gathered}
j^2 = s^2_{t+1}= \frac{1}{M}\left[s^2_t ( M - a^2_t ) + a^2_t - \delta^C (a^{1+}_t + a^{1-}_t)\right], \\
\widetilde{j^2 } = \widetilde{s^2_{t+1}}= \frac{1}{M}\left[\widetilde{s^2_t} (M - a^2_t ) + a^2_t - \delta^C (a^{1+}_t + a^{1-}_t)\right]. 
\end{gathered}
\end{equation}
We start from $j^2 \ge \widetilde{j^2 }$ and reduce it to an always true statement.
\begin{center}
$j^2 \ge \widetilde{j^2 } \Leftrightarrow $ \\
$\frac{1}{M}\left[s^2_t ( M - a^2_t ) + a^2_t - \delta^C (a^{1+}_t + a^{1-}_t)\right] \ge \frac{1}{M}\left[\widetilde{s^2_t} (M - a^2_t ) + a^2_t - \delta^C (a^{1+}_t + a^{1-}_t)\right] \Leftrightarrow $ \\
$s^2_t ( M - a^2_t ) + a^2_t - \delta^C (a^{1+}_t + a^{1-}_t) \ge \widetilde {s^2_t} ( M - a^2_t ) + a^2_t - \delta^C (a^{1+}_t + a^{1-}_t) \Leftrightarrow $ \\
$s^2_t ( M - a^2_t ) \ge \widetilde {s^2_t} ( M - a^2_t ).$  \\
\end{center}
As $a^2_t \leq M$,  $(M - a^2_t )$ is always non-negative.  Thus, we can divide both sides of the last inequality by this non-negative value and get $ s^2_t  \ge \widetilde{s^2 _t}$ which is always true.

\end{proof}

\begin{lemma}\label{lema2}
 If  $ s^2_t$, $a^1_t$, and $\theta$ are fixed, then replacing less than \[	\xi
=
\Bigl\lfloor\dfrac{\delta^C (a^{1+}_t + a^{1-}_t)-M(s^2_t - \theta)}{1-s^2_t}\Bigr\rfloor\] causes a transition to absorbing state $\eta$.
\end{lemma}
  
\begin{proof}
Transition from state $s^2_t$ to absorbing state $\eta$ means $s^2_{t+1}  < \theta$. From Equation \eqref{Added1} we know that
\begin{equation*}
s^2_{t+1}= \frac{1}{M}\left[s^2_t ( M - a^2_t ) + a^2_t - \delta^C (a^{1+}_t + a^{1-}_t)\right].
\end{equation*}
Therefore, 
\begin{align*}
s^2_{t+1}=j^2 < \theta \Leftrightarrow  \\
\frac{1}{M}[s^2_t ( M - a^2_t ) + a^2_t - \delta^C (a^{1+}_t + a^{1-}_t)] < \theta \Leftrightarrow \\
s^2_t ( M - a^2_t ) + a^2_t - \delta^C (a^{1+}_t + a^{1-}_t) < M\theta \Leftrightarrow \\
(1- s^2_t ) a^2_t < M\theta + \delta^C (a^{1+}_t + a^{1-}_t) - Ms^2_t \Leftrightarrow \\
 a^2_t < \frac{M\theta + \delta^C (a^{1+}_t + a^{1-}_t) - Ms^2_t}{1- s^2_t } \Leftrightarrow \\
  a^2_t < \frac{  \delta^C (a^{1+}_t + a^{1-}_t) - M(s^2_t - \theta)}{1- s^2_t }. 
 \end{align*} 
Because $a^2_t$ is an integer, the upper bound can be derived using the floor function. 
\[a^2_t < 	\xi
=
\Bigl\lfloor\dfrac{\delta^C (a^{1+}_t + a^{1-}_t) - M(s^2_t - \theta)}{1-s^2_t}\Bigr\rfloor. \]

\end{proof}

In this Appendix, when we use the transition probability and reward functions with only the second dimension of the state and action as arguments, we assume that the first dimension of the system is fixed (i.e., $s^1_t$ and $a^1_t$ are constants). 

\begin{lemma}\label{lema3}
For the stochastic SAIRP, $q_t(k \mid s^2_t, a^2_t)$ is a deterministic function equal to 
\begin{equation*}
q_t(k \mid s^2_t, a^2_t) =\left\{
	\begin{array}{ll}
	0 & \text{if} \quad j^2 < k, \\
		p_t(j^2 \mid s^2_t, a^2_t) =1 & \text{if} \quad j^2 \ge k . \\
	\end{array}
\right.	
\end{equation*}
\end{lemma}

\begin{proof}
As defined in Section \ref{ProblemStatement}, 
\begin{align*}
q_t(k \mid s, a) = \sum_{j=k}^{\infty} p_t(j \mid s, a).
\end{align*}

Hence, when the first dimension of the system is fixed, we have 
\begin{align*}
q_t(k \mid s^2_t, a^2_t) = \sum_{j^2=k}^{\infty} p_t(j^2 \mid s^2_t, a^2_t).
\end{align*}

From the second state transition function given by Equation \eqref{eq:4a}, the value of second state in the future, $s^2_{t+1}= j^2$, does not depend on $D_t$.  
From Lemma~\ref{lema2}, if we start from $s^2_t$ and take action $a^2_t$, the capacity will transition to the absorbing state $j^2 = \eta$, if \[a^2_t< \Bigl\lfloor{	\xi
}\Bigr\rfloor \] or it will transition to state $j^2 \ge \theta$, if \[a^2_t \ge \Bigl\lfloor{	\xi}\Bigr\rfloor.\]
Thus, this is a deterministic transition with into two intervals, $\big[\eta,j^2\big]$ and $(j^2, 1\big]$.  If $k \in (j^2, 1\big]$, then \mbox{$q_t(k \mid s^2_t, a^2_t)$} equals zero. 
However, if $k \in \big[\eta,j^2\big]$, then $k \leq j^2$ and $q_t(k \mid s^2_t, a^2_t)$ equals 1.

\end{proof}

We use the notation we define in Definition~\ref{Def1} in Lemma~\ref{lema4}, Lemma~\ref{Thm1}, and Theorem~\ref{Thm2}.

\begin{definition}{}\label{Def1}
The second dimension of the state transitions to the following state when starting from $s^2_t$ and taking action $a^2_t$.
\begin{equation}\label{DFA1}
j^2_A =  \frac{1}{M}\left[s^2_t ( M - a^2_t ) + a^2_t - \delta^C (a^{1+}_t + a^{1-}_t)\right].
\end{equation}
The second dimension of the state transitions to the following state when starting from $\widetilde {s^2_t}$ and taking action $a^2_t$.
\begin{equation}\label{DF1B}
j^2_B =\frac{1}{M}\left[\widetilde{s^2_t} ( M - a^2_t ) + a^2_t - \delta^C (a^{1+}_t + a^{1-}_t)\right]. 
 \end{equation}
 The second dimension of the state transitions to the following state when starting from $s^2_t$ and taking action $\widetilde {a^2_t}$.
 \begin{equation}\label{DF1C}
j^2_C=  \frac{1}{M}\left[s^2_t ( M - \widetilde{a^2_t} ) + \widetilde{a^2_t} - \delta^C (a^{1+}_t + a^{1-}_t)\right].
 \end{equation}
The second dimension of the state transitions to the following state when starting from $\widetilde {s^2_t}$ and taking action $ \widetilde a^2_t$.
\begin{equation}\label{DF1D}
 j^2_D = \frac{1}{M}\left[\widetilde{s^2_t }( M - \widetilde{a^2_t }) + \widetilde{a^2_t} - \delta^C (a^{1+}_t + a^{1-}_t)\right].
 \end{equation}\\
 \end{definition} 
 
 \begin{lemma}\label{lema4}
If $ s^2_t  \ge \widetilde{s^2 _t} $ and $ a^2_t  \ge \widetilde{a^2 _t}$, then $j^2_A \ge \max(j^2_B, j^2_C) \ge \min(j^2_B, j^2_C) \ge j^2_D$ .
  \end{lemma} 
  \begin{proof} 
Based on Definition \ref{Def1}, we need to show $j^2_A$ has the greatest value and $j^2_D$ has the lowest value. Thus, as follows we prove $j^2_A \ge j^2_B$, $j^2_A \ge j^2_C$, $j^2_A \ge  j^2_D$, $j^2_B \ge j^2_D$, and $j^2_C \ge j^2_D$.  We proceed with a proof by cases starting with the claim and reducing it to an always true statement.
 \begin{enumerate}[i.]
 \item $j^2_A \ge j^2_D$
 $$ j^2_A \ge j^2_D \Leftrightarrow$$
 $$ \frac{1}{M}\left[s^2_t ( M - a^2_t ) + a^2_t - \delta^C (a^{1+}_t + a^{1-}_t)\right] \ge  \frac{1}{M}\left[\widetilde{s^2_t }( M - \widetilde{a^2_t }) + \widetilde{a^2_t} - \delta^C (a^{1+}_t + a^{1-}_t)\right]  \Leftrightarrow$$
 $$ s^2_t ( M - a^2_t ) + a^2_t - \delta^C (a^{1+}_t + a^{1-}_t) \ge \widetilde{s^2_t }( M - \widetilde{a^2_t }) + \widetilde{a^2_t} - \delta^C (a^{1+}_t + a^{1-}_t)  \Leftrightarrow$$
 $$ Ms^2_t - s^2_ta^2_t + a^2_t - M \widetilde{s^2_t } + \widetilde{s^2_t } \widetilde{a^2_t } -\widetilde{a^2_t}  \ge 0  \Leftrightarrow$$
  $$ M(s^2_t  - \widetilde{s^2_t }) + a^2_t (1-s^2_t ) -\widetilde{a^2_t}(1-\widetilde{s^2_t })  \ge 0.$$
  We know $a^2_t (1-{s^2_t }) \ge  \widetilde{a^2_t }(1- {s^2_t })$, so it suffices to show that
  $$ M(s^2_t  - \widetilde{s^2_t }) + \widetilde {a^2_t }(1-  {s^2_t} ) -\widetilde{a^2_t}(1-\widetilde{s^2_t })  \ge 0  \Leftrightarrow$$
   $$ M(s^2_t  - \widetilde{s^2_t }) + \widetilde {a^2_t }(\widetilde{s^2_t } -  {s^2_t} )  \ge 0  \Leftrightarrow$$
   $$ (M- \widetilde {a^2_t }) (s^2_t  - \widetilde{s^2_t }) \ge 0.$$
   Because $ M \ge \widetilde {a^2_t } \ \text{and} \ s^2_t \ge \widetilde{s^2_t }$ the last statement is always non-negative and we prove our claim that $j^2_A \ge j^2_D$.

\item $j^2_A \ge  j^2_B$
 $$j^2_A \ge  j^2_B \Leftrightarrow$$
 $$ \frac{1}{M}\left[s^2_t ( M - a^2_t ) + a^2_t - \delta^C (a^{1+}_t + a^{1-}_t)\right] \ge  \frac{1}{M}\left[\widetilde{s^2_t} ( M - a^2_t ) + a^2_t - \delta^C (a^{1+}_t + a^{1-}_t)\right]  \Leftrightarrow$$
  $$ s^2_t ( M - a^2_t ) + a^2_t - \delta^C (a^{1+}_t + a^{1-}_t) \ge  \widetilde{s^2_t} ( M - a^2_t ) + a^2_t - \delta^C (a^{1+}_t + a^{1-}_t)  \Leftrightarrow$$
    $$ s^2_t ( M - a^2_t )  \ge  \widetilde{s^2_t} ( M - a^2_t ).$$
Because $ M \ge \widetilde {a^2_t }$ we have
 $$ s^2_t   \ge  \widetilde{s^2_t}.$$
The last statement is always true and we prove our claim that $j^2_A \ge j^2_B$. 
 \item $j^2_A \ge  j^2_C$ 
 $$ j^2_A \ge j^2_C \Leftrightarrow$$
 $$ \frac{1}{M}\left[s^2_t ( M - a^2_t ) + a^2_t - \delta^C (a^{1+}_t + a^{1-}_t)\right] \ge  \frac{1}{M}\left[s^2_t ( M - \widetilde{a^2_t} ) + \widetilde{a^2_t} - \delta^C (a^{1+}_t + a^{1-}_t)\right] \Leftrightarrow$$
  $$ s^2_t ( M - a^2_t ) + a^2_t - \delta^C (a^{1+}_t + a^{1-}_t) \ge  s^2_t ( M - \widetilde{a^2_t} ) + \widetilde{a^2_t} - \delta^C (a^{1+}_t + a^{1-}_t) \Leftrightarrow$$
  $$ s^2_t ( M - a^2_t ) + a^2_t  \ge  s^2_t ( M - \widetilde{a^2_t} ) + \widetilde{a^2_t}  \Leftrightarrow$$
  $$ s^2_t ( M - a^2_t - M+  \widetilde{a^2_t} ) + a^2_t  - \widetilde{a^2_t} \ge  0  \Leftrightarrow$$
    $$ s^2_t (   \widetilde{a^2_t} - a^2_t  ) + a^2_t  - \widetilde{a^2_t} \ge  0  \Leftrightarrow$$
        $$ (1 - s^2_t )( a^2_t  - \widetilde{a^2_t}) \ge  0 .$$
   Because $ (1 - s^2_t ) \ge 0 \ \text{and} \ (a^2_t - \widetilde {a^2_t }) \ge 0 $ the last statement is always non-negative and we prove our claim that $j^2_A \ge j^2_C$. \\
   So far we know $j^2_A \ge \max(j^2_B, j^2_C) \ge \min(j^2_B, j^2_C)$ and $j^2_A \ge j^2_D$. Now, we show $\min(j^2_B, j^2_C) \ge j^2_D$.  It suffices to show $ j^2_B \ge j^2_D$ and $j^2_C \ge j^2_D$. First, we show $ j^2_B \ge j^2_D$.
   \item $ j^2_B \ge j^2_D$
    $$j^2_B \ge j^2_D \Leftrightarrow$$
    $$ \frac{1}{M}\left[\widetilde{s^2_t} ( M - a^2_t ) + a^2_t - \delta^C (a^{1+}_t + a^{1-}_t)\right]  \ge  \frac{1}{M}\left[\widetilde{s^2_t }( M - \widetilde{a^2_t }) + \widetilde{a^2_t} - \delta^C (a^{1+}_t + a^{1-}_t)\right] \Leftrightarrow$$
        $$ \widetilde{s^2_t} ( M - a^2_t ) + a^2_t - \delta^C (a^{1+}_t + a^{1-}_t)  \ge  \widetilde{s^2_t }( M - \widetilde{a^2_t }) + \widetilde{a^2_t} - \delta^C (a^{1+}_t + a^{1-}_t) \Leftrightarrow$$
    $$ \widetilde{s^2_t} ( M - a^2_t ) + a^2_t \ge \widetilde{s^2_t }( M - \widetilde{a^2_t }) + \widetilde{a^2_t} \Leftrightarrow$$
    $$ \widetilde{s^2_t} ( \widetilde{a^2_t } - a^2_t ) + (a^2_t - \widetilde{a^2_t})  \ge 0 \Leftrightarrow$$
    $$ (1-\widetilde{s^2_t})( a^2_t - \widetilde{a^2_t } )  \ge 0.$$
   Because $ (1 - \widetilde{s^2_t }) \ge 0 \ \text{and} \ (a^2_t - \widetilde {a^2_t }) \ge 0 $ the last statement is always non-negative and we prove our claim that $j^2_B \ge j^2_D$. 

\item $j^2_C \ge  j^2_D$
 $$j^2_C \ge j^2_D \Leftrightarrow$$
    $$\frac{1}{M}\left[s^2_t ( M - \widetilde{a^2_t} ) + \widetilde{a^2_t} - \delta^C (a^{1+}_t + a^{1-}_t)\right]   \ge \frac{1}{M}\left[\widetilde{s^2_t }( M - \widetilde{a^2_t }) + \widetilde{a^2_t} - \delta^C (a^{1+}_t + a^{1-}_t)\right]  \Leftrightarrow$$
        $$s^2_t ( M - \widetilde{a^2_t} ) + \widetilde{a^2_t} - \delta^C (a^{1+}_t + a^{1-}_t)   \ge \widetilde{s^2_t }( M - \widetilde{a^2_t }) + \widetilde{a^2_t} - \delta^C (a^{1+}_t + a^{1-}_t)  \Leftrightarrow$$
    $$ {s^2_t }( M - \widetilde{a^2_t }) + \widetilde{a^2_t} \ge \widetilde{s^2_t }( M - \widetilde{a^2_t }) + \widetilde{a^2_t} \Leftrightarrow$$
    $$ {s^2_t }( M - \widetilde{a^2_t })  \ge \widetilde{s^2_t }( M - \widetilde{a^2_t })   \Leftrightarrow$$
    $$  {s^2_t } \ge \widetilde{s^2_t }. $$
    
The last statement is always true and we prove our claim that $j^2_C \ge j^2_D$. From iv and v,
 we conclude that $\min(j^2_B, j^2_C)  \ge j^2_D$.  From i, ii, iii, iv, and v, we prove $j^2_A \ge \max(j^2_B, j^2_C) \ge \min(j^2_B, j^2_C) \ge j^2_D$.
\end{enumerate}

\end{proof}

Using Lemmas~\ref{lema1}, \ref{lema2}, \ref{lema3}, and \ref{lema4}, we prove in Lemma~\ref{Thm1} that  \textcolor{black}{the stochastic SAIRPs violate the sufficient conditions for the optimality of a monotone policy in the second dimension of the state}.   Then, in Theorem~\ref{Thm2}, we prove the special conditions under which a monotone policy is optimal in the second dimension of the \textcolor{black}{state}. 

  
\setcounter{theorem}{0}
\setcounter{lemma}{0}
  

\begin{lemma}\label{Thm1} \textcolor{black}{The stochastic SAIRPs violate the sufficient conditions for the optimality of a monotone policy in the second dimension of the state.}
 \end{lemma} 
\begin{proof}
First, we need to show that at least one of the following conditions is not always satisfied for our problem. 
We will prove that condition 4 is not satisfied; That is, $q_t(k \mid s^2_t , a^2_t)$  \textbf{is not} subadditive on $S^2 \times A'$ for every $k \in S^2$ where $A'$ is the set of possible actions in the second dimension, independent of the state of the system.

 \begin{enumerate}

\item $r_t (s^2_t, a^2_t )$ is non-decreasing in $s^2_t$ for all $a^2_t \in A'$. \\
If  $ s^2_t  \ge \widetilde{s^2 _t} $, it suffices to show that  $r_t (s^2_t, a^2_t ) \ge r^2_t (\widetilde{s^2_t}, a^2_t )$. From Equations \eqref{eq:imRew} and \eqref{eq:rho} we know
$$r_t (s^2_t, a^2_t ) = \frac{\beta (1+ s^2_t - 2\theta)} { 1- \theta} \left[\min\{D_t , s^1_t - a ^{1-}_t\}\right] - K_t a^{1+}_t + J_t a^{1-}_t - L_t a^2_t.$$ 
We start with our claim, $r_t (s^2_t, a^2_t ) \ge r_t (\widetilde{s^2_t}, a^2_t )$, and reduce it to an always true statement.
\begin{center}
$r_t (s^2_t, a^2_t ) \ge r_t (\widetilde{s^2_t}, a^2_t )  \Leftrightarrow $ \\
$ \frac{\beta (1+ s^2_t - 2\theta)} { 1- \theta} \left[\min\{D_t, s^1_t - a ^{1-}_t\}\right] - K_t a^{1+}_t + J_t a^{1-}_t - L_t a^2_t \ge \frac{\beta (1+ \widetilde{s^2_t} - 2\theta)} { 1- \theta} \left[\min\{D_t , s^1_t - a ^{1-}_t\}\right] - K_t a^{1+}_t + J_t a^{1-}_t - L_t a^2_t  \Leftrightarrow $ \\
$\frac{\beta (1+ s^2_t - 2\theta)} { 1- \theta} \ge \frac{\beta (1+ \widetilde{s^2_t} - 2\theta)}{ 1- \theta}   \Leftrightarrow $ \\
$ s^2_t  \ge \widetilde{s^2 _t}. $
\end{center}

\item $q_t(k \mid s^2_t , a^2_t)$ is non-decreasing in $s^2_t$ for all $k \in S^2$ and $a \in A'$. \\
If $ s^2_t  \ge \widetilde{s^2 _t} $, we seek to show $$q_t(k \mid s^2_t , a^2_t) \ge q_t(k \mid \widetilde{s^2_t} , a^2_t).$$


It is sufficient to show that if $j^2_A \ge j^2_B$ from Definition \ref{Def1}, then $q_t(k \mid s^2_t , a^2_t) \ge q_t(k \mid \widetilde{s^2_t} , a^2_t)$. We show $j^2_A \ge j^2_B$ in Lemma~\ref{lema1} and prove our claim.
\item $r_t (s^2_t, a^2_t )$ is a subadditive function on $S^2 \times A'$. \\
 For $r_t (s^2_t, a^2_t )$ to be subadditive, it means the incremental effect on the expected total reward of replacing less batteries is less when the average capacity is greater. 
Consider, $ s^2_t  \ge \widetilde{s^2 _t}$  and $ a^2_t  \ge \widetilde{a^2 _t}$. It suffices to show that
$$r_t (s^2_t, a^2_t ) + r_t (\widetilde {s^2_t}, \widetilde{a^2_t} ) \le  r_t (s^2_t, \widetilde {a^2_t} ) + r_t (\widetilde {s^2_t}, {a^2_t} ) \Leftrightarrow $$ 
$$  \sum_{j=0}^{\infty} P (D_t = j ) \frac{\beta (1+ s^2_t - 2\theta)} { 1- \theta} \left[\min\{j , s^1_t - a ^{1-}_t\}\right] - K_t a^{1+}_t + J_t a^{1-}_t - L_t a^2_t$$   
$$+ \sum_{j=0}^{\infty} P (D_t = j ) \frac{\beta (1+ \widetilde {s^2_t} - 2\theta)} { 1- \theta} \left[\min\{j , s^1_t - a ^{1-}_t\}\right] - K_t a^{1+}_t + J_t a^{1-}_t - L_t \widetilde{a^2_t} \le $$ 
$$ \sum_{j=0}^{\infty} P (D_t = j ) \frac{\beta (1+ {s^2_t} - 2\theta)} { 1- \theta} \left[\min\{j , s^1_t - a ^{1-}_t\}\right] - K_t a^{1+}_t + J_t a^{1-}_t - L_t \widetilde{a^2_t}  $$
$$+ \sum_{j=0}^{\infty} P (D_t = j ) \frac{\beta (1+ \widetilde {s^2_t} - 2\theta)} { 1- \theta} \left[\min\{j , s^1_t - a ^{1-}_t\}\right] - K_t a^{1+}_t + J_t a^{1-}_t - L_t {a^2_t}.$$ 
As the right-hand side and the left-hand side of the inequality are the same, we prove our claim.

\item $q_t(k \mid s^2_t , a^2_t)$  is subadditive on $S^2 \times A'$ for every $k \in S^2$. 
It suffices to show that $$q_t(k \mid s^2_t , a^2_t) + q_t(k \mid \widetilde{s^2 _t} , \widetilde{a^2 _t}) \le q_t(k \mid \widetilde{s^2 _t} , a^2_t) + q_t(k \mid s^2_t , \widetilde{a^2 _t}),$$  for every $k \in S^2$ if $s^2_t \ge  \widetilde{s^2 _t}$ and $a^2_t \ge  \widetilde{a^2 _t}$.

From Lemma~\ref{lema3},  $$q_t(k \mid s^2_t, a^2_t) =\left\{
	\begin{array}{ll}
	0 & \text{if} \quad j^2 < k, \\
		p_t(j^2 \mid s^2_t, a^2_t) =1 & \text{if} \quad j^2 \ge k . \\
		
	\end{array}
\right.	
$$

From Lemmas~\ref{lema2} and \ref{lema3}, the starting state  $s^2_t$ transfers to $j^2_A$ or $j^2_C$, if we take action $a^2_t$ or $\widetilde{a^2_t}$, respectively. If the number of batteries replaced is greater than a threshold $\xi
$, then the future state is at least the capacity threshold, $\theta$. Otherwise, we will transfer to the absorbing state, $\eta$. Similarly, the starting state  $\widetilde{s^2_t}$ transfers to $j^2_B$ or $j^2_D$, if we take action $a^2_t$ or $\widetilde{a^2_t}$, respectively. If the number of batteries replaced is greater than a threshold $$\widetilde{	\xi}=\Bigl\lfloor\dfrac{\delta^C (a^{1+}_t + a^{1-}_t) - M(\widetilde{s^2_t} - \theta)}{1-\widetilde{s^2_t} }\Bigr\rfloor
,$$ then the future state is at least $\theta$. Otherwise, we will transfer to the absorbing state, $\eta$.
 \\
Recall in Lemma~\ref{lema4} and Definition \ref{Def1} we have:
$$j^2_A =  \frac{1}{M}\left[s^2_t ( M - a^2_t ) + a^2_t - \delta^C (a^{1+}_t + a^{1-}_t)\right],$$
 $$j^2_B =\frac{1}{M}\left[\widetilde{s^2_t} ( M - a^2_t ) + a^2_t - \delta^C (a^{1+}_t + a^{1-}_t)\right],$$ 
$$j^2_C=  \frac{1}{M}\left[s^2_t ( M - \widetilde{a^2_t} ) + \widetilde{a^2_t} - \delta^C (a^{1+}_t + a^{1-}_t)\right],$$ 
 $$j^2_D = \frac{1}{M}\left[\widetilde{s^2_t }( M - \widetilde{a^2_t }) + \widetilde{a^2_t} - \delta^C (a^{1+}_t + a^{1-}_t)\right],$$ and 
 $$j^2_A \ge \max(j^2_B, j^2_C) \ge \min(j^2_B, j^2_C) \ge j^2_D.$$

To show the subadditivity property of the $q_t(k \mid s^2_t, a^2_t)$ for every $k \in S^2$, we divide $S^2$ space into five intervals. Then, we show that if $ k \in \left(\max(j^2_B, j^2_C), j^2_A\right]$ the subadditivity condition is not satisfied.
 \begin{enumerate}[i.]
\item $ 0 \le k \le  j^2_D$ 
$$q_t(k \mid s^2_t , a^2_t) + q_t(k \mid \widetilde{s^2 _t} , \widetilde{a^2 _t}) \le q_t(k \mid \widetilde{s^2 _t} , a^2_t) + q_t(k \mid s^2_t , \widetilde{a^2 _t}) \Leftrightarrow$$
$$ 1+1 \le 1+1 .$$

\item $ j^2_D < k \le  \min(j^2_B, j^2_C)$ 
$$q_t(k \mid s^2_t , a^2_t) + q_t(k \mid \widetilde{s^2 _t} , \widetilde{a^2 _t}) \le q_t(k \mid \widetilde{s^2 _t} , a^2_t) + q_t(k \mid s^2_t , \widetilde{a^2 _t}) \Leftrightarrow$$
$$ 1+0 \le 1+1 .$$

\item $ \min(j^2_B, j^2_C) < k \le  \max(j^2_B, j^2_C)$ 
 $$q_t(k \mid s^2_t , a^2_t) + q_t(k \mid \widetilde{s^2 _t} , \widetilde{a^2 _t}) \le q_t(k \mid \widetilde{s^2 _t} , a^2_t) + q_t(k \mid s^2_t , \widetilde{a^2 _t}) \Leftrightarrow$$
$$ 1+0 \le 1+0 .$$

\item $ \max(j^2_B, j^2_C) < k \le  j^2_A$ 
 $$q_t(k \mid s^2_t , a^2_t) + q_t(k \mid \widetilde{s^2 _t} , \widetilde{a^2 _t}) \le q_t(k \mid \widetilde{s^2 _t} , a^2_t) + q_t(k \mid s^2_t , \widetilde{a^2 _t}) \Leftrightarrow$$
$$ 1+0 \nleq 0+0 .$$
\item $ j^2_A < k   $ 
 $$q_t(k \mid s^2_t , a^2_t) + q_t(k \mid \widetilde{s^2 _t} , \widetilde{a^2 _t}) \le q_t(k \mid \widetilde{s^2 _t} , a^2_t) + q_t(k \mid s^2_t , \widetilde{a^2 _t}) \Leftrightarrow$$
$$ 0+0 \le 0+0 .$$
\end{enumerate}
Due to the result of part iv, we \textbf{can not} conclude that $q_t(k \mid s^2_t , a^2_t)$ is subadditive on $S^2 \times A'$ for every $k \in S^2$.  

\item $r_N (s_N)$ is non-decreasing in $s^2_N$.\\
Consider  $ s^2_N  \ge \widetilde{s^2 _N}$, it suffices to show that $r_N(s_N) \ge r_N (\widetilde{s_N})$ where: 
\begin{equation}\label{eq:12}
r_N (s_N)=
\left\{
	\begin{array}{ll}
		\rho_{s_N^2} s^1_N  & \text{ if} \ \ \ s^2_N \geq \theta, \\
		0 & \text{ otherwise}.
	\end{array}
\right.	
\end{equation}
We examine the three following cases.

 \begin{enumerate}[i.]
\item  $s^2_N \ge \widetilde {s^2_N} \ge \theta$ 
$$ r_N (s_N) \ge r_N (\widetilde{s_N}) \Leftrightarrow $$
$$ \rho_{s^2_{N}} s^1_N \ge \rho_{\widetilde{s^2_{N}}} s^1_N \Leftrightarrow$$
$$   \frac{\beta (1+ s^2_N - 2\theta)} {1- \theta} s^1_N
 \ge     \frac{\beta (1+ \widetilde {s^2_N} - 2\theta)} {1- \theta} s^1_N \Leftrightarrow$$
$$ s^2_N  \ge \widetilde{s^2 _N}. $$

\item  $s^2_N \ge \theta > \widetilde {s^2_N} $ 
$$ r_N (s_N) \ge r_N (\widetilde{s_N}) \Leftrightarrow $$
$$ \rho_{s^2_{N}} s^1_N \ge 0 \Leftrightarrow$$
$$   \frac{\beta (1+ s^2_N - 2\theta)} {1- \theta} s^1_N
 \ge  0 \Leftrightarrow$$
 $$   \frac{\beta (1+ s^2_N - 2\theta)} {1- \theta} s^1_N
 \ge  0\Leftrightarrow$$
$$ 1+ s^2_N -2 \theta  \ge 0. $$
Because $1+ s^2_N -2\theta  \ge 1+\theta -2\theta = 1- \theta$ and $1-\theta \ge 0$ we can conclude that: $$ 1+ s^2_N -2 \theta  \ge 0. $$

\item  $ \theta > s^2_N \ge \widetilde {s^2_N} $ 
$$ r_N (s_N) \ge r_N (\widetilde{s_N}) \Leftrightarrow $$
$$ 0 \ge 0.$$
\end{enumerate}
\end{enumerate}

We show that we \textbf{can not} prove the subadditivity condition and in turn,  \textcolor{black}{we proved that the stochastic SAIRPs violate the sufficient conditions for the optimality of a monotone policy in the second dimension of the state. }


\end{proof}

\begin{theorem}\label{Thm2} There exist optimal decision rules $d^*_t : S \rightarrow A_{s_t}$ for the stochastic SAIRP which are monotone non-increasing in the second dimension of the \textcolor{black}{state} for \mbox{$t= 1, \dots, N-1$} if there is an upper-bound $U$ on the number of batteries replaced \textcolor{black}{at} each decision epoch where $U=\frac{M\varepsilon}{2(1-s^2_t)}$, when $M$ is the number of batteries at the swap station and $\varepsilon$ is the discretized increment in capacity.  \end{theorem} 

\begin{proof} 
We prove this monotonicity claim by showing the aforementioned five conditions are satisfied \citep{Puterman05}. When we fix the first dimension of the state and action, in Lemma~\ref{Thm1}, we show that conditions i, ii, iii, and v are true.  Thus, it suffices to prove condition iv is true in order to show the existence of monotone optimal decision rules for the second dimension. First, we prove the following claim and use it to show the subadditivity condition and in turn, monotonicity. \\
\begin{claim}
$j^2_A$ and $j^2_C$ represents the same point if $a^2_t - \widetilde{a^2_t} \le U$. 
\end{claim}
\begin{proof}
 We know $j^2_A$ and $j^2_C$ represents the same point if $j^2_A - j^2_C \le \frac{\varepsilon}{2}$ due to the precision in rounding.
Thus,  $$j^2_A - j^2_C \le \frac{\varepsilon}{2} \Leftrightarrow $$
$$ \frac{1}{M}\left[s^2_t ( M - a^2_t ) + a^2_t - \delta^C (a^{1+}_t + a^{1-}_t)\right] - \frac{1}{M}\left[s^2_t ( M - \widetilde{a^2_t} ) + \widetilde{a^2_t} - \delta^C (a^{1+}_t + a^{1-}_t)\right] \le \frac{\varepsilon}{2} \Leftrightarrow$$ 
$$ \frac{1}{M}\left[s^2_t ( M - a^2_t - ( M - \widetilde{a^2_t} ) + a^2_t -  \widetilde{a^2_t}\right] \le \frac{\varepsilon}{2} \Leftrightarrow$$ 
$$ \frac{1}{M}\left[s^2_t (  - a^2_t + \widetilde{a^2_t} ) + a^2_t -  \widetilde{a^2_t}\right] \le \frac{\varepsilon}{2} \Leftrightarrow$$ 
$$ \frac{1}{M}\left[(a^2_t - \widetilde{a^2_t})( 1- s^2_t )\right]  \le \frac{\varepsilon}{2} \Leftrightarrow$$ 
$$ (a^2_t - \widetilde{a^2_t}) \le \frac{M\varepsilon}{2( 1- s^2_t )}.$$
Let  $\widetilde{a^2_t} =0 $ to get the least upper-bound for the number of batteries to be replaced, we will have 
$$ a^2_t  \le \frac{M\varepsilon}{2( 1- s^2_t )}.$$ 
\end{proof}

\noindent With the case that $j^2_A = j^2_C$, we have $$j^2_A = j^2_C \ge j^2_B \ge j^2_D.$$
Now, we revisit the condition, we seek to show that $q_t(k \mid s^2_t , a^2_t)$  is subadditive on $S^2 \times A'$ for every $k \in S^2$. \\
It suffices to show that $$q_t(k \mid s^2_t , a^2_t) + q_t(k \mid \widetilde{s^2 _t} , \widetilde{a^2 _t}) \le q_t(k \mid \widetilde{s^2 _t} , a^2_t) + q_t(k \mid s^2_t , \widetilde{a^2 _t}),$$  for every $k \in S^2$ if $s^2_t \ge  \widetilde{s^2 _t}$ and $a^2_t \ge  \widetilde{a^2 _t}$.\\
We divide $S^2$ space into four intervals and show the condition is satisfied for every interval.
 \begin{enumerate}[i.]
\item $ 0 \le k \le  j^2_D$ 
$$q_t(k \mid s^2_t , a^2_t) + q_t(k \mid \widetilde{s^2 _t} , \widetilde{a^2 _t}) \le q_t(k \mid \widetilde{s^2 _t} , a^2_t) + q_t(k \mid s^2_t , \widetilde{a^2 _t}) \Leftrightarrow$$
$$ 1+1 \le 1+1.$$

\item $ j^2_D < k \le  j^2_B$ 
$$q_t(k \mid s^2_t , a^2_t) + q_t(k \mid \widetilde{s^2 _t} , \widetilde{a^2 _t}) \le q_t(k \mid \widetilde{s^2 _t} , a^2_t) + q_t(k \mid s^2_t , \widetilde{a^2 _t}) \Leftrightarrow$$
$$ 1+0 \le 1+1. $$

\item $ j^2_B < k \le   j^2_C = j^2_A$ 
 $$q_t(k \mid s^2_t , a^2_t) + q_t(k \mid \widetilde{s^2 _t} , \widetilde{a^2 _t}) \le q_t(k \mid \widetilde{s^2 _t} , a^2_t) + q_t(k \mid s^2_t , \widetilde{a^2 _t}) \Leftrightarrow$$
$$ 1+0 \le 1+0. $$

\item $ j^2_A < k   $ 
 $$q_t(k \mid s^2_t , a^2_t) + q_t(k \mid \widetilde{s^2 _t} , \widetilde{a^2 _t}) \le q_t(k \mid \widetilde{s^2 _t} , a^2_t) + q_t(k \mid s^2_t , \widetilde{a^2 _t}) \Leftrightarrow$$
$$ 0+0 \le 0+0. $$
\end{enumerate}

We conclude that $q_t(k \mid s^2_t , a^2_t)$ is subadditive on $S^2 \times A'$ for every $k \in S^2$. As all conditions are valid, we deduce that there exists monotone optimal decision rules in the second dimension of the \textcolor{black}{state} for the stochastic SAIRP when there is an upper-bound $U$ on the number of batteries replaced \textcolor{black}{at} each decision epoch.

\end{proof}

\clearpage
\subsection{Monotonicity of Value Functions}\label{APP2}



In this Appendix, we prove that the MDP value function for the stochastic SAIRP is monotone non-decreasing in the first, second, and both dimensions of the state. In Theorems~\ref{Thm3} and \ref{Thm4}, we show the monotonicity of the value function regarding $s^2_t$ and $s^1_t$, respectively, and in Theorem~\ref{Thm5} we prove that the value function is monotone in $(s^1_t, s^2_t$). To prove these theorems, we need to ensure that four conditions are satisfied as given by Papadakia and Powell \cite{Papadaki07} and Jiang and Powell \cite{Jiang15}. These two articles use different notation, thus, for clarity, we define $S$ as the state space, $A$ as the action space, and the transition function as $f : S \times A \rightarrow S$.  To prove the theorems, we first state key definitions used.

\begin{definition}{}\label{Def2}
Partial ordering operator $\preceq$ on the $N$-dimensional  set $S$ is defined as $s  \preceq s' $ for any $s, s' \in S$, if $s(i) \leq s'(i)$ for all $i \in \{1, 2, \dots N\}$ \citep{Papadaki07}.
\end{definition}

\begin{definition}{}\label{Def3}
An $N$-dimensional  real-valued function $F$ is partially non-decreasing on the set $S$, if for all $ s^-, s^+ \in S$ where $ s^- \preceq s^+$, we have $F(s^-) \le F(s^+)$ \citep{Papadaki07}.
\end{definition}

\begin{theorem}\label{Thm3} The MDP value function of the stochastic SAIRP is monotonically non-decreasing in $s^2_t$.      
\end{theorem} 
\begin{proof} 
Fixing the first dimension of the state space, the MDP value function is monotone if the following four conditions are satisfied \citep{Papadaki07, Jiang15}. 
\begin{enumerate}

\item For $e \succeq 0$ we have $f(s+e, a) \succeq f(s,a)$ for all $a \in A$ \citep{Papadaki07}. Equivalently, if every $s^2_t, \widetilde{s^2_t} \in S^2$ with $s^2_t \ge \widetilde{s^2_t}$, action $a_t=(a^1_t,a^2_t) \in A$ is taken, and demand equals $D_t$, the second state transition function $f^2 $ satisfies:
$$f^2 (s^1_t, s^2_t, a^1_t , a^2_t) \ge f^2  (s^1_t, \widetilde{s^2_t}, a^1_t , a^2_t).$$
From Equation \eqref{Added1}, if the beginning state is $s^2_t$, then $f^2 $ equals $j^2$, and if the beginning state is $\widetilde{s^2_t}$, then $f^2$ equals $\widetilde{j^2}$.  Because $a^2_t, a^1_t, s^1_t$ are constant, using Lemma~\ref{lema1} we prove our claim as
$$f^2 (s^1_t, s^2_t, a^1_t , a^2_t) = j^2 \ge  (s^1_t, \widetilde{s^2_t}, a^1_t , a^2_t) = \widetilde{j^2}.$$
\item The one period cost function $r_t (s, a)$ is partially non-decreasing in $ s \in S$ for all $a \in A$, $t = 0, 1, \dots, N-1$ \citep{Papadaki07}.  It is suffices to show for every $t < N$, $s^2_t, \widetilde{s^2_t} \in S^2$ with $s^2_t \ge \widetilde{s^2_t}$, if we take action $a_t=(a^1_t, a^2_t) \in A$, then the one period reward function satisfies
\begin{equation}\label{eq:MV-SecondDimsion}
r_t(s^1_t, s^2_t, a^1_t , a^2_t, D_t) \ge r_t(s^1_t, \widetilde{s^2_t}, a^1_t , a^2_t ,D_t). 
\end{equation}
We show that we can reduce Equation \eqref{eq:MV-SecondDimsion} to an always true statement.
\begin{center}
$r_t (s^1_t, s^2_t, a^1_t , a^2_t, D_t) \ge r_t(s^1_t, \widetilde{s^2_t}, a^1_t , a^2_t ,D_t )  \Leftrightarrow $ \\
$ \frac{\beta (1+ s^2_t - 2\theta)} { 1- \theta} \left[\min\{D_t, s^1_t - a ^{1-}_t\}\right] - K_t a^{1+}_t + J_t a^{1-}_t - L_t a^2_t \ge \frac{\beta (1+ \widetilde{s^2_t} - 2\theta)} { 1- \theta} \left[\min\{D_t , s^1_t - a ^{1-}_t\}\right] - K_t a^{1+}_t + J_t a^{1-}_t - L_t a^2_t  \Leftrightarrow $ \\
$\frac{\beta (1+ s^2_t - 2\theta)} { 1- \theta} \ge \frac{\beta (1+ \widetilde{s^2_t} - 2\theta)}{ 1- \theta}   \Leftrightarrow $ \\
$ s^2_t  \ge \widetilde{s^2 _t}. $
\end{center}
\item  The terminal cost function $r_N (s)$ is partially non-decreasing in $s \in S$ \citep{Papadaki07}. It suffices to show that for every $s^2_N, \widetilde{s^2_N} \in S^2$ with $s^2_N \ge \widetilde{s^2_N}$, that $r_N(s_N) \ge r_N (\widetilde{s_N}).$\\
We proved this claim in condition (5) of Lemma~\ref{Thm1} and Theorem~\ref{Thm2}. \\
\item For each $t <N$, $s^2_t$ and $D_{t+1}$ are independent \citep{Jiang15}. 

In our model, demand is a random variable and does not depend on the current state or action, including $s^2_t$.  \\ 

\end{enumerate}
As all the conditions are valid for the stochastic SAIRP, we can conclude that the value function is monotone in $s^2_t$. 

\end{proof}

\begin{theorem}\label{Thm4} The MDP value function of the stochastic SAIRP is monotonically non-decreasing in $s^1_t$.      
\end{theorem} 
\begin{proof} 
Fixing the second dimension of the state space, the MDP value function is monotone if the following four conditions are satisfied \citep{Papadaki07, Jiang15}.
\begin{enumerate}

\item For $e \succeq 0$ we have $f (s+e, a) \succeq f (s,a)$ for all $a \in A$ \citep{Papadaki07}.  It suffices to show for every $s^1_t, \widetilde{s^1_t} \in S^1$ with $s^1_t \ge \widetilde{s^1_t}$, if we take action $a_t=(a^1_t,a^2_t) \in A$ and demand equals $D_t$, the first state transition function $f^1$ satisfies
$$f^1 (s^1_t, s^2_t, a^1_t , a^2_t, D_t) \ge f^1  (\widetilde{s^1_t}, s^2_t, a^1_t , a^2_t ,D_t).$$
Using Equation \eqref{eq:3}, if the beginning state is $s^1_t$, then $f^1$ equals $j^1$, and if the starting state is $\widetilde{s^1_t}$, then $f^1$ equals $\widetilde{j^1}$.  It suffices to show that
$$j^1 \ge \widetilde{j^1} \Leftrightarrow$$
$$s^1_{t} + a^2_{t}+a^{1+}_{t}- a^{1-}_{t}-\min\{ D_t, s^1_t-a^{1-}_t \} \ge \widetilde{s^1_{t}} + a^2_{t}+a^{1+}_{t}- a^{1-}_{t}-\min\{ D_t, \widetilde{s^1_t}-a^{1-}_t \}  \Leftrightarrow$$
\begin{equation} \label{eq:14}
s^1_{t} -\min\{ D_t, s^1_t-a^{1-}_t \} \ge \widetilde{s^1_{t}} -\min\{ D_t, \widetilde{s^1_t}-a^{1-}_t \}.  
\end{equation}
Three cases can happen for the stochastic demand $D_t$, and we show that  we can reduce Equation \eqref{eq:14} to an always true statement in all cases.
 \begin{enumerate}[i.]
\item $s^1_t -a^{1-}_t \ge \widetilde{s^1_t}-a^{1-}_t  > D_t$ 
$$s^1_{t} - D_t  \ge \widetilde{s^1_{t}} - D_t  \Leftrightarrow $$
$$s^1_{t}   \ge \widetilde{s^1_{t}}. $$

\item $s^1_t -a^{1-}_t \ge D_t \ge \widetilde{s^1_t}-a^{1-}_t $ 
$$s^1_{t} - D_t  \ge \widetilde{s^1_{t}} - (\widetilde{s^1_t}-a^{1-}_t)  \Leftrightarrow $$
$$s^1_t -a^{1-}_t \ge D_t.$$

\item $  D_t > s^1_t -a^{1-}_t \ge \widetilde{s^1_t}-a^{1-}_t $ 
$$s^1_{t} - (s^1_t -a^{1-}_t)  \ge \widetilde{s^1_{t}} - (\widetilde{s^1_t}-a^{1-}_t)  \Leftrightarrow $$
$$a^{1-}_t \ge a^{1-}_t.$$
\end{enumerate}
From i, ii, and iii, we conclude that $j^1 \ge \widetilde{j^1}$. Thus,
$$f^1 (s^1_t, s^2_t, a^1_t , a^2_t, D_t) \ge f^1  (\widetilde{s^1_t}, s^2_t, a^1_t , a^2_t ,D_t).$$

\item The one period cost function $r_t (s, a)$ is partially non-decreasing in $ s \in S$ for all $a \in A$, $t = 0, 1, \dots, N-1$ \citep{Papadaki07}. It suffices to show for every $t < N$, $s^1_t, \widetilde{s^1_t} \in S^1$ with $s^1_t \ge \widetilde{s^1_t}$, if we take action $a_t=(a^1_t,a^2_t) \in A$, then the one period reward function satisfies
$$r_t(s^1_t, s^2_t, a^1_t , a^2_t, D_t) \ge r_t (\widetilde{s^1_t}, s^2_t, a^1_t , a^2_t ,D_t).$$
Using Equation \eqref{eq:imRew}, we reduce the following statement to an always true statement. 
\begin{center}
$r_t(s^1_t, s^2_t, a^1_t , a^2_t, D_t) \ge r_t (\widetilde{s^1_t}, s^2_t, a^1_t , a^2_t ,D_t) 
\Leftrightarrow $
$\rho_{s^2_{t}} (\min\{ D_t, s^1_t -a^{1-}_t \}) - K_t a^{1+}_t + J_t a^{1-}_t - L_t a^2_t \ge \rho_{s^2_{t}} (\min\{ D_t, \widetilde{s^1_t}-a^{1-}_t \}) - K_t a^{1+}_t + J_t a^{1-}_t - L_t a^2_t \Leftrightarrow$\\
\end{center}
\begin{equation}\label{eq:15}
 \min\{ D_t, s^1_t-a^{1-}_t \} \ge  \min\{ D_t, \widetilde{s^1_t}-a^{1-}_t \}.
\end{equation}
The second dimension is fixed, so we could cancel out $\rho_{s^2_{t}}$ from both sides. \\
Similar to part 1, three cases can happen.
 \begin{enumerate}[i.]
\item $s^1_t -a^{1-}_t \ge \widetilde{s^1_t}-a^{1-}_t  > D_t$ \\
We can reduce Equation \eqref{eq:15} to $D_t   \ge D_t $ which is always true.
\item $s^1_t -a^{1-}_t \ge D_t \ge \widetilde{s^1_t}-a^{1-}_t $ \\
We can reduce Equation \eqref{eq:15} to $D_t   \ge  \widetilde{s^1_t}-a^{1-}_t$ which is true for this case.
\item $  D_t > s^1_t -a^{1-}_t \ge \widetilde{s^1_t}-a^{1-}_t $ \\
We can reduce Equation \eqref{eq:15} to $s^1_t   \ge  \widetilde{s^1_t}$  which is always true.
\end{enumerate}
So, we can conclude that
$$r_t(s^1_t, s^2_t, a^1_t , a^2_t, D_t) \ge r_t (\widetilde{s^1_t}, s^2_t, a^1_t , a^2_t ,D_t).$$

\item  The terminal cost function $r_N (s)$ is partially non-decreasing in $s \in S$ \citep{Papadaki07}. It suffices to show that for every $s^1_N, \widetilde{s^1_N} \in S^1$ with $s^1_N \ge \widetilde{s^1_N}$, that $r_N(s_N) \ge r_N (\widetilde{s_N}).$

As the second state is fixed equal to $s^2_N$, two possible cases can happen.

 \begin{enumerate}[i.]
\item $s^2_N < \theta$ 
$$r_N(s_N)  \ge r_N (\widetilde{s_N}) \Leftrightarrow $$
 $$0 \ge 0 .$$ 
\item $s^2_N \ge \theta$. 
$$r_N(s_N)  \ge r_N (\widetilde{s_N}) \Leftrightarrow $$
 $$\rho_{s^2_{N}} s^1_N \ge \rho_{s^2_{N}} \widetilde{s^1_N}.$$
$\rho_{s^2_{N}}$ is not a function of the first dimension, so we can cancel it out from both sides.
$$ s^1_N \ge \widetilde{s^1_N}.$$
\end{enumerate}
Thus, we conclude that
$$r_N(s_N)  \ge r_N (\widetilde{s_N}).$$
\item For each $t <N$, $s^1_t$ and $D_{t+1}$ are independent \citep{Jiang15}. \\
In our model, demand is a random variable and does not depend on the current state or action, including $s^1_t$.  
\end{enumerate}
As all the conditions are valid for the stochastic SAIRP, we conclude that the value function is monotone in $s^1_t$. 

\end{proof}


If there exists a monotone optimal policy, then Theorem~\ref{Thm4} is implied.  Widrick et al. \cite{Widrick16} proved that there exists a monotone optimal policy in $s^1_t$ \emph{only} when demand is governed by a non-increasing discrete distribution.  Thus, Theorem~\ref{Thm4} provides a stronger result, as it does not depend on the probability distribution of demand.

\begin{theorem}\label{Thm5} The MDP value function of the stochastic SAIRP is monotonically non-decreasing in $(s^1_t, s^2_t)$.      
\end{theorem} 
\begin{proof} 
The  MDP value function is monotone if the following four conditions are satisfied \citep{Papadaki07, Jiang15}.
\begin{enumerate}
\item For $e \succeq 0$ we have $f (s+e, a) \succeq f (s,a)$ for all $a \in A$ \citep{Papadaki07}. Using Definition \ref{Def2} for partially ordered functions, it suffices to show that for every $s_t, \widetilde{s_t} \in (S^1 \times S^2)$ with $s^1_t \ge \widetilde{s^1_t}$ and $s^2_t \ge \widetilde{s^2_t}$, if we take action $a_t=(a^1_t,a^2_t) \in A$ and demand equals $D_t$, then the $i^{th}$ state transition function $f^i $  satisfies
$$f^i (s_t, a _t,D_t) \ge f^i  (\widetilde{s_t}, a_t ,D_t) \quad \forall i =1,2. $$ 
We show the relationship between transition functions for each dimension.

 \begin{enumerate}[i.]
 \item $f^1 (s_t, a_t ,D_t) \ge f^1 (\widetilde{s_t}, a_t ,D_t)$\\
 Using Equation \eqref{eq:3}, we can state that if the beginning state is $s^1_t$, then $f^1 $ equals $j^1$ and  if the starting state is $\widetilde{s^1_t}$, then $f^1 $ equals $\widetilde{j^1}$. We proved the claim when $s^2_t \ge \widetilde{s^2_t} \ge \theta$, in part 1 of Theorem~\ref{Thm4}. 
If $ \theta > s^2_t \ge \widetilde{s^2_t}$, as we are in the absorbing state, the only possible action is $a_t=(0,0)$ that leads to $j^1 = s^1_t \ge\widetilde{ j^1}  = \widetilde{s^1_t}$. Similarly, if $ s^2_t \ge \theta > \widetilde{s^2_t}$, the only allowed action is $a_t=(0,0)$  
because it is the only feasible action for both $s^2_t$ and $\widetilde{s^2_t}$. Thus, $j^1 = s^1_t \ge\widetilde{ j^1}  = \widetilde{s^1_t}$.

\item $f^2 (s_t, a_t ,D_t) \ge f^2  (\widetilde{s_t}, a_t ,D_t)$\\
Using Equation \eqref{Added1}, we reduce $f^2 (s_t, a_t ,D_t) \ge f^2  (\widetilde{s_t}, a_t ,D_t)$ to an always true statement:
$$f^2 (s_t, a_t ,D_t) \ge f^2  (\widetilde{s_t}, a_t ,D_t) \Leftrightarrow$$
$$ \frac{1}{M}\left[s^2_t ( M - a^2_t ) + a^2_t - \delta^C (a^{1+}_t + a^{1-}_t)\right] \ge  \frac{1}{M}\left[\widetilde{s^2_t} ( M - a^2_t ) + a^2_t - \delta^C (a^{1+}_t + a^{1-}_t)\right]  \Leftrightarrow$$
  $$ s^2_t ( M - a^2_t ) + a^2_t - \delta^C (a^{1+}_t + a^{1-}_t) \ge  \widetilde{s^2_t} ( M - a^2_t ) + a^2_t - \delta^C (a^{1+}_t + a^{1-}_t)  \Leftrightarrow$$
    $$ s^2_t ( M - a^2_t )  \ge  \widetilde{s^2_t} ( M - a^2_t ).$$
Because $ M \ge \widetilde {a^2_t }$, we know
 $$ s^2_t   \ge  \widetilde{s^2_t}.$$
The last statement is always true and we prove that $f^2 (s_t, a_t ,D_t) \ge f^2  (\widetilde{s_t}, a_t ,D_t)$.
\end{enumerate}

\item The one period cost function $r_t (s, a)$ is partially non-decreasing in $ s \in S$ for all $a \in A$, $t = 0, 1, \dots, N-1$ \citep{Papadaki07}. Using Definition \ref{Def2} for partially ordered functions, it suffices to show that for every $s_t, \widetilde{s_t}\in (S^1 \times S^2)$ with $s^1_t \ge \widetilde{s^1_t}$ and $s^2_t \ge \widetilde{s^2_t}$, if we take action $a_t=(a^1_t, a^2_t) \in A$, then the one period reward function satisfies
$$r_t(s_t, a_t, D_t) \ge r_t(\widetilde{s_t}, a_t, D_t).$$ 
Using Equation \eqref{eq:imRew}, we reduce the following statement to an always true statement. 
$$r_t(s_t, a_t, D_t) \ge r_t(\widetilde{s_t}, a_t, D_t) \Leftrightarrow$$
$$\rho_{s^2_{t}} (\min\{ D_t, s^1_t -a^{1-}_t \}) - K_t a^{1+}_t + J_t a^{1-}_t - L_t a^2_t \ge \rho_{\widetilde{{s^2_{t}}}} (\min\{ D_t, \widetilde{s^1_t}-a^{1-}_t \}) - K_t a^{1+}_t + J_t a^{1-}_t - L_t a^2_t \Leftrightarrow$$
$$\bigg(\frac{\beta(1+s^2_t - 2 \theta)}{1- \theta}\bigg) (\min\{ D_t, s^1_t -a^{1-}_t \})  \ge \bigg(\frac{\beta(1+\widetilde{s^2_t} - 2 \theta)}{1- \theta} \bigg)(\min\{ D_t, \widetilde{s^1_t}-a^{1-}_t \})  \Leftrightarrow$$
\begin{equation}\label{eq:17}
(1+s^2_t - 2 \theta) (\min\{ D_t, s^1_t-a^{1-}_t \}) \ge (1+\widetilde{s^2_t}- 2 \theta) (\min\{ D_t, \widetilde{s^1_t}-a^{1-}_t \}).
\end{equation}

Three cases can happen for the stochastic demand $D_t$. We show that we can reduce Equation \eqref{eq:17} to an always true statement in all cases.
 \begin{enumerate}[i.]

\item $  D_t > s^1_t -a^{1-}_t \ge \widetilde{s^1_t}-a^{1-}_t $ \\
Because $(1+s^2_t - 2 \theta) \ge (1+\widetilde{s^2_t} - 2 \theta) \ge 0$ and $\min\{ D_t, s^1_t-a^{1-}_t \} \ge \min\{ D_t, \widetilde{s^1_t}-a^{1-}_t \} \ge 0$, we can conclude that Equation \eqref{eq:17} is always true.

\item $s^1_t -a^{1-}_t \ge D_t \ge \widetilde{s^1_t}-a^{1-}_t $ \\
We know $(1+s^2_t - 2 \theta) \ge (1+\widetilde{s^2_t} - 2 \theta)$, thus it suffices to show that

$$(1+s^2_t - 2 \theta) (\min\{ D_t, s^1_t-a^{1-}_t \}) \ge (1+s^2_t- 2 \theta) (\min\{ D_t, \widetilde{s^1_t}-a^{1-}_t \}) \Leftrightarrow$$
$$\min\{ D_t, s^1_t-a^{1-}_t \} \ge \min\{ D_t, \widetilde{s^1_t}-a^{1-}_t \} \Leftrightarrow$$
$$D_t \ge  \widetilde{s^1_t}-a^{1-}_t.$$
The last statement is always true for this case.  

\item $s^1_t -a^{1-}_t \ge \widetilde{s^1_t}-a^{1-}_t  > D_t$ \\
Using the same approach as the previous part, we can reduce Equation \eqref{eq:17} to $D_t   \ge D_t $ which is always true. 
$$(1+s^2_t - 2 \theta) (\min\{ D_t, s^1_t-a^{1-}_t \}) \ge (1+s^2_t- 2 \theta) (\min\{ D_t, \widetilde{s^1_t}-a^{1-}_t \}) \Leftrightarrow$$
$$\min\{ D_t, s^1_t-a^{1-}_t \} \ge \min\{ D_t, \widetilde{s^1_t}-a^{1-}_t \} \Leftrightarrow$$
$$D_t \ge D_t.$$

\end{enumerate}
So, we conclude that
$$r_t(s_t, a_t, D_t) \ge r_t(\widetilde{s_t}, a_t, D_t).$$

\item  The terminal cost function $r_N (s)$ is partially non-decreasing in $s \in S$ \citep{Papadaki07}. It suffices to show that for every $s_N, \widetilde{s_N} \in (S^1 \times S^2) $ with $s^1_N \ge \widetilde{s^1_N}$ and $s^2_N \ge \widetilde{s^2_N}$, that $r_N(s_N) \ge r_N (\widetilde{s_N}).$
Using Equation \eqref{eq:12}, three cases can happen. 
 \begin{enumerate}[i.]

\item $s^2_N \ge \widetilde{s^2_N} \ge \theta$. 
$$r_N(s_N)  \ge r_N (\widetilde{s_N}) \Leftrightarrow $$
 $$\rho_{s^2_{N}} s^1_N \ge \rho_{\widetilde{s^2_{N}}} \widetilde{s^1_N}.$$
Because $\frac{\beta(1+s^2_t - 2 \theta)}{1-\theta} \ge \frac{\beta(1+\widetilde{s^2_t} - 2 \theta)}{1-\theta}$ and $ s^1_N \ge \widetilde{s^1_N}$ the last statement is true.

\item $s^2_N \ge \theta > \widetilde{s^2_N}$ 
$$r_N(s_N)  \ge r_N (\widetilde{s_N}) \Leftrightarrow $$
 $$\rho_{s^2_{N}} s^1_N \ge 0 .$$ 
 
\item $\theta > s^2_N \ge \widetilde{s^2_N} $  
 $$r_N(s_N)  \ge r_N (\widetilde{s_N}) \Leftrightarrow $$
 $$ 0 \ge 0 .$$
 
\end{enumerate}
\item For each $t <N$, $(s^1_t, s^2_t) $ and $D_{t+1}$ are independent \citep{Jiang15}.  \\
In our model, demand is a random variable and does not depend on the current state or action, including $(s^1_t, s^2_t)$.  
\end{enumerate}
As all the conditions are valid for the stochastic SAIRP, we conclude that the value function is monotone in $(s^1_t, s^2_t)$. 

\end{proof}

\clearpage

\subsection{Algorithmic and Experimental Parameter Settings}\label{APP3}

\begin{table}[!ht]\centering
\ra{1}
\resizebox{0.9\textwidth}{!}{
\begin{tabular}{@{}cl l@{}}\toprule

	Parameter 	&	Value 		& Description \\ \midrule
$\overline{M}$ & 2	&		The starting number of batteries used for the small SAIRPs solved using BI \\
 $M$ & 7 & The number of batteries in the modest-sized SAIRPs \\
 $M$ & 100 & The number of batteries in the realistic-sized SAIRPs \\
  $\overline{T}$ & 168 & The time horizon (number of hours)  in the small SAIRP \\
  $T$ & 168 & The time horizon (number of hours) in the modest-sized SAIRPs \\
 $T$ & 744& The time horizon (number of hours) in the realistic-sized SAIRPs \\
$\theta$ & 80\% & The replacement threshold \\
$\varepsilon$ & 0.001 & The capacity increment used in discretizing the $2^{nd}$ dimension of the state  \\
 $maxIteration$+1 & 	3 &		The maximum number of regression-based initialization iterations \\  
$\tau$ & 500000&			The maximum number of core ADP iterations in the modest-sized SAIRPs \\
$\tau$ & 100000&			The maximum number of core ADP iterations in the realistic-sized SAIRPs \\
$w$ & 25000 & The harmonic stepsize parameter \\
$\mu_1$ & 600 & The STC stepsize parameter \\
$\mu_2$ & 1000& The STC stepsize parameter \\
$\zeta$ & 0.7& The STC stepsize parameter \\
$\alpha_0$ & 1 & The STC stepsize parameter \\ \bottomrule
\end{tabular}}
\label{Parameters}
\end{table}

\clearpage

\bibliographystyle{elsarticle-num}

\bibliography{TRE} 

\begin{thebibliography}{10}
\expandafter\ifx\csname url\endcsname\relax
  \def\url#1{\texttt{#1}}\fi
\expandafter\ifx\csname urlprefix\endcsname\relax\def\urlprefix{URL }\fi
\expandafter\ifx\csname href\endcsname\relax
  \def\href#1#2{#2} \def\path#1{#1}\fi

\bibitem{EVEverywhere}
{United States Department of Energy}, {EV} everywhere grand challenge road to
  success, last accessed on April 1, 2021 at
  \url{https://www.energy.gov/sites/prod/files/2014/02/f8/eveverywhere_road_to_success.pdf}
  (January 2014).

\bibitem{RJB15}
Z.~Rezvani, J.~Jansson, J.~Bodin, Advances in consumer electric vehicle
  adoption research: {A} review and research agenda, Transportation Research
  Part D: Transport and Environment 34 (2015) 122--136.

\bibitem{Saxena2015}
S.~Saxena, C.~L. Floch, J.~MacDonald, S.~Moura, Quantifying {EV} battery
  end-of-life through analysis of travel needs with vehicle powertrain models,
  Journal of Power Sources 282 (2015) 265--276.

\bibitem{AmazonDrones}
{CBS NEWS}, Amazon unveils futuristic plan: Delivery by drone, last accessed on
  April 1, 2021 at
  \url{http://www.cbsnews.com/news/amazon-unveils-futuristic-plan-delivery-by-drone/}
  (December 2013).

\bibitem{UPSDrones}
E.~Weise, {UPS} tested launching a drone from a truck for deliveries, USA
  Today, last accessed on April 1, 2021 at
  \url{https://www.usatoday.com/story/tech/news/2017/02/21/ups-delivery-top-of-van-drone-workhorse/98057076/}
  (February 2017).

\bibitem{DHL}
DHL, Successful trial integration of {DHL} parcelcopter into logistics chain,
  last accessed on April 1, 2021 at
  \url{http://www.dhl.com/en/press/releases/releases_2016/all/parcel_ecommerce/successful_trial_integration_dhl_parcelcopter_logistics_chain.html}
  (May 2016).

\bibitem{RoyalMail}
T.~Wallace, {Royal Mail} wants to use drones and driverless trucks, last
  accessed on April 1, 2021 at The Telegraph:
  \url{http://www.telegraph.co.uk/technology/11984099/Royal-Mail-wants-to-use-drones-and-driverless-trucks.html}
  (November 2015).

\bibitem{DubaiDrone}
B.~Mutzabaugh, Drone taxis? {D}ubai plans roll out of self-flying pods, last
  accessed on April 1, 2021 at
  \url{https://www.usatoday.com/story/travel/flights/todayinthesky/2017/02/13/dubai-passenger-carrying-drones-could-flying-july/97850596/}
  (February 2017).

\bibitem{AgDrones2}
J.~Jensen, Agricultural drones: How drones are revolutionizing agriculture and
  how to break into this booming market, {UAV} Coach, last accessed on April 1,
  2021 at: \url{https://uavcoach.com/agricultural-drones/} (April 2019).

\bibitem{Park17}
S.~Park, L.~Zhang, S.~Chakraborty, Battery assignment and scheduling for drone
  delivery businesses, in: {IEEE/ACM} International Symposium on Low Power
  Electronics and Design, Taipei, Taiwan, 2017, pp. 1--6.

\bibitem{DroneTime}
D.~James, 14 drones with the best flight times, last accessed on April 1, 2021
  at \url{http://www.dronesglobe.com/guide/long-flight-time/} (March 2020).

\bibitem{ChargeTime}
{Drones Etc.}, Phantom 3 batteries -- {Tips and Tutorial}, last accessed on
  April 1, 2021 at
  \url{https://www.dronesetc.com/blogs/news/30238721-phantom-3-batteries-tips-and-tutorial}
  (June 2015).

\bibitem{Lacey13}
G.~Lacey, T.~Jiang, G.~Putrus, R.~Kotter, The effect of cycling on the state of
  health of the electric vehicle battery, in: 48th International Universities'
  Power Engineering Conference, Dublin, Ireland, 2013, pp. 1--7.

\bibitem{Shirk15}
M.~Shirk, J.~Wishart, Effects of electric vehicle fast charging on battery life
  and vehicle performance, in: SAE Technical Paper, SAE 2015 World Congress
  {\&} Exhibition, SAE International, Detroit, MI, 2015, pp. 1--13.

\bibitem{Dunn2011}
B.~Dunn, H.~Kamath, J.-M. Tarascon, Electrical energy storage for the grid: A
  battery of choices, Science 334~(6058) (2011) 928--935.

\bibitem{Nuvve2}
Nuvve, Toyota {T}susho and {C}hubu {E}lectric {P}ower announce {J}apan's first
  ever {V2G} project using {N}uvve's technology, last accessed on March 24,
  2021 at
  \url{https://nuvve.com/toyota-tsusho-and-chubu-electric-power-announce-japans-first-ever-v2g-project-using-nuvves-technology/}
  (March 2021).

\bibitem{Nuvve}
Nuvve, {Vehicle-To-Grid} technology, last accessed on March 24, 2021 at
  \url{https://nuvve.com/technology/} (March 2021).

\bibitem{Zhang2020}
P.~Zhang, {NIO} signs agreement with state grid subsidiary to build 100
  charging and battery swap stations by 2021, last accessed on March 23, 2021
  at
  \url{https://cntechpost.com/2020/12/14/nio-national-grid-subsidiary-to-build-100-charging-battery-swap-stations-by-2021/}
  (December 2020).

\bibitem{UKv2g}
{Inside EVs}, {E.ON} and {Nissan} are launching new {V2G} trial project in the
  {UK}, last accessed on March 24, 2021 at
  \url{https://insideevs.com/news/437785/eon-nissan-v2g-trial-project-uk/}
  (August 2020).

\bibitem{mak2013infrastructure}
H.-Y. Mak, Y.~Rong, Z.-J.~M. Shen, Infrastructure planning for electric
  vehicles with battery swapping, Management Science 59~(7) (2013) 1557--1575.

\bibitem{Puterman05}
M.~L. Puterman, Markov decision processes: Discrete stochastic dynamic
  programming, 1st Edition, John Wiley \& Sons, Hoboken, New Jersey, 2005.

\bibitem{Asadi19}
A.~Asadi, S.~{{Nurre Pinkley}}, A stochastic scheduling, allocation, and
  inventory replenishment problem for battery swap stations, Transportation
  Research Part E: Logistics and Transportation Review 146 (2021) 102212.

\bibitem{Jiang15}
D.~R. Jiang, W.~B. Powell, An approximate dynamic programming algorithm for
  monotone value functions, Operations Research 63~(6) (2015) 1489--1511.

\bibitem{Widrick16}
R.~S. Widrick, S.~G. Nurre, M.~J. Robbins, Optimal policies for the management
  of an electric vehicle battery swap station, Transportation Science 52~(1)
  (2018) 59--79.

\bibitem{Nurre14}
S.~G. Nurre, R.~Bent, F.~Pan, T.~C. Sharkey, Managing operations of plug-in
  hybrid electric vehicle {(PHEV)} exchange stations for use with a smart grid,
  Energy Policy 67 (2014) 364--377.

\bibitem{worley2011optimization}
O.~Worley, D.~Klabjan, Optimization of battery charging and purchasing at
  {E}lectric {V}ehicle battery swap stations, in: IEEE Vehicle Power and
  Propulsion Conference, Chicago, IL, 2011, pp. 1--4.

\bibitem{Sun14}
B.~{Sun}, X.~{Tan}, D.~H.~K. {Tsang}, Optimal charging operation of battery
  swapping stations with {QoS} guarantee, in: IEEE International Conference on
  Smart Grid Communications (SmartGridComm), Venice, Italy, 2014, pp. 13--18.

\bibitem{Schneider18}
F.~Schneider, U.~W. Thonemann, D.~Klabjan, Optimization of battery charging and
  purchasing at electric vehicle battery swap stations, Transportation Science
  52~(5) (2018) 1211--1234.

\bibitem{Kang16}
Q.~Kang, J.~Wang, M.~Zhou, A.~C. Ammari, Centralized charging strategy and
  scheduling algorithm for electric vehicles under a battery swapping scenario,
  IEEE Transactions on Intelligent Transportation Systems 17~(3) (2016)
  659--669.

\bibitem{Zhang14}
L.~Zhang, S.~Lou, Y.~Wu, L.~Yi, B.~Hu, Optimal scheduling of electric vehicle
  battery swap station based on time-of-use pricing, in: IEEE PES Asia-Pacific
  Power and Energy Engineering Conference, Hong Kong, 2014, pp. 1--6.

\bibitem{Shen19}
Z.-J.~M. Shen, B.~Feng, C.~Mao, L.~Ran, Optimization models for electric
  vehicle service operations: A literature review, Transportation Research Part
  B: Methodological 128 (2019) 462--477.

\bibitem{Kwizera2018}
O.~Kwizera, S.~G. Nurre, Using drones for delivery: A two-level integrated
  inventory problem with battery degradation and swap stations, in: Proceedings
  of the Industrial and Systems Engineering Research Conferences, Orlando, FL,
  2018, pp. 1--6.

\bibitem{Tan18}
X.~Tan, G.~Qu, B.~Sun, N.~Li, D.~H.~K. Tsang, Optimal scheduling of battery
  charging station serving electric vehicles based on battery swapping, IEEE
  Transactions on Smart Grid 10~(2) (2019) 1372--1384.

\bibitem{Sarker15}
M.~Sarker, H.~Pand\u{z}i\'{c}, M.~Ortega-Vazquez, Optimal operation and
  services scheduling for an electric vehicle battery swapping station, IEEE
  Transactions on Power Systems 30~(2) (2015) 901--910.

\bibitem{Shavarani2018}
S.~M. Shavarani, M.~G. Nejad, F.~Rismanchian, G.~Izbirak, Application of
  hierarchical facility location problem for optimization of a drone delivery
  system: a case study of {Amazon} prime air in the city of {San Francisco},
  The International Journal of Advanced Manufacturing Technology 95~(9) (2018)
  3141--3153.

\bibitem{Kim13}
J.~Kim, B.~D. Song, J.~R. Morrison, On the scheduling of systems of {UAVs} and
  fuel service stations for long-term mission fulfillment, Journal of
  Intelligent {\&} Robotic Systems 70~(1) (2013) 347--359.

\bibitem{HONG18}
I.~Hong, M.~Kuby, A.~T. Murray, A range-restricted recharging station coverage
  model for drone delivery service planning, Transportation Research Part C:
  Emerging Technologies 90 (2018) 198--212.

\bibitem{Yang2015}
J.~Yang, H.~Sun, Battery swap station location-routing problem with capacitated
  electric vehicles, Computers and Operations Research 55 (2015) 217--232.

\bibitem{Pan10}
F.~Pan, R.~Bent, A.~Berscheid, D.~Izraelevitz, Locating {PHEV} exchange
  stations in {V2G}, in: First IEEE International Conference on Smart Grid
  Communications, Gaithersburg, MD, USA, 2010, pp. 173--178.

\bibitem{GZW12}
Y.~Gao, K.~Zhao, C.~Wang, Economic dispatch containing wind power and electric
  vehicle battery swap station, in: {IEEE PES Transmission and Distribution
  Conference and Exposition}, Orlando, FL, 2012, pp. 1--7.

\bibitem{Dai2014}
Q.~Dai, T.~Cai, S.~Duan, F.~Zhao, Stochastic modeling and forecasting of load
  demand for electric bus battery-swap station, IEEE Transactions on Power
  Delivery 29~(4) (2014) 1909--1917.

\bibitem{Maillart04}
O.~Alagoz, L.~M. Maillart, A.~J. Schaefer, M.~S. Roberts, The optimal timing of
  living-donor liver transplantation, Management Science 50~(10) (2004)
  1420--1430.

\bibitem{Chhatwal10}
J.~Chhatwal, O.~Alagoz, E.~S. Burnside, Optimal breast biopsy decision-making
  based on mammographic features and demographic factors, Operations Research
  58~(6) (2010) 1577--1591.

\bibitem{Zhang12}
J.~Zhang, B.~T. Denton, H.~Balasubramanian, N.~D. Shah, B.~A. Inman,
  Optimization of prostate biopsy referral decisions, Manufacturing \& Service
  Operations Management 14~(4) (2012) 529--547.

\bibitem{Khojandi14}
A.~Khojandi, L.~M. Maillart, O.~A. Prokopyev, M.~S. Roberts, T.~Brown, W.~W.
  Barrington, Optimal implantable cardioverter defibrillator {(ICD)} generator
  replacement, INFORMS Journal on Computing 26~(3) (2014) 599--615.

\bibitem{bloch01}
S.~Bloch-Mercier, Monotone markov processes with respect to the reversed hazard
  rate ordering: an application to reliability, Journal of Applied Probability
  38~(1) (2001) 195--208.

\bibitem{JONGE18}
B.~{de Jonge}, P.~A. Scarf, A review on maintenance optimization, European
  Journal of Operational Research 285~(3) (2020) 805--824.

\bibitem{StochasticInventory}
E.~L. Porteus, Foundations of Stochastic Inventory Theory, Stanford University
  Press, Stanford, CA, 2002.

\bibitem{Clark1960}
A.~J. Clark, H.~Scarf, Optimal policies for a multi-echelon inventory problem,
  Management Science 6~(4) (1960) 475--490.

\bibitem{Clark1972}
A.~J. Clark, An informal survey of multi-echelon inventory theory, Naval
  Research Logistics 19~(4) (1972) 621--650.

\bibitem{DeCroix1998}
G.~A. DeCroix, A.~Arreola-Risa, Optimal production and inventory policy for
  multiple products under resource constraints, Management Science 44~(7)
  (1998) 950--961.

\bibitem{Scarf}
H.~Scarf, The optimality of ${(S,s)}$ policies in the dynamic inventory
  problem, in: Arrow, Karlin, Suppers (Eds.), Mathematical Methods in the
  Social Sciences, Stanford University Press, Stanford, CA, 1960, pp. 196--202.

\bibitem{vanderLaan1997}
E.~V.~D. Laan, M.~Salomon, Production planning and inventory control with
  remanufacturing and disposal, European Journal of Operational Research
  102~(2) (1997) 264--278.

\bibitem{ElHafsi2009}
M.~ElHafsi, Optimal integrated production and inventory control of an
  assemble-to-order system with multiple non-unitary demand classes, European
  Journal of Operational Research 194~(1) (2009) 127--142.

\bibitem{Maity2011}
A.~K. Maity, One machine multiple-product problem with production-inventory
  system under fuzzy inequality constraint, Applied Soft Computing 11~(2)
  (2011) 1549--1555.

\bibitem{Golari2017}
M.~Golari, N.~Fan, T.~Jin, Multistage stochastic optimization for
  production-inventory planning with intermittent renewable energy, Production
  and Operations Management 26~(3) (2017) 409--425.

\bibitem{VanHorenbeek2013}
A.~V. Horenbeek, J.~Bur\'{e}, D.~Cattrysse, L.~Pintelon, P.~Vansteenwegen,
  Joint maintenance and inventory optimization systems: A review, International
  Journal of Production Economics 143~(2) (2013) 499--508.

\bibitem{Elwany2008}
A.~H. Elwany, N.~Z. Gebraeel, Sensor-driven prognostic models for equipment
  replacement and spare parts inventory, IIE Transactions 40~(7) (2008)
  629--639.

\bibitem{Rausch2010}
M.~Rausch, H.~Liao, Joint production and spare part inventory control strategy
  driven by condition based maintenance, IEEE Transactions on Reliability
  59~(3) (2010) 507--516.

\bibitem{Nahmias1982}
S.~Nahmias, Perishable inventory theory: A review, Operations Research 30~(4)
  (1982) 680--708.

\bibitem{Toktay2000}
L.~B. Toktay, L.~M. Wein, S.~A. Zenios, Inventory management of
  remanufacturable products, Management Science 46~(11) (2000) 1412--1426.

\bibitem{Zhou2011}
S.~X. Zhou, Z.~Tao, X.~Chao, Optimal control of inventory systems with multiple
  types of remanufacturable products, Manufacturing \& Service Operations
  Management 13~(1) (2011) 20--34.

\bibitem{Govindan2015}
K.~Govindan, H.~Soleimani, D.~Kannan, Reverse logistics and closed-loop supply
  chain: A comprehensive review to explore the future, European Journal of
  Operational Research 240~(3) (2015) 603--626.

\bibitem{Ribbernick15}
H.~Ribbernick, K.~Darcovich, F.~Pincet, Battery life impact of vehicle-to-grid
  application of electric vehicles, in: 28th International Electric Vehicle
  Symposium and Exhibition, Vol.~2, Korean Society of Automotive Engineers,
  Goyang, Korea, 2015, pp. 1535--1545.

\bibitem{Plett11}
G.~L. Plett, Recursive approximate weighted total least squares estimation of
  battery cell total capacity, Journal of Power Sources 196~(4) (2011)
  2319--2331.

\bibitem{Abe12}
M.~Abe, K.~Nishimura, E.~Seki, H.~Haruna, T.~Hirasawa, S.~Ito, T.~Yoshiura,
  Lifetime prediction for heavy-duty industrial lithium-ion batteries that
  enables highly reliable system design, Hitachi Review 61~(6) (2012) 259--263.

\bibitem{Hussein15}
A.~Hussein, Capacity fade estimation in electric vehicle li-ion batteries using
  artificial neural networks, IEEE Transactions on Industry Applications 51~(3)
  (2015) 2321--2330.

\bibitem{Dubbary11}
M.~Dubarry, C.~Truchot, B.~Y. Liaw, K.~Gering, S.~Sazhin, D.~Jamison,
  C.~Michelbacher, Evaluation of commercial lithium-ion cells based on
  composite positive electrode for plug-in hybrid electric vehicle
  applications. {Part {II}}. degradation mechanism under {2C} cycle aging,
  Journal of Power Sources 196~(23) (2011) 10336--10343.

\bibitem{Xu2018}
B.~Xu, A.~Oudalov, A.~Ulbig, G.~Andersson, D.~S. Kirschen, Modeling of
  lithium-ion battery degradation for cell life assessment, IEEE Transactions
  on Smart Grid 9~(2) (2018) 1131--1140.

\bibitem{Abdollahi15}
A.~{Abdollahi}, N.~{Raghunathan}, X.~{Han}, B.~{Pattipati}, B.~{Balasingam},
  K.~R. {Pattipati}, Y.~{Bar-Shalom}, B.~{Card}, Battery health degradation and
  optimal life management, in: IEEE AUTOTESTCON, National Harbor, MD, USA,
  2015, pp. 146--151.

\bibitem{wood11}
E.~Wood, M.~Alexander, T.~H. Bradley, Investigation of battery end-of-life
  conditions for plug-in hybrid electric vehicles, Journal of Power Sources
  196~(11) (2011) 5147--5154.

\bibitem{Powell}
W.~B. Powell, Approximate Dynamic Programming: Solving the Curses of
  Dimensionality, 2nd Edition, John Wiley \& Sons, New York, NY, USA, 2011.

\bibitem{Bertmis02}
D.~Bertsimas, R.~Demir, An approximate dynamic programming approach to
  multidimensional knapsack problems, Management Science 48~(4) (2002)
  550--565.

\bibitem{Powell05}
W.~B. Powell, H.~Topaloglu, Approximate dynamic programming for large-scale
  resource allocation problems, INFORMS Tut{OR}ials in Operations Research
  (2005) 123--147.

\bibitem{Simao09}
H.~P. Simao, J.~Day, A.~P. George, T.~Gifford, J.~Nienow, W.~B. Powell, An
  approximate dynamic programming algorithm for large-scale fleet management: A
  case application, Transportation Science 43~(2) (2009) 178--197.

\bibitem{Erdelyi10}
A.~Erdelyi, H.~Topaloglu, Approximate dynamic programming for dynamic capacity
  allocation with multiple priority levels, IIE Transactions 43~(2) (2010)
  129--142.

\bibitem{Maxwell10}
M.~S. Maxwell, M.~Restrepo, S.~G. Henderson, H.~Topaloglu, Approximate dynamic
  programming for ambulance redeployment, INFORMS Journal on Computing 22~(2)
  (2010) 266--281.

\bibitem{Cimen15}
M.~{\c C}imen, C.~Kirkbride, Approximate dynamic programming algorithms for
  multidimensional inventory optimization problems, IFAC Proceedings Volumes
  46~(9) (2013) 2015--2020.

\bibitem{Meissnera18}
J.~Meissner, O.~V. Senicheva, Approximate dynamic programming for lateral
  transshipment problems in multi-location inventory systems, European Journal
  of Operational Research 265~(1) (2018) 49--64.

\bibitem{Cimen17}
M.~{\c C}imen, C.~Kirkbride, Approximate dynamic programming algorithms for
  multidimensional flexible production-inventory problems, International
  Journal of Production Research 55~(7) (2017) 2034--2050.

\bibitem{Nasrollahzadeh18}
A.~Nasrollahzadeh, A.~Khademi, M.~E. Mayorga, Real-time ambulance dispatching
  and relocation, Manufacturing and Service Operations Management 20~(3) (2018)
  467--480.

\bibitem{Papadaki07}
K.~Papadakia, W.~B. Powell, Monotonicity in multidimensional {M}arkov decision
  processes for the batch dispatch problem, Operations Research Letters 35~(2)
  (2007) 267--272.

\bibitem{Rettke16}
A.~J. Rettke, M.~J. Robbins, B.~J. Lunday, Approximate dynamic programming for
  the dispatch of military medical evacuation assets, European Journal of
  Operational Research 254~(3) (2016) 824--839.

\bibitem{Darken92}
C.~Darken, J.~Moody, Towards faster stochastic gradient search, in: Neural
  Information Processing Systems, Vol.~4, Morgan Kaufmann, San Mateo, CA, 1992,
  pp. 1009--1016.

\bibitem{George06}
A.~P. George, W.~B. Powell, Adaptive stepsizes for recursive estimation with
  applications in approximate dynamic programming, Machine Learning 65~(1)
  (2006) 167--198.

\bibitem{Grid16}
{National Grid}, Hourly electric supply charges, last accessed on April 1, 2021
  at
  \url{https://www.nationalgridus.com/niagaramohawk/business/rates/5_hour_charge.asp}
  (2016).

\bibitem{SpreadingW}
{DJI}, {DJI Spreading Wings S1000} specs, last accessed on April 1, 2021 at
  \url{https://www.dji.com/spreading-wings-s1000/spec} (2019).

\bibitem{morrow08}
K.~Morrow, D.~Karner, J.~Francfort, Plug-in hybrid electric vehicle charging
  infrastructure review, Tech. rep., U.S. Department of Energy, Idaho National
  Laboratory (November 2008).

\bibitem{Tesla17}
Tesla, Supercharger, last accessed on April 1, 2021 at
  \url{https://www.tesla.com/supercharger} (2017).

\bibitem{Thinkprogress17}
J.~Romm, Chart of the month: Driven by {Tesla}, battery prices cut in half
  since 2014, last accessed on April 1, 2021 at
  \url{https://archive.thinkprogress.org/chart-of-the-month-driven-by-tesla-battery-prices-cut-in-half-since-2014-718752a30a42/}
  (January 2017).

\bibitem{H2A}
{Nexant, Inc.}, {Air Liquide}, {Argonne National Laboratory}, {Chevron
  Technology Venture}, {Gas Technology Institute}, {National Renewable Energy
  Laboratory}, {Pacific Northwest National Laboratory}, {TIAX LLC}, {H2A}
  hydrogen delivery infrastructure analysis models and conventional pathway
  options analysis results, Online, last accessed on April 1, 2021 at
  \url{https://www.energy.gov/sites/prod/files/2014/03/f9/nexant_h2a.pdf} (May
  2008).

\bibitem{Liion17}
{Battery University}, {BU}-801b: How to define battery life, last accessed on
  April 1, 2021 at:
  \url{https://batteryuniversity.com/learn/article/how_to_define_battery_life#:~:text=The%20service%20life%20of%20a,elevated%20temperatures%20also%20induces%20stress.}
  (August 2017).

\bibitem{DEBNATH2014}
U.~K. Debnath, I.~Ahmad, D.~Habibi, Quantifying economic benefits of second
  life batteries of gridable vehicles in the smart grid, International Journal
  of Electrical Power and Energy Systems 63 (2014) 577--587.

\bibitem{Montgomery08}
D.~Montgomery, Design and analysis of experiments, 8th Edition, John Wiley \&
  Sons, New Jersey, 2008.

\end{thebibliography}

\end{document}